\documentclass[11pt]{amsart}
\usepackage{amsmath,amsfonts,amssymb,amsthm}
\usepackage{latexsym,bm,graphicx}
\usepackage{mathrsfs}
\usepackage{color}
\usepackage{hyperref}
\usepackage[all]{xy}

\title[Transformation theorems]{Almost splitting maps, transformation theorems and smooth fibration theorems}
\author{Hongzhi Huang and Xian-Tao Huang}
\address{Hongzhi Huang \\ Department of Mathematics\\ Jinan University\\ Guangzhou 510632}
\email{\href{mailto:huanghz@jnu.edu.cn}{huanghz@jnu.edu.cn}
}
\address{Xian-Tao Huang\\School of Mathematics\\ Sun Yat-sen University\\ Guangzhou 510275}
\email{\href{mailto:hxiant@mail2.sysu.edu.cn
	}{hxiant@mail2.sysu.edu.cn
}}

\newtheorem{thm}{Theorem}[section]
\newtheorem{prop}[thm]{Proposition}
\newtheorem{lem}[thm]{Lemma}

\newtheorem{cor}[thm]{Corollary}

\newtheorem{conj}[thm]{Conjecture}

\theoremstyle{definition}
\theoremstyle{remark}

\newtheorem{defn}[thm]{Definition}
\newtheorem{rem}[thm]{Remark}

\numberwithin{equation}{section}

\newcommand {\vol }{\mathrm{vol}}

\newcommand {\myd }{\mathrm{d}}

\newcommand {\supp }{\mathrm{supp}}

\newcommand {\R}{\mathbb{R}}

\newcommand {\N}{\mathbb{N}}

\newcommand {\inj}{\mathrm{inj}}

\newcommand {\Ric}{\mathrm{Ric}}

\newcommand {\myarrow}[1]{\mathop{\longrightarrow}\limits^{#1}}
\newcommand {\RCD}{\mathrm{RCD}}

\newcommand\tbbint{{-\mkern -16mu\int}}

\newcommand\dbbint{{-\mkern -19mu\int}}

\newcommand\bbint{
{\mathchoice{\dbbint}{\tbbint}{\tbbint}{\tbbint}}
}

\newcommand{\XXint}[3]{{
		\setbox0=\hbox{$#1{#2#3}{\int}$}
		\vcenter{\hbox{$#2#3$}}\kern-.5\wd0}}

\begin{document}
%\today

\maketitle
\begin{abstract} In this paper, we introduce a notion, called generalized Reifenberg condition, under which we prove a smooth fibration theorem for collapsed manifolds with Ricci curvature bounded below, which gives a unified proof of smooth fibration theorems in many previous works (including the ones proved by Fukaya and Yamaguchi respectively).
A key tool in the proof of this fibration theorem is the transformation technique for almost splitting maps, which originates from Cheeger-Naber (\cite{CN}) and Cheeger-Jiang-Naber (\cite{CJN21}).
More precisely, we show that a transformation theorem of Cheeger-Jiang-Naber (see Proposition 7.7 in \cite{CJN21}) holds for possibly collapsed manifolds.
Some other applications of the transformation theorems are given in this paper.

\vspace*{10pt}

\vspace*{10pt}

\end{abstract}
%\tableofcontents
%\setcounter{tocdepth}{1}

\section{Introduction}  %注入一级标题示例

\subsection{Smooth fibration theorems via almost splitting maps}

The almost splitting theorem, developed by Cheeger and Colding (see \cite{CC97}), is a quantitative version of Cheeger-Gromoll's splitting theorem for manifolds with non-negative Ricci curvature, and is fundamental in the study of Ricci limit space (i.e. measured Gromov-Hausdorff limits of Riemannian manifolds with Ricci curvature uniformly bounded from below).
Meanwhile, the technical tool, almost splitting maps, arisen from the proof of the almost splitting theorem, has turned out to be a powerful geometric tool in studying manifolds with Ricci curvature uniformly bounded from below (see \cite{KW}, \cite{CN}, \cite{CJN21} etc.).
Roughly speaking, an almost splitting map is a harmonic map approximating a coordinate projection in $W^{2,2}$-sense.
Therefore, without further assumptions, less can be told about infinitesimal information at a given point from an almost splitting map.
For example, an almost splitting map may be degenerated at some points even in the case of uniformly bounded Ricci curvature (see \cite{An}).
However, in the case of uniformly bounded sectional curvature, using elliptic estimate on tangent spaces, it is not hard to show that an almost splitting map is close to a coordinate map in $C^{1,\alpha}$-sense.
Specially, by gluing locally defined almost splitting maps, this provides a new proof of Fukaya's smooth fibration theorem (\cite{Fu87}) for manifolds with uniformly bounded sectional curvature (the readers can refer to Appendix B of \cite{NZ} for details).
However, in the case of uniformly sectional curvature bounded below, due to lack of a uniform $C^{1,\alpha}$-bound, it is not direct to know the non-degeneracy of an almost splitting map.
In this article, we will give an affirmative answer to the non-degeneracy of an almost splitting map in this case.
We have the following non-degeneracy theorem, which is conjectured by Xiaochun Rong.

\begin{thm}\label{NonDegeneracyofSplittingMaps}
	For any $n,\epsilon>0$, there exists $\delta>0$ depending only on $n$ and $\epsilon$, such that the following holds.
Suppose $(M,g,p)$ is an $n$-dimensional manifold with sectional curvature $\mathrm{sec}\geq -\delta$ and there exists a $(\delta,k)$-splitting map (see Definition \ref{def-harm-split} for the definition) $u : B_{3}(p)\rightarrow\mathbb{R}^{k}$ (with $1\leq k\leq n$), then for any $x\in B_{1}(p)$, $du : T_{x} M\rightarrow\mathbb{R}^{k}$ is non-degenerate.
In addition, for any $x,y\in B_{\frac12}(p)$, we have
\begin{equation}%\label{BiHolder}
	(1-\epsilon)d(x,u^{-1}(u(y)))^{1+\epsilon}\le |u(x)-u(y)|\le (1+\epsilon)d(x,u^{-1}(u(y))).
\end{equation}
\end{thm}

Theorem \ref{NonDegeneracyofSplittingMaps} makes up for blanks in using almost-splitting-maps technique to prove a smooth fibration theorem in the case of uniformly lower sectional curvature bound.

As mentioned above, almost-splitting-maps technique provides a new proof of Fukaya's smooth fibration theorem for manifolds with uniformly bounded sectional curvature.
In \cite{Hua}, the first author uses almost splitting maps to prove a smooth fibration theorem for manifolds with locally bounded Ricci covering geometry.
The key ingredient in \cite{Hua} is the canonical Reifenberg theorem of Cheeger, Jiang and Naber (see Theorem 7.10 in \cite{CJN21}), which is a type of non-degeneracy theorem of almost splitting maps.
Hence by substituting the Reifenberg theorem of \cite{CJN21} in the argument in \cite{Hua} by Theorem \ref{NonDegeneracyofSplittingMaps}, we can recover the existence part of Yamaguchi's smooth fibration theorem (see \cite{Y91}).
We point out that in Yamaguchi's fibration theorem, an almost Riemannian submersion property is proved.
But in Theorem \ref{NonDegeneracyofSplittingMaps} we fail to gain such a regularity, and instead, we derive a bi-H\"{o}lder property (\ref{BiHolder}).
This is partially due to the fact that an almost splitting map is not sensitive to a sectional curvature lower bound.
However, by employing the technique of the almost-splitting-maps, we are enable to study fibration theorems in a more extensive class of collapsed manifolds in the context of Ricci curvature bounded below, as demonstrated in Theorem \ref{FiberBundleThm} below.

Note that in \cite{An}, Anderson constructs a sequence of $n$-manifolds $(M_{i},g_i)$ ($n\ge 4$) with two sided Ricci curvature bound $|\Ric_{g_{i}}|\rightarrow0$, and $M_{i}$ collapses to a torus but every $M_{i}$ admits no fibration over tori.
Hence to guarantee a fibration theorem in Ricci case, extra conditions are needed.
There have been many previous works generalizing smooth fibration theorem to collapsing manifolds with Ricci conditions, see \cite{DWY,Wei97,NZ,HKRX,Hua,HW20,Ro22} etc.
In this paper, we introduce a notion, called {$(\Phi, r_{0}; k,\delta)$-generalized Reifenberg condition}, under which we can prove a smooth fibration theorem on manifolds with Ricci curvature bounded below, see Theorem \ref{FiberBundleThm}.
The generalized Reifenberg condition is implied by sectional curvature lower bound, or implied by the other geometric assumptions in the above mentioned paper where smooth fibration theorems hold.

Recall that, given an $n$-Riemannian manifold $M$, we say $p\in M$ is a $(\delta,r)$-Reifenberg point (where $\delta,r>0$), if for any $s\le r$, $d_{GH}(B_s(p),B_s(0^n))\le \delta s$ (here $d_{\mathrm{GH}}$ denotes the  Gromov-Hausdorff distance);
we say $M$ is uniformly $(\delta,r)$-Reifenberg, if each point of $M$ is $(\delta,r)$-Reifenberg.
Cheeger and Colding showed that, given any $\epsilon>0$, there exists a sufficiently small $\delta>0$ depending on $n,\epsilon$ so that, suppose the $n$-dimensional manifold $M$ satisfies $\Ric_M\ge-\delta$ and there is some $r\in(0,1]$ such that $d_{GH}(B_r(p),B_r(0^n))\le \delta r$, then $p$ is $(\epsilon,r)$-Reifenberg (see Theorem A.1.5 of \cite{CC97}).
Also by \cite{CC97}, we known that, to control topology of manifolds with lower Ricci curvature bound, the uniformly Reifenberg condition is a good candidate, since it provides sufficient rigidity to control local topology and it only imposes bi-h\"older regularity for the distance.
However, the Reifenberg condition is uneasy to come in to study collapsed manifolds.
The generalized Reifenberg condition introduced in this paper is aimed to provide a candidate of Reifenberg condition for possibly collapsed manifolds.

\begin{defn}\label{defn-k-euc}
Given a Riemannian manifold $M$, we say $B_{r}(p)\subset M$ is $(\delta,k)$-Euclidean, if there exists a metric space $(Z,d_{Z})$ such that
\begin{align}
d_{GH}(B_{r}(p),B_{r}((0^{k},x))\leq\delta r,
\end{align}
where $(0^{k},x)\in \mathbb{R}^{k}\times Z$.
\end{defn}

\begin{defn}\label{K-Reifenberg}
Given a Riemannian manifold $(M, g)$, $p\in M$ and $k\in \mathbb{Z}^{+}$, we define
\begin{align}
\theta^{(M,g)}_{p,k}(r)=\inf\{\epsilon>0\mid B_{r}(p)\text{ is }(\epsilon,k)\text{-Euclidean}\},
\end{align}
\begin{align}\label{1.4}
\Theta^{(M,g)}_{p,k}(r)=\sup_{s\in[0,r]} \theta^{(M,g)}_{p,k}(s).
\end{align}

We will use the notations $\theta_{p,k}(r)$, $\Theta_{p,k}(r)$ if there is no ambiguity on $(M,g)$.
\end{defn}

\begin{defn}\label{defn1.6}
Given a Riemannian manifold $(M, g)$ and $p\in M$, if there exist $\delta\geq 0$, $k\in \mathbb{Z}^{+}$, $r_{0}>0$, and a function $\Phi:\R^+\to \R^+$ with $\lim_{\delta'\to0}\Phi(\delta')=0$, so  that
\begin{equation}\label{CollapsingReifenberg}
\Theta^{(M,g)}_{p,k'}(r)\le \Phi(\max\{\delta,\theta^{(M,g)}_{p,k'}(r)\})
\end{equation}
holds for every $r\in(0,r_0]$ and every integer $k'\geq k$,
then we say $p$ satisfies the {$(\Phi, r_{0}; k,\delta)$-generalized Reifenberg condition}.

If (\ref{CollapsingReifenberg}) holds for every $p\in U\subset M$, $r\in(0,r_0]$ and $k'\geq k$, then we say $U$ satisfies the {$(\Phi,r_0; k,\delta)$-generalized Reifenberg condition}.

When $r_{0}=1$, we will say $(\Phi; k,\delta)$-generalized Reifenberg condition for simplicity.
We will say
$(k,\delta)$-generalized Reifenberg condition if there is no ambiguity on $\Phi$ and $r_0$.
\end{defn}

Roughly speaking, the $(k,\delta)$-generalized Reifenberg condition says that if there is a geodesic ball which almost splits off an $\R^{k}$ factor, then it still almost splits off (in a quantized sense) an $\R^{k}$ factor in any concentric geodesic balls with smaller radius.
As mentioned above, by Cheeger-Colding's theory, an $n$-dimensional manifold with $\Ric\geq -\delta$ always satisfies the {$(\Phi; n,\delta)$-generalized Reifenberg condition} for a positive function $\Phi(\cdot)$ depending only on $n$.
In Section \ref{sec-5}, we will study $(k,\delta)$-generalized Reifenberg condition and give more sufficient conditions for $(k,\delta)$-generalized Reifenberg condition.

Now we can state the non-degeneracy theorem under the $(k,\delta)$-generalized Reifenberg condition.

\begin{thm}\label{NonDegeneracyofSplittingMaps-RicCase}
Given $\epsilon>0$ and a function $\Phi:\R^+ \to \R^+$ with $\lim_{\delta\to0}\Phi(\delta)=0$, there exists $\delta_{0}>0$ depending on $n$, $\epsilon$ and $\Phi$ so that the following holds for every $\delta\in(0,\delta_{0})$.
Suppose $(M,g,p)$ is an $n$-dimensional manifold with $\Ric\geq -(n-1)\delta$ and there exists a $(\delta,k)$-splitting map $u : B_{3}(p)\rightarrow\mathbb{R}^{k}$ (with $1\leq k\leq n$).
Suppose in addition that $B_{1}(p)$ satisfies the $(\Phi; k,\delta)$-generalized Reifenberg condition,
then for any $x\in B_{1}(p)$, $du : T_{x} M\rightarrow\mathbb{R}^{k}$ is non-degenerate.
And for any $x,y\in B_{\frac12}(p)$, we have
\begin{equation}\label{BiHolder}
	(1-\epsilon)d(x,u^{-1}(u(y)))^{1+\epsilon}\le |u(x)-u(y)|\le (1+\epsilon)d(x,u^{-1}(u(y))).
\end{equation}
\end{thm}

According to Propositions \ref{prop-smp-1}, \ref{prop-smp-3}, Theorem \ref{NonDegeneracyofSplittingMaps-RicCase} includes Theorem 7.10 of \cite{CJN21} and Theorem \ref{NonDegeneracyofSplittingMaps} as special cases.
Note that we do not assume any non-collapsed condition here.

Utilizing Theorem \ref{NonDegeneracyofSplittingMaps-RicCase} and arguing as in \cite{Hua}, we can prove the following smooth fibration theorem.

\begin{thm}\label{FiberBundleThm}
	Given $\epsilon>0$ and a function $\Phi:\R^+\to \R^+$ with $\lim_{\delta\to0}\Phi(\delta)=0$, there exists $\delta_{0}>0$ depending on $n$, $\epsilon$ and $\Phi$ so that the following holds for every $\delta\in(0,\delta_{0})$.
Let $M$ be a compact $n$-dimensional manifold with $\Ric_M\ge-(n-1)$, and $N$ be a $k$-dimensional manifold with $|\sec_N|\le 1$ and $\inj_N\ge 1$.
Furthermore, suppose $M$ satisfies the $(\Phi;k,\delta)$-generalized Reifenberg condition and $d_{GH}(M,N)\le \delta$, then there exists a $C(n)$-Lipschitz map $f:M\to N$ which is both a smooth fibration and an $\epsilon$-Gromov-Hausdorff approximation.
\end{thm}

By combining with Propositions \ref{prop-smp-1}, \ref{prop-smp-3}, and \ref{cor-smp-3}, Theorem \ref{FiberBundleThm} recovers the (existence part of) smooth fibration theorems in \cite{Fu87,Y91,DWY,Wei97,NZ,HKRX,Hua,HW20,Ro22} and others.

Note that in the work of generalizing the collapsing fibration theorem to the context of Ricci curvature prior to Theorem \ref{FiberBundleThm}, all fiber structures are nilpotent, i.e. a fiber is an infra-nilmanifold and the structural group is affine, which can be seen as a generalization of Fukaya's nilpotent fibration.
As far as we know, the most extensive condition to guarantee the fibration to be nilpotent, under the assumption of Ricci curvature bounded below, is the locally bounded Ricci covering geometry, which is first proposed to investigate by Rong (\cite{Ro18}).
Recently, Rong (\cite{Ro22}) demonstrated the existence of a nilpotent fibration on such a class of manifolds.
In Rong's work, due to the lack of a smoothing technique for this particular class of manifolds, a crucial step involves providing a proof for a generalization of Gromov's almost flat manifolds within the context of locally bounded Ricci covering geometry, independent of Gromov's original work (\cite{Gr78}).

However, due to the fact that the generalized Reifenberg condition is implied by the sectional curvature lower bound, we cannot expect the fiber type of Theorem \ref{FiberBundleThm} to be an infra-nilmanifold.
A simple counterexample is a sequence of round spheres with diameter tending to $0$.
Hence, in this aspect, Theorem \ref{FiberBundleThm} fills the blank left by the absence of a generalization of Yamaguchi's smooth fibration theorem in the context of Ricci curvature.
The fiber properties in Theorem \ref{FiberBundleThm} are investigated in our upcoming work \cite{HH}, where we have proved that the fibers share similar topological properties as the ones in \cite{Y91}.

\subsection{Transformation theorems for collapsed manifolds.}

The proof of Theorem \ref{NonDegeneracyofSplittingMaps-RicCase} is enlightened by the proof of the canonical Reifenberg theorem in \cite{CJN21}, using a transformation technique.

Transformation theorems roughly say that, on a manifold with almost nonnegative Ricci curvature, under certain assumptions, a $(\delta, k)$-splitting harmonic map is $(\epsilon,k)$-splitting in smaller scales up to a transformation by a lower triangle matrix with positive diagonal entries.

In \cite{CN}, Cheeger and Naber prove a transformation theorem under an analytic condition (see Theorem 1.11 in \cite{CN}), which is a key in their proof of the codimension 4 conjecture.
Later on, in \cite{CJN21}, Cheeger, Jiang and Naber prove a geometric transformation theorem (Theorem 7.2 in \cite{CJN21}).
Recently, in \cite{BNS22}, Bru\`{e}, Naber and Semola give a transformation theorem on non-collapsed $\textmd{RCD}(K,N)$ spaces.
Recall that $\textmd{RCD}(K,N)$ spaces are metric measure spaces with generalized Ricci curvature bounded from below by $K$ and dimension bounded from above by $N$, these spaces include all $n$-manifolds with $\Ric \geq K, n\leq N$ and their Ricci limit spaces, see \cite{AGMR15,EKS15,Gig13} etc. for the $\mathrm{RCD}$ theory, and see \cite{DePGil18} etc. for non-collapsed $\mathrm{RCD}$ theory.

Recently, the idea of transformation is used in other works and has many interesting applications, see  \cite{HP22} \cite{WZh21} etc.

Our proof of Theorem \ref{NonDegeneracyofSplittingMaps-RicCase} follows the ideas of the canonical Reifenberg theorem in \cite{CJN21}, whose proof uses a transformation theorem, i.e. Proposition 7.7 in \cite{CJN21}.
Note that in the statement of Proposition 7.7 of \cite{CJN21}, they require the manifolds to be non-collapsed.
To prove Theorem \ref{NonDegeneracyofSplittingMaps-RicCase}, we need a version of transformation theorem which holds for possibly collapsed manifolds, see the following theorem:

\begin{thm}\label{thm-splitting-stable}
For any $N \geq 1$, $\epsilon>0$ and $\eta>0$, there exists $\delta_{0}=\delta_0(N, \epsilon, \eta)>0$ such that, for every $\delta\in(0,\delta_{0})$, the following holds.
Suppose $(X, d, m)$ is an $\mathrm{RCD}(-(N-1)\delta,N)$ space, and there is some $s\in(0,1)$ such that, for any $r\in[s,1]$, $B_{r}(p)$ is $(\delta, k)$-Euclidean but not $(\eta, k+1)$-Euclidean.
Let $u : (B_{1}(p),p)\rightarrow(\mathbb{R}^{k},0^{k})$ be a $(\delta, k)$-splitting map,
then for each $r\in[s,1]$, there exists a $k\times k$ lower triangle matrix $T_{r}$ with positive diagonal entries so that $T_{r}u : B_{r}(p)\rightarrow\mathbb{R}^{k}$ is a $(\epsilon,k)$-splitting map, and
for any $t, r \in [s,1]$,
\begin{align}
| T_{r}\circ T_{t}^{-1}|\leq  (1+C\epsilon) \max\{\bigl(\frac{t}{r}\bigr)^{C\epsilon}, \bigl(\frac{r}{t}\bigr)^{C\epsilon}\}
\end{align}
holds for a constant $C=C(N) > 1$, where $|\cdot|$ means $L^{\infty}$-norm of a matrix.
\end{thm}
We remark that Proposition 7.7 of \cite{CJN21} is just a part of the geometric transformation theorem of Cheeger-Jiang-Naber, i.e. Theorem 7.2 in \cite{CJN21}.
Under the non-collapsing condition, Theorem 7.2 in \cite{CJN21} contains an estimate of the Hessian of $T_{r}u$ (see (7.1) in \cite{CJN21}), which is more deeper.
It is an interesting question that whether there is some Hessian estimate for $T_{r}u$ on possibly collapsed manifolds.

The proof of Theorem \ref{thm-splitting-stable} follows the idea of Cheeger, Jiang and Naber's proof closely: a contradicting sequence will converge to a harmonic function with almost linear growth on the limit space,
and then we need to verify that by suitable choice of parameters in the theorem, this almost linear harmonic function must be linear (we call it the gap property in the following), and then we will obtain a contradiction.
The only difference between the two proofs is the gap property.
More precisely, in the proof of \cite{CJN21}, under the non-collapsing assumption, the above limit spaces are always metric cones, and the harmonic functions on them behave well, basing on which the gap property (i.e. Lemma 7.8 in \cite{CJN21}) is proved.
Instead, to obtain the gap property on possibly collapsed non-compact manifolds (or $\mathrm{RCD}$ spaces), our starting point is a theorem on the characterization of linear growth harmonic functions on manifolds with nonnegative Ricci curvature, which is proved by Cheeger, Colding and Minicozzi in \cite{CCM}.
Firstly we generalize Cheeger-Colding-Minicozzi's theorem to $\RCD(0,N)$ spaces in Proposition \ref{thm-split-infin}, and then we use a method to utilize Proposition \ref{thm-split-infin} to study the almost linear growth harmonic functions.
See Section \ref{sec-3} for details.

Combining Theorem \ref{thm-splitting-stable} with the $(k,\delta)$-generalized Reifenberg condition (\ref{CollapsingReifenberg}), we can prove the following theorem.

\begin{thm}\label{TransformationThmUderSplittingMonotonicity}
Given $\epsilon>0$, $n, k\in \mathbb{Z}^{+}$ with $k\leq n$, and a positive function $\Phi(\cdot)$ with $\lim_{\delta\to0^+}\Phi(\delta)=0$, there exists $\delta_{0}>0$ depending on $n$, $\epsilon$ and $\Phi$ so that the following holds for every $\delta\in(0,\delta_{0})$.
Suppose $(M,g)$ is an $n$-dimensional Riemannian manifold with $\Ric\geq -(n-1)\delta$ and $p\in M$ satisfies the $(\Phi; k,\delta)$-generalized Reifenberg condition.
Let $u : (B_{2}(p),p)\rightarrow(\mathbb{R}^{k},0^{k})$ (with $1\leq k\leq n$) be a $(\delta,k)$-splitting map.
Then for any $r\in(0,1]$, there exists a $k\times k$ lower triangle matrix $T_{r}$ with positive diagonal entries such that $T_{r}u : B_{r}(p)\rightarrow\mathbb{R}^{k}$ is an $(\epsilon,k)$-splitting map, and
for any $t, r \in (0,1]$,
\begin{align}
| T_{r}\circ T_{t}^{-1}|\leq  (1+C\epsilon) \max\{\bigl(\frac{t}{r}\bigr)^{C\epsilon}, \bigl(\frac{r}{t}\bigr)^{C\epsilon}\}
\end{align}
holds for a constant $C=C(n) > 1$.
\end{thm}

\begin{rem}
A necessary condition of Theorem \ref{TransformationThmUderSplittingMonotonicity} is that,
\begin{equation}\label{k-Reifenberg}
	\text{for any } r\in(0,1],\, B_r(p) \text{ is } (\tilde{\delta},k)\text{-Euclidean for some sufficiently small }\tilde{\delta}>0.
	\end{equation}
However, as showed in the following example, if the $(\Phi; k,\delta)$-generalized Reifenberg condition is replaced by condition (\ref{k-Reifenberg}), then we cannot conclude that any $(\delta,k)$-splitting map $u:(B_{2}(p),p)\rightarrow(\mathbb{R}^{k},0^{k})$ satisfies the conclusion of Theorem \ref{TransformationThmUderSplittingMonotonicity}.

By Theorem 1.4 in \cite{CN13}, there exists a Riemannian manifold $(M,g,p)$ with non-negative Ricci curvature, maximal volume growth and $M$ does not split off any $\R$-factor, while there exists a sequence of $R_i\to\infty$ such that $(M_{i},g_{i},p_{i})=(M,R_i^{-2}g,p)$ converges to a metric cone which splits off an $\mathbb{R}$-factor.
Up to a scaling down, we may assume $B_1(p)$ is not $(\eta,1)$-Euclidean for some $\eta>0$.
We may also assume there exist $\delta_i$-splitting functions $u_i:(B_1(p_{i}),p_{i})\rightarrow(\R,0)$ for some $\delta_i\to0$.
Putting $(\tilde{M}_i,\tilde{g}_{i},\tilde{p}_i):=(M_{i}\times\R,g_{i}+dt^{2},(p_{i},0))$ and $v_i(x,t):=u_i(x)$ for $(x,t)\in \tilde{M}_i$,
then $\Ric_{\tilde{M}_i}\ge0$, $B_r(\tilde{p}_i)$ is $(0,1)$-Euclidean for any $r\in(0,1]$, and there exist $\delta_i$-splitting functions $v_i:B_1(\tilde{p}_i)\rightarrow\R$.
We will show that, for $i$ sufficiently large, such $(\tilde{M}_i,\tilde{g}_i,v_i)$ does not satisfy the conclusion of Theorem \ref{TransformationThmUderSplittingMonotonicity}.
Otherwise, take $r_i=2R_i^{-1}$, then for large $i$, there exists a positive number $T_{r_{i}}$ such that $T_{r_i}v_i:B_{r_i}(p_i)\rightarrow\R$ is a $\Psi_{1}(\delta_{i})$-splitting function.
This implies $T_{r_i}(R_{i}u_i):B_2(p)\to\R$ is a $\Psi_{1}(\delta_{i})$-splitting function, and thus $B_1(p)$ is $(\Psi_{2}(\delta_{i}),1)$-Euclidean.
Thus we obtain a contradiction when $i$ is sufficiently large.
	
However, it is unclear whether the condition (\ref{k-Reifenberg}) guarantees that \textbf{there exists} a $(\delta,k)$-splitting map $u:B_{\frac{1}{2}}(p)\to\R^k$ satisfying the conclusion of Theorem \ref{TransformationThmUderSplittingMonotonicity}.
\end{rem}

Theorem \ref{TransformationThmUderSplittingMonotonicity} is proved by an induction and contradiction argument basing on Theorem \ref{thm-splitting-stable}, see Section \ref{sec-4.2} for details.
Basing on Theorem \ref{TransformationThmUderSplittingMonotonicity}, one can finish the proof of Theorem \ref{NonDegeneracyofSplittingMaps-RicCase}, see Section \ref{sec-4.3}.

\subsection{Other applications of the Transformation theorems}

The following theorem can be viewed as a blowdown version of Theorem \ref{thm-splitting-stable}.

\begin{thm}\label{BlowupTransformation}
Given $\eta>0$ and $N \geq 1$, there exists $\epsilon(N,\eta)>0$ such that, for any $\epsilon\in(0,\epsilon(N,\eta))$, there exists $\delta(N,\eta,\epsilon)>0$ such that, every $\delta\in(0, \delta(N,\eta,\epsilon))$ satisfies the following.
Let $(X,p,d,m)$ be an $\mathrm{RCD}(0,N)$ space.
Suppose there exist $k\in \mathbb{Z}^{+}$ with $1\leq k\leq N$ and $R_{0}\geq1$ so that
\begin{align}\label{1.10}
B_R(p)\text{ is }(\delta,k)\text{-Euclidean and not }(\eta,k+1)\text{-Euclidean for every }R\ge R_{0}.
\end{align}
Putting
\begin{align}
\mathcal{H}_{1+\epsilon}(X,p)=\{u\in W^{1,2}_{\mathrm{loc}}(X)|\Delta u=0, u(p)=0, |u(x)|\le Cd(x,p)^{1+\epsilon}+C \text{ for some }C>0\},\nonumber
\end{align}
 then
\begin{itemize}
\item [(1)] there exist $u^1,\ldots,u^k\in\mathcal{H}_{1+\epsilon}(X,p)$ such that for any $R\ge R_{0}$, there exists a lower diagonal matrix $T_R$ with positive diagonal entries such that $T_R\circ(u^1,\ldots,u^k):B_R(p)\rightarrow\R^k$ is $(\epsilon,k)$-splitting;
\item [(2)] for every $R\ge R_{0}$, $T_R$ satisfies
    \begin{align}\label{tr-growth-1}
     |T_{R}| \le R^{\epsilon} \quad \text{and}\quad     |T_{R}^{-1}|\le R^{\epsilon};
    \end{align}
\item [(3)] $\mathcal{H}_{1+\epsilon}(X,p)$ has dimension $k$.
\end{itemize}

\end{thm}

Theorem \ref{BlowupTransformation} is proved in Section \ref{sec-7}.
We remark that in Section \ref{sec-3}, we prove a Liouville-type result, i.e. Proposition \ref{thm-liou-har-RCD}, which is related to Theorem \ref{BlowupTransformation}.

A sufficient condition of (\ref{1.10}) on a manifold $M$ with nonnegative Ricci curvature is that it has almost maximal volume growth.
In this case, Theorem \ref{BlowupTransformation} provides a harmonic map $u:M \to \R^n$, whose components form a basis of $\mathcal{H}_{1+\epsilon}(M,p)$.
With a bit extra work, we can observe that $u$ is a diffeomorphism. %see the following proposition.

\begin{prop}\label{thm-can-diff}
	There exists $\delta(n)>0$, such that the following holds for every $\delta\in(0,\delta(n))$.
Suppose $(M,g)$ is an $n$-dimensional manifold with nonnegative Ricci curvature, and
\begin{align}
\lim_{R\rightarrow+\infty}\frac{\mathrm{vol}(B_R(p))}{\mathrm{vol}(B_R(0^n))}\ge 1-\delta.
\end{align}
Then there exists a proper harmonic map $u:M \to \R^n$ with at most $(1+\Psi(\delta|n))$-growth so that $u$ is a diffeomorphism.
\end{prop}

\begin{rem}
For an $n$-manifold $(M,g)$ with nonnegative Ricci curvature and almost maximal volume growth, it is proved by Cheeger and Colding (see \cite{CC97}) that $M$ is diffeomorphic to $\mathbb{R}^{n}$, and the diffeomorphism is constructed by the Reifenberg method.
It is an interesting question that whether there are diffeomorphisms constructed from other geometric tools or geometric analysis tools.
Cheeger-Jiang-Naber's canonical Reifenberg theorem (Theorem 7.10 in \cite{CJN21}) gives harmonic diffeomorphisms from $B_{R}(p)$ onto its image for every $R>0$.
Proposition \ref{thm-can-diff} is a global version of Theorem 7.10 in \cite{CJN21}.
We remark that by generalizing Perelman's pseudo-locality theorem (see \cite{Per02}), Wang is able to construct a diffeomorphism to $\mathbb{R}^{n}$ by the immortal Ricci flow solution initiated from $(M,g)$, see \cite{Wang20}.

We also remark that, a similar result still holds on non-collapsed $\RCD$ spaces (in the sense of \cite{DePGil18}).
That is, let $(X,p,d,\mathcal{H}^n)$ be a noncompact $\RCD(0,n)$ space such that $\mathcal{H}^n(B_R(p))\ge(1-\delta)\mathcal{H}^n(B_R(0^n))$ holds for any $R>0$, then there exists a homeomorphism $u:M\to\R^n$ which is given by a harmonic map with at most $(1+\Psi(\delta|n))$-growth.
Note that by Theorem 1.3 of \cite{KM20}, such $X$ is known to homeomorphic to $\R^n$.
\end{rem}

We also give some applications of Theorem \ref{BlowupTransformation} in the nonnegative-sectional-curvature case as follows.

\begin{prop}\label{NonnegativeSecSplitting}
There exists $\delta(n)>0$ such that the following holds.
Suppose $M$ is a complete $n$-manifold with non-negative sectional curvature.
Let $C_{\infty,1}M$ be the unit ball centered at the cone point in the tangent cone at infinity of $M$.
Suppose $C_{\infty,1}M$ is $(\delta(n),k)$-Euclidean, then there exists a harmonic map $u:M\to \R^k$ which is non-degenerate.
Furthermore, if we assume $d_{GH}(C_{\infty,1}M,B_1(0^k))\le\delta(n)$, where $B_{1}(0^k)$ is the unit ball in $\R^{k}$, then there exists a proper harmonic map $u:M\to\R^k$ which is a trivial fiber bundle with compact fibers.
\end{prop}

We don't know whether Proposition \ref{NonnegativeSecSplitting} holds for non-negatively curved Alexandrov spaces.

\begin{prop}\label{prop-1.14}
Given $\epsilon> 0$ and $k, n\in \mathbb{Z}^+$ with $k\leq n$, there exists $\delta(n,\epsilon) > 0$ such that the following holds for every $\delta\in(0,\delta(n,\epsilon)]$.
Suppose $(M,g)$ is a complete $n$-dimensional Riemannian manifold with non-negative sectional curvature and $h_{1+\delta}(M,p)=k$ (where $p\in M$), then there exists some $R_0\geq1$ such that $B_{R}(p)$ is $(\epsilon,k)$-Euclidean for every $R\geq R_0$.
\end{prop}

Proposition \ref{prop-1.14} still holds for non-negatively curved Alexandrov spaces, and we make the following conjecture, which can be viewed as a quantified version of Proposition \ref{thm-split-infin}.

\begin{conj}\label{conj7.2}
Given $\epsilon> 0$ and $N\geq 1$, there exists $\delta(N,\epsilon) > 0$ such that the following holds for every $\delta\in(0,\delta(N,\epsilon)]$.
Suppose $(X,d,m)$ is an $\mathrm{RCD}(0,N)$ space with $h_{1+\delta}(X,p)=k$ (where $p\in X$, $k\leq N$), then there exists some $R_0\geq1$ such that $B_{R}(p)$ is $(\epsilon,k)$-Euclidean for every $R\geq R_0$.
\end{conj}

\vspace*{10pt}

We point out that transformation theorems \ref{thm-splitting-stable} and \ref{TransformationThmUderSplittingMonotonicity} still holds for some functions which are not necessary harmonic, see \cite{Hua22} (see also \cite{BNS22} \cite{HP22} \cite{WZh21}).
As applications, in \cite{Hua22} some finite topological type theorems on certain non-compact manifolds with nonnegative Ricci curvature are proved.
One can expect more applications of the transformation technique in the study of manifolds with Ricci curvature lower bound in the future.

\subsection{Some other remarks}

At the end of this introduction, we give some additional remarks on the notions and results of this paper.

\begin{rem}\label{rem1.15}
Definitions \ref{K-Reifenberg} and \ref{defn1.6} can be applied to metric spaces $(X,d)$, including Ricci limit space, $n$-dimensional Alexandrov space, and more generally on $\mathrm{RCD}(K,N)$.
Suppose $(X,d,m)$ is an $\mathrm{RCD}(-(n-1),n)$ spaces (where $n\in\mathbb{Z}^{+}$), by \cite{MN19}, there exists $0<\theta(n)<<1$ such that, for any $p\in X$, $N\geq n+1$ and $0<r<1$, we have $\theta^{(X,d)}_{p,N}(r)\geq \theta(n)$.
Hence in this case, in the definition of $(\Phi, r_{0}; k,\delta)$-generalized Reifenberg condition, we only need to consider $k\leq k'\leq n$ in (\ref{CollapsingReifenberg}).

We also remark that some results of this paper, involving generalized Reifenberg condition, still hold on some metric (measure) spaces.
For example, by \cite{LN20}, we know that $n$-dimensional Alexandrov spaces with curvature bounded from below by $-\delta$ satisfy the $(\Phi; k,\delta)$-generalized Reifenberg condition.
One can check that the proof of Theorem \ref{TransformationThmUderSplittingMonotonicity} does not use the smooth structure, thus we have the following result:
\begin{prop}
For any $N \in \mathbb{Z}^{+}$, and $\epsilon>0$, there exists $\delta_{0}=\delta(N, \epsilon)$ such that, for any $\delta\in(0,\delta_{0})$ the following holds.
Suppose $(X,d,m)$ is an $\RCD(-(N-1),N)$ space and $(X,d)\in \mathrm{Alex}^{n}(-\delta)$ (where $n\leq N$).
Suppose $u : (B_{2}(p),p)\rightarrow (\mathbb{R}^{k},0^{k})$ is a $(\delta,k)$-splitting map (with $k\leq n$), then for any $r\in(0,1]$, there exists a $k\times k$ lower triangle matrix $T_{r}$ with positive diagonal entries so that
$T_{r}u : B_{r}(p)\rightarrow\mathbb{R}^{k}$ is a $(\epsilon,k)$-splitting map.
\end{prop}
\end{rem}

In the following we give the outline of this paper.
Section \ref{sec-2} contains some notions and background results which are used in this paper.
In Section \ref{sec-3}, we mainly prove the gap theorem \ref{thm-gap-har-RCD} for harmonic functions with almost linear growth. Some Liouville-type theorems are also obtained in this section.
In Section \ref{sec-4}, we prove the transformation theorems \ref{thm-splitting-stable} and \ref{TransformationThmUderSplittingMonotonicity} and the non-degeneracy theorem \ref{NonDegeneracyofSplittingMaps-RicCase}.
In Section \ref{sec-5}, we give some fundamental properties of the generalized Reifenberg condition and give sufficient conditions to guarantee the generalized Reifenberg condition.
In Section \ref{sec-6}, we prove the smooth fibration theorem \ref{FiberBundleThm} and a local version of fibration theorem \ref{localFiberBundleThm}.
In Section \ref{sec-7}, we prove Theorem \ref{BlowupTransformation}, Propositions \ref{thm-can-diff}, \ref{NonnegativeSecSplitting} and \ref{prop-1.14}.
Section \ref{sec-8} is an appendix, where we prove a generalized covering lemma, which is used in Section \ref{sec-5}.

\section{Preliminary} \label{sec-2} %注入一级标题示例

Some of the results in this paper are stated in the setting of $\RCD$ spaces.
$\textmd{RCD}(K,N)$ spaces are metric measure spaces $(X,d,m)$ with generalized Ricci curvature bounded from below by $K$ and dimension bounded from above by $N$, see \cite{AGMR15,EKS15,Gig13} etc.
$\textmd{RCD}(K,N)$ spaces include all $n$-manifolds with $\Ric \geq K, n\leq N$ and their Ricci limit spaces.
Recently, the theory of $\textmd{RCD}$ develops fast, many classical results on manifolds with lower Ricci curvature bound and theorems of Ricci limit spaces have been generalized to $\RCD$ theory.
If $N\in \mathbb{Z}^{+}$, and $m=\mathcal{H}^{N}$, then we say $(X,d,m)$ is a non-collapsed $\textmd{RCD}(K,N)$ space, see \cite{DePGil18} etc.

For general results on (measured) Gromov-Hausdorff convergence, the readers can refer to \cite{Gr81,CC97,GMS15} etc.
We use $d_{\mathrm{GH}}$ and $d_{\mathrm{mGH}}$ to denote the Gromov-Hausdorff distance and measured Gromov-Hausdorff distance respectively.

For results on calculus in general metric measure spaces, the readers can refer to \cite{AGS14-1,C99,Gig15} etc.

In the following, we recall some notions and properties on convergence of functions defined on varying metric measure spaces.

\begin{defn}
Suppose $(Z_{i}, z_{i}, d_{i})$ pointed Gromov-Hausdorff converges to $(Z_{\infty}, z_{\infty}, d_{\infty})$ with a sequence of $\epsilon_{i}$-Gromov-Hausdorff approximations $\Phi_{i}: Z_{i}\rightarrow Z_{\infty}$, where $\epsilon_{i}\rightarrow0$.
Suppose $f_{i}$ is a function on $X_{i}$ and $f_{\infty}$ is a function on $Z_{\infty}$.
Given $K\subset Z_{\infty}$.
If for every $\epsilon>0$, there exists $\delta>0$ such that $|f_{i}(x_{i})-f_{\infty}(x_{\infty})|<\epsilon$ holds for every $i\geq\delta^{-1}$, $x_{i}\in Z_{i}$, $x_{\infty}\in K$ with $d_{\infty}(\Phi_{i}(x_{i}),x_{\infty})< \delta$, then we say $f_{i}$ converges to $f_{\infty}$ uniformly on $K$.
\end{defn}

The following theorem is a generalization of the classical Arzela-Ascoli Theorem, see Proposition 27.20 in \cite{Vi09} or Proposition 2.12 in \cite{MN19}.
\begin{thm}\label{AA}
Suppose $(X_{i}, p_{i}, d_{i})\xrightarrow{pGH}(X_{\infty}, p_{\infty}, d_{\infty})$, $R\in(0,\infty]$.
Suppose for every $i$, $f_{i}$ is a Lipschitz function defined on $B_{R}(p_{i})\subset X_{i}$ and $\mathrm{Lip}f_{i}\leq L$ on $B_{R}(p_{i})$, $| f_{i}(p_{i})|\leq C$  for some uniform constants $L$ and $C$.
Then there exists a subsequence of $f_{i}$, still denoted by $f_{i}$, and a Lipschitz function $f_{\infty}: B_{R}(p_{\infty})\rightarrow\mathbb{R}$, such that $f_{i}$ converges uniformly to $f_{\infty}$ on $\overline{B_{r}(p_{\infty})}$ for every $r<R$.
\end{thm}

\begin{defn}
Suppose $(Z_{i}, z_{i}, d_{i}, m_{i})\xrightarrow{pmGH}(Z_{\infty}, z_{\infty}, d_{\infty}, m_{\infty})$.
Given any $p\in(1,\infty)$.
Suppose that $f_{i}$ is a Borel function on $Z_{i}$, we say $f_{i}\rightarrow f :Z_{\infty}\rightarrow\mathbb{R}$ in the weak sense if for any uniformly converging sequence of compactly supported Lipschitz functions $\varphi_{i}\rightarrow \varphi$, we have
\begin{align}
\lim_{i\rightarrow\infty}\int f_{i}\varphi_{i} dm_{i}=\int f\varphi dm_{\infty}.
\end{align}
We say $f_{i}\rightarrow f$ in the weak $L^{p}$-sense if in addition $f_{i}, f$ have uniformly bounded $L^{p}$ integrals.
We say $f_{i}\rightarrow f$ in the $L^{p}$-sense if $f_{i}\rightarrow f$ in the weak $L^{p}$-sense and
\begin{align}
\lim_{i\rightarrow\infty}\int |f_{i}|^{p} dm_{i}=\int |f|^{p} dm_{\infty}.
\end{align}
Suppose $f_{i}\in W^{1,p}(Z_{i})$, $f\in W^{1,p}(Z)$, then we say $f_{i}\rightarrow f$ in the $W^{1,p}$-sense if $f_{i}\rightarrow f$ in the $L^{p}$-sense and
\begin{align}
\lim_{i\rightarrow\infty}\int |\nabla f_{i}|^{p} dm_{i}=\int |\nabla f|^{p} dm_{\infty}.
\end{align}
For $p = 1$, given $f_{i} \in L^{1}(Z_{i})$, $f\in L^{1}(Z_{\infty})$, we say $f_{i}\rightarrow f$ in the $L^{1}$-sense if $\sigma\circ f_{i}\rightarrow \sigma\circ f$ in the $L^{2}$-sense, where $\sigma(z) = \mathrm{sign}(z)\sqrt{|z|}$ is the signed square root.
\end{defn}

It is not hard to check that uniform convergence implies $L^{p}$ convergence for any $1\leq p < \infty$.

The readers can refer to \cite{AH17} for basic properties of $L^{p}$-convergence (especially Section 3 in \cite{AH17}).
Especially, in the proof of the transformation theorem \ref{thm-splitting-stable}, we need the following basic property, whose proof is based on the properties in Section 3 of \cite{AH17}.
\begin{thm}\label{2.7777777iruiquo}
Suppose $(X_{i}, p_{i}, d_{i}, m_{i})\xrightarrow{pmGH}(X_{\infty}, p_{\infty}, d_{\infty}, m_{\infty})$, and $f_{i}:B_{R}(p_{i})\rightarrow \mathbb{R}$ (with $f_{i}\in L^{2}(B_{R}(p_{i}))$) converges to $f_{\infty}:B_{R}(p_{\infty})\rightarrow \mathbb{R}$ in the $L^{2}$-sense, then for any constant $A\geq 0$,
$|f_{i}^{2}-A|$ converges to $|f_{\infty}^{2}-A|$ in the $L^{1}$-sense on $B_{R}(p_{\infty})$.
\end{thm}

The following theorem is also used in the proof of the transformation theorem \ref{thm-splitting-stable}.

\begin{thm}[see Theorem 4.4 in \cite{AH18}]\label{2.7777777}
Suppose the $\mathrm{RCD}(K,N)$ spaces $(X_{i}, p_{i}, d_{i}, m_{i})$ converges in the pointed measured Gromov-Hausdorff distance to $(X_{\infty}, p_{\infty}, d_{\infty}, m_{\infty})$.
Suppose $f_{i}$ is defined on $B_{R}(p_{i})$ such that $f_{i}\in W^{1,2}(B_{R}(p_{i}))\cap D_{B_{R}(p_{i})}(\Delta)$
and
\begin{align}
\bbint_{B_{R}(p_{i})}|f_{i}|^{2} +\bbint_{B_{R}(p_{i})}|\nabla f_{i}|^{2} +\bbint_{B_{R}(p_{i})}|\Delta f_{i}|^{2} \leq C
\end{align}
for some uniform constant $C$.
If $f_{i}$ converges in the $L^{2}$-sense to some $W^{1,2}$-function $f: B_{R}(p_{\infty})\rightarrow \mathbb{R}$, then
\begin{description}
  \item[(1)] $f_{i}\rightarrow f$ in the $W^{1,2}$-sense over $B_{R}(p_{\infty})$;
  \item[(2)] $\Delta f_{i}\rightarrow \Delta f$ in the weak $L^{2}$-sense;
  \item[(3)] $|\nabla f_{i}|^{2}$ converges to $|\nabla f|^{2}$ in the $L^{1}$-sense on any $B_{r}(p)$ with $r<R$.
\end{description}
\end{thm}

The following theorem allows us to transplant harmonic functions on the limit space back to the convergence sequence of spaces.
\begin{thm}[see \cite{AH17,ZZ17}; see \cite{Ding04} for the Ricci-limit case]\label{thm-transplant}
Suppose $\RCD(-(N-1),N)$ spaces $(X_{i}, p_{i}, d_{i}, m_{i})\xrightarrow{pmGH}  ({X}_{\infty}, {p}_{\infty}, {d}_{\infty}, {m}_{\infty})$.
If $f , u \in L^{2}({X}_{\infty})$ have compact support, and assume $\Delta u = f$ and f is Lipschitz.
Then for any $R > 0$, there exist solutions $\Delta u_{i} = f_{i}$ on $B_{R}(p_{i})$ such that $u_{i}$ and $f_{i}$ converge respectively to $u$ and $f$ locally uniformly in $B_{R}(p_{\infty})$.
\end{thm}

The notion of $\delta$-splitting maps on manifolds or Ricci-limit spaces, which originates from Cheeger and Colding's works in \cite{CC96,CC97} etc., was introduced in \cite{CN}.
The notion of $\delta$-splitting maps on $\RCD$ spaces was first given in \cite{BPS19}.

\begin{defn}\label{def-harm-split}
Let $(X,d,m)$ be an $\mathrm{RCD}(-(N-1),N)$ space, $p\in X$ and $\delta > 0$ be fixed.
We say that $u:=(u_{1},\ldots,u_{k}) : B_{r}(p)\to \mathbb{R}^{k}$ is a $(\delta,k)$-splitting map if $\Delta u=0$ and
\begin{description}
  \item[(i)] $|\nabla u_{a}| \leq  1+\delta$;
  \item[(ii)] $r^{2}\bbint_{B_{r}(p)}|\mathrm{Hess}u_{a}|^{2}dm <\delta$ ;
  \item[(iii)] $\bbint_{B_{r}(p)}|\langle\nabla u_{a},\nabla u_{b}\rangle-\delta_{ab}|dm <\delta$ for any $a,b\in\{1,\ldots,k\}$.
\end{description}
We use $u : (B_{r}(p),p)\to (\mathbb{R}^{k},0^{k})$ to denote a map $u : B_{r}(p)\to \mathbb{R}^{k}$ with $u(p)=0^{k}\in \R^{k}$.
\end{defn}

We note that similar to the manifolds case, if $(X,d,m)$ is an $\mathrm{RCD}(-(N-1)\delta,N)$ space, and $u: B_{2r}(p)\rightarrow \mathbb{R}^{k}$ is harmonic, and conditions (i) and (iii) in Definition \ref{def-harm-split} hold, then by the Bochner inequality and the existence of  good cut-off functions, condition (ii) holds automatically.

We recall the following functional version of the splitting theorem, which originates from Gigli's proof of splitting theorem on $\RCD(0,N)$ spaces (\cite{Gig13}).

\begin{thm}[see \cite{Gig13,Han18-2} etc.]\label{thm-split-RCD}
Let $(X,d,m)$ be an $\mathrm{RCD}(0,N)$ space and suppose there exist functions $u^{i} : X\rightarrow\mathbb{R}$ $(1\leq i\leq k)$ such that $\Delta u^{i}=0$ and $\langle\nabla u^{i},\nabla u^{j} \rangle\equiv \delta^{ij}$.
Then $(X,d,m)$ is isomorphic to $(\mathbb{R}^{k}\times Y, d_{\mathrm{Eucl}}\times d _{Y}, \mathcal{L}^{k}\otimes m_{Y})$, where $(Y,d_{Y}, m_{Y})$ is an $\mathrm{RCD}(0, N-k)$ space or $Y=\{\mathrm{Pt}\}$.
\end{thm}

There is a local version of the functional splitting theorem, see Theorem 3.4 in \cite{BNS22}.
By a compactness argument, the local version of functional splitting theorem gives the following theorem, see Theorem 3.8 in \cite{BNS22}.

\begin{thm}\label{thm-split-GHisom}
Let $1\leq N<\infty$ be fixed.
For every positive number $\delta \ll 1$, there exists $\epsilon(N,\delta)>0$ such that any $\epsilon\in(0,\epsilon(N,\delta))$ satisfies the followings.
If $(X,d,m)$ is an $\mathrm{RCD}(-\epsilon(N-1),N)$ space, $p\in X$, then
\begin{description}
  \item[(1)] if $d_{\mathrm{mGH}}(B_{4}(p),B_{4}^{\mathbb{R}^{k}\times Z}(0^{k},z))\leq \epsilon$ for some integer $k$ and some pointed metric measure space $(Z,z,d_{Z},m_{Z})$, then there exists a $(\delta,k)$-splitting map $u=(u^{1},\ldots,u^{k}) : B_{1}(p)\rightarrow\mathbb{R}^{k}$;
  \item[(2)] if there exists an $(\epsilon,k)$-splitting map $u : B_{4}(p)\rightarrow\mathbb{R}^{k}$ for some integer $k$, then $d_{\mathrm{mGH}}(B_{1}(p),B_{1}^{\mathbb{R}^{k}\times Z}(0^{k},z))\leq \delta$ for some pointed metric measure space $(Z,z,d_{Z},m_{Z})$; moreover, there exists $f:B_{1}(p)\rightarrow Z$ such that $(u-u(p),f):B_{1}(p)\rightarrow B_{1}^{\mathbb{R}^{k}\times Z}(0^{k},z)$ is a $\delta$-Gromov-Hausdorff isometry.
\end{description}
\end{thm}

\begin{rem}
We remark that, in (1) of Theorem \ref{thm-split-GHisom}, if the assumption $$d_{mGH}(B_{4}(p),B_{4}^{\mathbb{R}^{k}\times Z}(0^{k},z))\leq \epsilon$$ is replaced by $d_{GH}(B_{4}(p),B_{4}^{\mathbb{R}^{k}\times Z}(0^{k},z))\leq \epsilon$, then the existence of $(\delta,k)$-splitting map $u: B_{1}(p)\rightarrow\mathbb{R}^{k}$ is no longer true, which can be seen from the following example:

According to Example 1.24 of \cite{CC97}, there exists a sequence of Riemannian manifolds $(M_{i},g_{i})$ such that each $M_{i}$ is diffeomorphic to $\R^{2}$ and each $g_{i}$ has nonnegative curvature, and their pointed measured Gromov-Hausdorff limit is $[0,\infty)$ equipped with the Euclidean distance and a measure given by integration of the $1$-form $rdr$.
Thus there exists a sequence of geodesic balls $B_{4}(p_{i})\subset M_{i}$ converging in the Gromov-Hausdorff sense to $(1,9)\subset [0,\infty)$.
If there exists $(\delta_{i},1)$-splitting map $u_{i}: (B_{1}(p_{i}),p_i)\rightarrow(\mathbb{R},5)$, where $\delta_{i}\downarrow 0$, then by the $W^{1,2}$-convergence, we will obtain a harmonic function (with respect to the limit measure) $u_{\infty}:(4,6)\rightarrow\mathbb{R}$ with $|\nabla u_{\infty}|\equiv 1$.
$|\nabla u_{\infty}|\equiv 1$ implies $u_{\infty}$ is linear with respect to the $\R$-factor, but in this case it is easy to check that $u_{\infty}$ is not harmonic with respect to the measure $rdr$.

However, by the splitting theorem (\cite{Gig13}), if $(X,d,m)$ is an $\mathrm{RCD}(-\epsilon(N-1),N)$ space, $p\in X$, and $B_{4}(p)$ is $(\epsilon,k)$-Euclidean, then there exists a $(\Psi(\epsilon|N),k)$-splitting map $u: B_{\sqrt{\epsilon}}(p)\rightarrow\mathbb{R}^{k}$.
\end{rem}

The following two lemmas, which will be used many times in this paper, are easy to check.

\begin{lem}\label{lem2.11}
For any positive number $\delta \ll 1$, suppose $(X,d)$ is a metric space, $B_{r}(p)\subset X$ is $(\delta,k)$-Euclidean, then for any $s\in[r\delta^{\frac{1}{2}},r]$, $B_{s}(p)$ is $(\sqrt{\delta},k)$-Euclidean.
\end{lem}

\begin{lem}\label{lem2.12}
Let $1\leq N<\infty$ be fixed.
For any positive number $\delta \ll 1$, let $(X,d,m)$ be an $\mathrm{RCD}(-(N-1),N)$ space, $p\in X$, and suppose $u:B_{r}(p)\rightarrow\mathbb{R}^{k}$ is a $(\delta,k)$-splitting map, where $r\leq 1$,  then for any $s\in[r\delta^{\frac{1}{2N}},r]$, $u:B_{s}(p)\rightarrow\mathbb{R}^{k}$ is a $(C\sqrt{\delta},k)$-splitting map, where $C>0$ depending only on $N$.
\end{lem}

Lemma \ref{lem2.11} follows directly from the definition of $({\delta},k)$-Euclidean, and Lemma \ref{lem2.12} follows from the definition of $({\delta},k)$-splitting maps and the volume comparison theorem.

Finally, we use $\mathrm{Alex}^{n}(k)$ to denote the space of $n$-dimensional Alexandrov spaces with curvature bounded from below by $k$.
Recall that, for any $(X,d)\in \mathrm{Alex}^{n}(k)$, by \cite{Pet11} \cite{ZZ10}, $(X,d,\mathcal{H}^{n})$ is an $\RCD((n-1)k,n)$ space.
In this paper, we will consider $\RCD(K,N)$ spaces $(X,d,m)$ with $(X,d)\in \mathrm{Alex}^{n}(k)$ for some $n\leq N$, where the measure $m$ is not necessary $\mathcal{H}^{n}$.
Such spaces naturally arise as the (collapsed) measured Gromov-Hausdorff limit of a sequence of $l$-dimensional Alexandrov spaces $(X_{i},p_{i},d_{i},\frac{1}{\mathcal{H}^{l}(B_{1}(p_{i}))}\mathcal{H}^{l})$ with uniform curvature lower bound (where $n< l\leq N$).

We recall the following useful theorem.

\begin{thm}[see Corollary 5.2 in \cite{LN20}]\label{thm-LN20}
For any $n, k\in \mathbb{Z}^{+}$ and $\epsilon> 0$, there exists $\delta=\delta(n,\epsilon) > 0$ so that if $X\in \mathrm{Alex}^{n}(-\delta)$ and $B_{3}(p)$ is $(\delta,k)$-Euclidean, then $B_{r}(x)$ is $(\epsilon, k)$-Euclidean for every $x\in B_{1}(p)$ and every $r\in(0,1]$.
\end{thm}

\section{Harmonic functions on $k$-splitting $\mathrm{RCD}(0,N)$ spaces} \label{sec-3}

On an $\RCD(0,N)$ space $(X,d,m)$, given $p\in X$, $k>0$, we consider the linear space
$$\mathcal{H}_{k}(X,p)=\{u\in W^{1,2}_{\mathrm{loc}}(X)\mid \Delta u=0, u(p)=0, |u(x)|\leq C(d(x,p)^{k}+1) \text{ for some }C>0\},$$
and denote by $h_{k}(X,p)= \mathrm{dim}\mathcal{H}_{k}(X,p)$.
We will also use the notations
$$\mathcal{H}_{k}(X)=\{u\in W^{1,2}_{\mathrm{loc}}(X)\mid \Delta u=0, |u(x)|\leq C(d(x,p)^{k}+1) \text{ for some }C>0\}$$
and $h_{k}(X)= \mathrm{dim}\mathcal{H}_{k}(X)$.

By \cite{J14}, we can take the locally Lipschitz representative of $u\in\mathcal{H}_{k}(X)$.

Let $(X,d,m)$ be a non-compact $\mathrm{RCD}(0,N)$ space.
Given $\epsilon>0$, we say a harmonic function $u : X\to \mathbb{R}$ has $\epsilon$-almost linear growth or at most $(1+\epsilon)$-growth if
\begin{align}\label{2.114455}
|u(x)|\leq Cd(x,p)^{1+\epsilon}+ C
\end{align}
for some $C> 0$.

To discuss harmonic functions on $k$-splitting $\mathrm{RCD}(0,N)$ spaces, we need the following lemma:

\begin{lem}\label{partial_deri}
Suppose $(X,d,m)$ is an $\mathrm{RCD}(0,N)$ space which is $k$-splitting ($k\leq N$).
Let $x^{1},\ldots,x^{k}$ be the standard coordinates in the $\mathbb{R}^{k}$-factor.
Then for any harmonic function $f : X\to \mathbb{R}$, we have $\langle\nabla f,\nabla x^{i}\rangle\in W^{1,2}_{\mathrm{loc}}(X)\cap D(\Delta)$ and $\Delta(\langle \nabla f,\nabla x^{i}\rangle)=0$.
\end{lem}

Lemma \ref{partial_deri} just means $\mathrm{Hess}(x^{i})=0$.
It is a corollary of an intermediate step in Gigli's proof of the splitting theorem on $\RCD(0,N)$ spaces.
More precisely, in Corollary 4.14 of \cite{Gig13}, it says that on an $\mathrm{RCD}(0,N)$ space $(X,d,m)$, suppose a function $u\in W^{1,2}_{\mathrm{loc}}(X)$ satisfies $\Delta u=0$ and $|\nabla u|\equiv 1$, then for any $f\in W^{1,2}_{\mathrm{loc}}(X)\cap D(\Delta)$ with $\Delta f\in W^{1,2}_{\mathrm{loc}}(X)$ and $g\in D(\Delta)$, it holds
\begin{align}\label{3.01}
\int \Delta g\langle \nabla u,\nabla f\rangle dm =\int g\langle \nabla u,\nabla \Delta f\rangle dm.
\end{align}
Lemma \ref{partial_deri} follows if we take $u=x^{i}$ in (\ref{3.01}).

In the proof of Proposition \ref{thm126}, we need some calculus rules of $\nabla$ and $\Delta$ on a splitting $\RCD(0,N)$ spaces $(X,d,m)=(\mathbb{R}\times Y, d_{\mathrm{Eucl}}\times d _{Y}, \mathcal{L}^{1}\otimes m_{Y})$. More precisely, we use Theorem 3.13, which roughly says that for any $f\in W^{1,2}_{\mathrm{loc}}(X)$, it holds that $\nabla^{X} f=\nabla^{\R}  f+\nabla^{Y} f$, and use Corollary 3.17 of \cite{GR18}, which roughly says that a separation of variables formula for $\Delta$ holds on $X$, i.e. $\Delta_{X}=\Delta_{\R}+\Delta_{Y}$.
The authors can refer to Section 3.2 of \cite{GR18} for more details.

\subsection{Structure for harmonic functions with polynomial growth on $k$-splitting spaces}

\begin{prop}\label{thm126}
Let $(X,d_{X},m_{X})$ be an $\mathrm{RCD}(0,N)$ space which is $k$-splitting.
Suppose $X=\mathbb{R}^{k}\times Y$, $d_{X}=d_{\mathrm{Eucl}}\times d _{Y}$, $m_{X}=\mathcal{L}^{k}\otimes m_{Y}$, where $(Y,d_{Y},m_{Y})$ is an $\mathrm{RCD}(0,N-k)$ space.
Let $p=(0^{k},y_{0})\in X$.
If $u : X\to \mathbb{R}$ is a non-constant harmonic function satisfying
\begin{align}\label{1.111111}
|u(x)|\leq C(1+d_{X}(x,p)^{s+\epsilon})
\end{align}
for some $s\in \mathbb{Z}^{+}$, $\epsilon\in[0,1)$ and $C>0$.
Then
\begin{align}
u(r_{1},r_{2},\ldots,r_{k},y)=\bar{P}_{s}(r_{1},\ldots,r_{k})+ \sum_{j=1}^{s-1}\sum_{|\alpha|=j}\bar{P}_{\alpha}(r_{1},\ldots,r_{k})\bar{Q}_{\alpha}(y)+\bar{Q}_{0}(y),
\end{align}
where $\bar{P}_{s}$ is a homogeneous polynomial function of degree $s$; for each index $\alpha=(\alpha_{1},\ldots,\alpha_{k})$ with $\alpha_{i}\in \mathbb{Z}^{+}\cup\{0\}$ and $\sum_{i=1}^{k}\alpha_{i}:=|\alpha|=j\in\{1,\ldots,s-1\}$, $\bar{P}_{\alpha}(r_{1},\ldots,r_{k})=\prod_{i=1}^{k} r_{i}^{\alpha_{i}}$, $\bar{Q}_{\alpha}\in W^{1,2}_{\mathrm{loc}}(Y)\cap D(\Delta_{Y,\mathrm{loc}})$ satisfies $|\bar{Q}_{\alpha}(y)|\leq \bar{C}(1+d_{Y}(y,y_{0})^{s-j+\epsilon})$, and $\Delta_{Y}\bar{Q}_{\alpha}\in L_{\mathrm{loc}}^{2}(Y)$; $\bar{Q}_{0}\in W^{1,2}_{\mathrm{loc}}(Y)\cap D(\Delta_{Y,\mathrm{loc}})$ satisfies $|\bar{Q}_{0}(y)|\leq \bar{C}(1+d_{Y}(y,y_{0})^{s+\epsilon})$, and $\Delta_{Y}\bar{Q}_{0}\in L_{\mathrm{loc}}^{2}(Y)$ for some $\bar{C}>0$.
\end{prop}

\begin{proof}
We prove the proposition by induction.
If $s=0$, by the gradient estimate, $u$ is a constant function.
The conclusion holds for this case.

Suppose Proposition \ref{thm126} holds for $s=l$, we will prove it holds for $s=l+1$.
Since $u\in \mathcal{H}_{l+1+\epsilon}(\mathbb{R}^{k}\times Y)$, by gradient estimate and Lemma \ref{partial_deri}, for each $i\in\{1,\ldots,k\}$,
$\langle\nabla u,\nabla r_{i}\rangle\in \mathcal{H}_{l+\epsilon}(\mathbb{R}^{k}\times Y)$.
Thus by the induction assumption,
\begin{align}
\langle\nabla u,\nabla r_{i}\rangle= {P}_{l}^{(i)}(r_{1},\ldots,r_{k})+ \sum_{j=1}^{l-1}\sum_{|\alpha|=j}{P}_{\alpha}^{(i)}(r_{1},\ldots,r_{k}){Q}_{\alpha}^{(i)}(y)+{Q}_{0}^{(i)}(y),
\end{align}
with ${P}_{l}^{(i)}$ is a homogeneous polynomial function of degree $l$; for each index $\alpha$ with $|\alpha|=j\in\{1,\ldots,l-1\}$, ${P}_{\alpha}^{(i)}$ is a polynomial of degree $j$, ${Q}_{\alpha}^{(i)}\in W^{1,2}_{\mathrm{loc}}(Y)\cap D(\Delta_{Y,\mathrm{loc}})$ satisfies $|{Q}_{\alpha}^{(i)}(y)|\leq \bar{C}_{i}(1+d_{Y}(y,y_{0})^{l-j+\epsilon})$, and $\Delta_{Y}{Q}_{\alpha}^{(i)}\in L_{\mathrm{loc}}^{2}(Y)$; ${Q}_{0}^{(i)}\in W^{1,2}_{\mathrm{loc}}(Y)\cap D(\Delta_{Y,\mathrm{loc}})$ satisfies $|{Q}_{0}^{(i)}(y)|\leq \bar{C}(1+d_{Y}(y,y_{0})^{l+\epsilon})$, and $\Delta_{Y}{Q}_{0}^{(i)}\in L_{\mathrm{loc}}^{2}(Y)$ for some $\bar{C}_{i}>0$.
Thus
\begin{align}\label{a1.11111}
u(r_{1},r_{2},\ldots,r_{k},y)=&\int_{0}^{r_{1}}\langle\nabla u,\nabla s_{1}\rangle(s_{1},r_{2},\ldots,r_{k},y) ds_{1}+u(0,r_{2},\ldots,r_{k},y),
\\=&\tilde{P}^{(1)}_{l+1}(r_{1},\ldots,r_{k})+ \sum_{j=1}^{l-1}\sum_{|\alpha|=j}\tilde{P}^{(1)}_{\alpha}(r_{1},\ldots,r_{k})\tilde{Q}_{\alpha}^{(1)}(y) +{Q}_{0}^{(1)}(y)r_{1}\nonumber
\\ &+u(0,r_{2},\ldots,r_{k},y), \nonumber
\end{align}
\begin{align}\label{a1.11112}
u(0,r_{2},\ldots,r_{k},y)=&\int_{0}^{r_{2}}\langle\nabla u(0,s_{2},\ldots,r_{k},y),\nabla s_{2}\rangle ds_{2}+u(0,0,r_{3},\ldots,r_{k},y)\\
=&\tilde{P}^{(2)}_{l+1}(r_{2},\ldots,r_{k})+ \sum_{j=1}^{l-1}\sum_{|\alpha|=j}\tilde{P}^{(2)}_{\alpha}(r_{2},\ldots,r_{k})\tilde{Q}_{\alpha}^{(2)}(y)\nonumber\\
&+{Q}_{0}^{(2)}(y)r_{2}+u(0,0,r_{3},\ldots,r_{k},y),\nonumber
\end{align}
\begin{align}
\ldots\nonumber
\end{align}
\begin{align}\label{a1.11113}
u(0,\ldots,0,r_{k},y)=&\int_{0}^{r_{k}}\langle\nabla u(0,\ldots,0,s_{k},y),\nabla s_{k}\rangle ds_{k}+u(0,\ldots,0,y)\\
=&\tilde{P}^{(k)}_{l+1}(r_{k})+ \sum_{i=1}^{l-1}\sum_{|\alpha|=j}\tilde{P}^{(k)}_{\alpha}(r_{k})\tilde{Q}_{\alpha}^{(k)}(y) +{Q}_{0}^{(k)}(y)r_{2}\nonumber\\
&+u(0,\ldots,0,y),\nonumber
\end{align}
where for every $i=1,\ldots,k$, $\tilde{P}_{l+1}^{(i)}$ is a homogeneous polynomial function of degree $l+1$, $\tilde{P}_{\alpha}^{(i)}$ is a polynomial function of degree $|\alpha|+1$, $\tilde{Q}_{\alpha}^{(i)}={Q}_{\alpha}^{(i)}$.
Obviously, $|u(0,\ldots,0,y)|\leq C(1+d_{Y}(y,y_{0})^{l+1+\epsilon})$.
In (\ref{a1.11111})-(\ref{a1.11113}), we use the rule in Theorem 3.13 of \cite{GR18}, and we note that (\ref{a1.11111})-(\ref{a1.11113}) hold for every $(r_{1},r_{2},\ldots,r_{k},y)$ as we have assumed $u$ and $\langle\nabla u,\nabla r_{i}\rangle$ are all Lipschitz.

By (\ref{a1.11111})-(\ref{a1.11113}), we obtain
\begin{align}
u(r_{1},r_{2},\ldots,r_{k},y)=\bar{P}_{l+1}(r_{1},\ldots,r_{k})+ \sum_{j=1}^{l}\sum_{|\alpha|=j}\bar{P}_{\alpha}(r_{1},\ldots,r_{k})\bar{Q}_{\alpha}(y)+\bar{Q}_{0}(y) \nonumber
%where $\tilde{P}_{\alpha}^{(2)}(r_{1},\ldots,r_{k})$ is a polynormaial function of degree at most $|\alpha|+1$.
\end{align}
for suitable polynomials $\bar{P}_{l+1}$, $\bar{P}_{\alpha}$ on $\mathbb{R}^{k}$ and functions $\bar{Q}_{\alpha}$, $\bar{Q}_{0}$ on $Y$ satisfying the conclusions in the theorem.
Note that to conclude $\Delta_{Y}\bar{Q}_{\alpha}\in L_{\mathrm{loc}}^{2}(Y)$ and $\Delta_{Y}\bar{Q}_{0}\in L_{\mathrm{loc}}^{2}(Y)$ we use the rules in Corollary 3.17 of \cite{GR18}.

The proof is completed.
\end{proof}

\begin{rem}\label{cor126}
On a $k$-splitting $\mathrm{RCD}(0,N)$ space $(X, d_{X}, m_{X}) =(\mathbb{R}^{k}\times Y, d_{\mathrm{Eucl}}\times d _{Y}, \mathcal{L}^{k}\otimes m_{Y})$,
if $u : X\to \mathbb{R}$ is a non-constant harmonic function satisfying
\begin{align}
\lim_{d(x,p)\rightarrow\infty}\frac{|u(x)|}{d(x,p)^{s+1}}=0,
\end{align}
for some $s\in \mathbb{Z}^{+}$, then similar to Proposition \ref{thm126}, we also have
\begin{align}
u(r_{1},r_{2},\ldots,r_{k},y)=\bar{P}_{s}(r_{1},\ldots,r_{k})+ \sum_{j=1}^{s-1}\sum_{|\alpha|=j}\bar{P}_{\alpha}(r_{1},\ldots,r_{k})\bar{Q}_{\alpha}(y)+\bar{Q}_{0}(y),
\end{align}
where $\bar{P}_{s}$, $\bar{P}_{\alpha}$ are polynomial functions on $\R^k$ satisfying the same properties as in Proposition \ref{thm126}, and $\bar{Q}_{\alpha}$, $\bar{Q}_{0}$ are functions on $Y$ satisfying suitable asymptotic growth property.
\end{rem}

The following result is important in the proof of the gap theorems \ref{thm-gap-har-RCD} and \ref{thm-gap-har-Alex}.

\begin{cor}\label{cor3.6}
Suppose $(X,d,m)$ is an $\mathrm{RCD}(0,N)$ space which is $k$-splitting.
If $u : X\to \mathbb{R}$ is a non-constant harmonic function satisfying
\begin{align}
\lim_{d(x,p)\rightarrow\infty}\frac{|u(x)|}{d(x,p)^{2}}=0,
\end{align}
then for any $\mathbb{R}$-coordinate function $v$, $\langle\nabla u, \nabla v\rangle$ is a constant function.
\end{cor}
\begin{proof}
Without loss of generality, we assume $v=r_{1}$ is the first $\mathbb{R}$-coordinate of the $k$-splitting $\mathrm{RCD}(0,N)$ space $(X, d_{X}, m_{X}) =(\mathbb{R}^{k}\times Y, d_{\mathrm{Eucl}}\times d _{Y}, \mathcal{L}^{k}\otimes m_{Y})$.
By Proposition \ref{thm126} or Remark \ref{cor126},
$u(r_{1},r_{2},\ldots,r_{k},y)=\sum_{i=1}^{k} c_{i}r_{i}+\bar{Q}_{0}(y)$, where $c_{i}$'s are constants.
Thus $\langle\nabla u, \nabla v\rangle\equiv c_{1}$.
\end{proof}

\subsection{Harmonic functions with linear growth}

It is proved in \cite{HKX13} that, on a noncompact $\mathrm{RCD}(0,N)$ space (where $N\in \mathbb{Z}^{+}$), $h_{1}(X,p)\leq N$.
This generalizes a theorem in \cite{LT89}.
The following proposition considers the relation of $h_{1}(X,p)$ and the tangent cone at infinity of $X$, and Corollary \ref{thm-har-rigidity} considers the case $h_{1}(X,p)=N$.
These two results generalize the main theorems in \cite{CCM}.

\begin{prop}\label{thm-split-infin}
Suppose $(X,d,m)$ is a noncompact $\mathrm{RCD}(0,N)$ space.
Suppose $h_{1}(X,p)=k$.
Then any tangent cone at infinity of $X$ is $k$-splitting.
\end{prop}

\begin{proof}
By the gradient estimate, for any non-zero function $f\in{\mathcal{H}}_{1}(X,p)$, there is some $C>0$ such that $\sup|\nabla f|^{2}= C$.
By the mean value theorem for bounded subharmonic functions on $\mathrm{RCD}(0,N)$ spaces (see Theorem 5.4 in \cite{HKX13}, see also \cite{L86} for a theorem on manifolds), we have
\begin{align}
\lim_{R\rightarrow\infty} \bbint_{B_{R}(p)}|\nabla f|^{2}=\sup|\nabla f|^{2}.
\end{align}
Then we can introduce an inner product on ${\mathcal{H}}_{1}(X,p)$ by
\begin{align}
( f_{i}, f_{j})&:=\lim_{R\rightarrow\infty} \frac{1}{4}\bbint_{B_{R}(p)}\bigl(|\nabla (f_{i}+f_{j})|^{2}-|\nabla (f_{i}-f_{j})|^{2}\bigr)dm \\
&=\lim_{R\rightarrow\infty} \frac{1}{4}\bbint_{B_{R}(p)}\langle\nabla f_{i},\nabla f_{j}\rangle dm.\nonumber
\end{align}

Now we fix an orthonormal base $\{u_{j}\}_{1\leq j\leq k}\subset {\mathcal{H}}_{1}(X,p)$ with respect to this inner product.

Let $(\tilde{X}_{\infty},\tilde{p}_{\infty},\tilde{d}_{\infty},\tilde{m}_{\infty})$ be any tangent cone at infinity of $X$.
Suppose $(\tilde{X}_{i},\tilde{p}_{i},\tilde{d}_{i},\tilde{m}_{i})=(X,p,\frac{1}{r_{i}}d,\frac{1}{m(B_{r_{i}}(p))}m)$ (where $r_{i}\rightarrow+\infty$) converges to $(\tilde{X}_{\infty},\tilde{p}_{\infty},\tilde{d}_{\infty},\tilde{m}_{\infty})$ in the pointed measured Gromov-Hausdorff distance.
Denote by $u_{j}^{(i)}=\frac{1}{r_{i}}u_{j}$, then $\sup|\tilde{\nabla} u_{j}^{(i)}|^{2}=1$, and
$\sup|\tilde{\nabla} (u_{j}^{(i)}+u_{l}^{(i)})|^{2}=2$ for $j\neq l$.
By Theorems \ref{AA} and \ref{2.7777777}, up to passing to a subsequence, $u_{j}^{(i)}$ converges locally uniformly and locally $W^{1,2}$ to a harmonic function $u_{j}^{(\infty)}: \tilde{X}_{\infty}\rightarrow \mathbb{R}$.
Thus for any $R>0$,
\begin{align}\label{3.14}
\bbint_{{B}_{R}(\tilde{p}_{\infty})}|\tilde{\nabla} u_{j}^{(\infty)}|^{2}=\lim_{i\rightarrow\infty} \bbint_{{B}_{R}(\tilde{p}_{i})}|\tilde{\nabla} u_{j}^{(i)}|^{2}=\lim_{i\rightarrow\infty} \bbint_{B_{r_{i}R}(p)}|\nabla u_{j}|^{2}=1,
\end{align}
and
\begin{align}\label{3.15}
\bbint_{B_{R}(\tilde{p}_{\infty})}\langle\tilde{\nabla} u_{j}^{(\infty)},\tilde{\nabla} u_{l}^{(\infty)}\rangle =\lim_{i\rightarrow\infty} \bbint_{B_{R}(\tilde{p}_{i})}\langle\tilde{\nabla} u_{j}^{(i)},\tilde{\nabla} u_{l}^{(i)}\rangle =\lim_{i\rightarrow\infty} \bbint_{B_{r_{i}R}(p)}\langle\nabla u_{j},\nabla u_{l}\rangle=0
\end{align}
for any $j\neq l$.

Note that $\sup|\tilde{\nabla} u_{j}^{(\infty)}|^{2}=1$, hence by (\ref{3.14}), $|\tilde{\nabla} u_{j}^{(\infty)}|^{2}\equiv1$ for each $j$.
Then by (\ref{3.15}) and the fact that $\sup|\tilde{\nabla} (u_{j}^{(\infty)}+u_{l}^{(\infty)})|^{2}=2$ for $j\neq l$, we have $\langle\tilde{\nabla} u_{j}^{(\infty)}, \tilde{\nabla} u_{l}^{(\infty)}\rangle\equiv 0$.
Thus by Theorem \ref{thm-split-RCD},
$(\tilde{X}_{\infty},\tilde{p}_{\infty},\tilde{d}_{\infty},\tilde{m}_{\infty})$ is $k$-splitting.
\end{proof}

\begin{cor}\label{thm-har-rigidity}
Given $N\in\mathbb{Z}^{+}$.
Suppose $(X,d,m)$ is a noncompact $\mathrm{RCD}(0,N)$ space with
$h_{1}(X)=N+1$.
Then $(X,d,m)$ is isomorphic to $(\mathbb{R}^{N}, d_{\mathrm{Eucl}},c\mathcal{H}^{N})$ for some constant $c>0$.
\end{cor}

\begin{proof}
By Proposition \ref{thm-split-infin}, every tangent cone at infinity of $(X,d,m)$ is isomorphic to $(\mathbb{R}^{N}, d_{\mathrm{Eucl}},c_{1}\mathcal{H}^{N})$ for some constant $c_{1}>0$.
Then by the lower semicontinuity of essential dimensions with respect to pointed measured Gromov-Hausdorff convergence (see Theorem 1.5 of \cite{Ki19}), we know the essential dimension of $X$ is equal to $N$.
Thus $(X,d,m)$ is weakly non-collapsed in the sense of \cite{DePGil18}.
By the main result of a recent paper \cite{BGHZ23}, we know weakly non-collapsed $\RCD$ spaces are non-collapsed in the sense of \cite{DePGil18}, i.e. $m=c\mathcal{H}^{N}$ for some constant $c>0$.
Then by Theorem 1.6 in \cite{DePGil18}, it is easy to see that $(X,d,m)$ is isomorphic to $(\mathbb{R}^{N}, d_{\mathrm{Eucl}},c\mathcal{H}^{N})$.
\end{proof}

Proposition \ref{thm-split-infin} can be strengthened when $(X,d)$ is an Alexandrov space with nonnegative curvature:

\begin{prop}\label{thm-split-alex}
Suppose $(X,d,m)$ is an $\RCD(0,N)$ space and $(X,d)\in \mathrm{Alex}^{n}(0)$ ($n\leq N$).
If $u : X\to \mathbb{R}$ is a non-constant harmonic function with linear growth, then $X$ splits off an $\mathbb{R}$-factor.
\end{prop}

\begin{proof}

After normalization, we can assume that $u(p)=0$ and $\sup|\nabla u|^{2}= 1$.
Following the proof of Theorem \ref{thm-split-infin}, suppose for some sequence $r_{i}\rightarrow+\infty$, $(\tilde{X}_{i},\tilde{p}_{i},\tilde{d}_{i},\tilde{m}_{i})=(X,p,\frac{1}{r_{i}}d,\frac{1}{m(B_{r_{i}}(p))}m) \xrightarrow{pmGH}(\tilde{X}_{\infty},\tilde{p}_{\infty},\tilde{d}_{\infty},\tilde{m}_{\infty})$,
with $(\tilde{X}_{\infty},\tilde{p}_{\infty},\tilde{d}_{\infty},\tilde{m}_{\infty})$ being an $\RCD(0,N)$ space and $(\tilde{X}_{\infty},\tilde{d}_{\infty})\in \mathrm{Alex}^{n'}(0)$ ($n'\leq n\leq N$),
and $u^{(i)}=\frac{1}{r_{i}}u$ converges to a harmonic function $u^{(\infty)}: \tilde{X}_{\infty}\rightarrow \mathbb{R}$ in locally uniform and locally $W^{1,2}$-sense, and $|\tilde{\nabla} u^{(\infty)}|^{2}\equiv1$.
Then by Theorem \ref{thm-split-RCD}, $\tilde{X}_{\infty}$ splits off a line, and then by Theorem \ref{thm-LN20}, $X$ itself splits off a line.
\end{proof}

\subsection{A gap theorem for harmonic functions with almost linear growth}\label{subsec3.3}

Based on Proposition \ref{thm-split-infin}, we will prove a gap theorem for harmonic functions with almost linear growth in this section, see Theorem \ref{thm-gap-har-RCD} below.
This gap theorem plays an important role in this paper.

\begin{thm}\label{thm-gap-har-RCD}
Given $\eta> 0$ and $N\geq k\geq1$ with $k\in \mathbb{Z}^{+}$, there exists $\epsilon= \epsilon(N,\eta) > 0$ such that the following holds.
Suppose $(X,d,m)$ is an $\mathrm{RCD}(0,N)$ space which is $k$-splitting, and $B_{r}(p)(\subset X)$ is not $(\eta,k+1)$-Euclidean for any $r\geq 1$.
If $u : X\to \mathbb{R}$ is a harmonic function with $\epsilon$-almost linear growth,
then $u$ is a linear combination of the $\mathbb{R}^{k}$-coordinates in $X$.
In particular, $h_{1+\epsilon}(X)=k+1$.
\end{thm}

\begin{proof}
Without loss of generality, we assume $u$ is a non-constant function and $u(p)=0$.
Let $x^{1},\ldots,x^{k}$ be the standard coordinates in $\mathbb{R}^{k}$-factor such that $x^{j}(p)=0$.
Then by Corollary \ref{cor3.6}, $\langle \nabla u,\nabla x^{j}\rangle$ is a constant function.
Denote by $a^{j}=\langle \nabla u,\nabla x^{j}\rangle$, and we choose suitable $C'$ such that
$$v(x):=\frac{1}{C'}(u(x)-\Sigma_{j=1}^{k}a^{j}x^{j})$$
satisfies the followings:
\begin{align}\label{0.111111}
v(p)=0,
\end{align}
\begin{align}\label{0.111112}
|v(x)|\leq d(x,p)^{1+\epsilon}+4,
\end{align}
\begin{align}\label{0.111113}
\langle \nabla v, \nabla x^{j}\rangle\equiv0  \text{ for every  } j\in\{1,\ldots,k\}.
\end{align}

The theorem is proved provided we can prove $v\equiv0$.

We claim that, there exists a positive constant $\epsilon(n,\eta)\ll 1$, depending on $n$, $\eta$, such that, for any $\epsilon\in(0,\epsilon(n,\eta))$, if $(X,d,m)$ is an $\mathrm{RCD}(0,N)$ space which is $k$-splitting and $B_{r}(p)$ is not $(\eta, k+1)$-Euclidean for any $r\geq 1$,
and $v: X\rightarrow \mathbb{R}$ is a harmonic function satisfying (\ref{0.111111}), (\ref{0.111112}) and (\ref{0.111113}), then
\begin{align} \label{1.25852}
|v(x)|\leq\frac{1}{2}d(x,p)^{1+\epsilon} \text{ for any } x\in X\setminus B_{1}(p).
\end{align}

Suppose the claim does not hold.
Then there exist $\epsilon_{i}\downarrow0$, $\{(X_{i},d_{i},m_{i})\}$, and $\{v_{i}\}$ such that, each $(X_{i},d_{i},m_{i})$ is an $\mathrm{RCD}(0,N)$ space which is $k$-splitting and $B_{r}(p_{i})$ is not $(\eta,k+1)$-Euclidean for any $r\geq 1$ (where $p_{i}\in X_{i}$), each $v_{i}:X_{i}\rightarrow \mathbb{R}$ is a harmonic function satisfying (\ref{0.111111}), (\ref{0.111113}) and
\begin{align}\label{0.111114}
|v_{i}(x)|\leq d_{i}(x,p_{i})^{1+\epsilon_{i}}+4,
\end{align}
but there exist $x_{i}$ satisfying $R_{i}:=d_{i}(x_{i},p_{i})\geq1$ and $|v_{i}(x_{i})|>\frac{1}{2}d_{i}(x_{i},p_{i})^{1+\epsilon_{i}}$.

Denote by $(\tilde{X}_{i},\tilde{p}_{i},\tilde{d}_{i},\tilde{m}_{i}):= (X_{i},p_{i},\frac{1}{R_{i}}d_{i},\frac{1}{m_{i}(B_{R_{i}}(p_{i}))}m_{i})$.
Passing to a subsequence, we assume $(\tilde{X}_{i},\tilde{p}_{i},\tilde{d}_{i},\tilde{m}_{i})$ converges to an $\mathrm{RCD}(0,N)$ space $(\tilde{X}_{\infty},\tilde{p}_{\infty},\tilde{d}_{\infty},\tilde{m}_{\infty})$ in the pointed measured Gromov-Hausdorff distance.
Then $\tilde{X}_{\infty}$ is $k$-splitting, and for any $r\geq 1$, $B_{r}(\tilde{p}_{\infty})$ is not $(\frac{\eta}{2},k+1)$-Euclidean.
We may assume that, the standard coordinates $\tilde{x}^{i,1},\ldots,\tilde{x}^{i,k}$ of $\mathbb{R}^{k}$-factor in $\tilde{X}_{i}$ with $\tilde{x}^{i,j}(p_{i})=0$, converge in locally uniform and locally $W^{1,2}$-sense to coordinates systems $\tilde{x}^{\infty,1},\ldots,\tilde{x}^{\infty,k}$ of the $\mathbb{R}^{k}$-factor in $\tilde{X}_{\infty}$.

Let $\tilde{v}_{i}:\tilde{X}_{i}\rightarrow\mathbb{R}$ be given by $\tilde{v}_{i}(x)=\frac{1}{R_{i}^{1+\epsilon_{i}}}{v}_{i}(x)$, then every $\tilde{v}_{i}$ is harmonic, and
\begin{align}
|\tilde{v}_{i}(x)|\leq \frac{1}{R_{i}^{1+\epsilon_{i}}}(d_{i}(x,p_{i})^{1+\epsilon_{i}}+4)\leq \tilde{d}_{i}(x,\tilde{p}_{i})^{1+\epsilon_{i}}+4.
\end{align}
By gradient estimate,
\begin{align}
\sup_{x\in {B}_{r}(\tilde{p}_{i})}|\nabla\tilde{v}_{i}(x)|\leq C+Cr^{\epsilon_{i}}
\end{align}
for every $r>0$, where $C$ is a positive constant depending only on $N$.
Thus by Theorems \ref{AA} and \ref{2.7777777}, up to passing to a subsequence, $\tilde{v}_{i}$ converges in locally uniform and locally $W^{1,2}$-sense to a harmonic function $\tilde{v}_{\infty}:\tilde{X}_{\infty}\rightarrow\mathbb{R}$ with
\begin{align}\label{0.69955}
|\tilde{v}_{\infty}(x)|\leq \tilde{d}_{\infty}(x,\tilde{p}_{\infty})+4.
\end{align}
Note that $\tilde{v}_{\infty}$ is not a constant function because $\tilde{v}_{\infty}(\tilde{p}_{\infty})=0$ and there exists some $\tilde{x}_{\infty}\in \partial {B}_{1}(\tilde{p}_{\infty})$ such that $\tilde{v}_{\infty}(\tilde{x}_{\infty})\geq\frac{1}{2}$.
By Corollary \ref{cor3.6},
$\langle \nabla \tilde{v}_{\infty}, \nabla \tilde{x}^{\infty,j}\rangle$ is a constant function for every $j\in\{1,\ldots,k\}$,
and thus
\begin{align}
&\langle \nabla \tilde{v}_{\infty}, \nabla \tilde{x}^{\infty,j}\rangle=\bbint_{{B}_{1}(\tilde{p}_{\infty})}\langle \nabla \tilde{v}_{\infty}, \nabla \tilde{x}^{\infty,j}\rangle d\tilde{m}_{\infty} \\
=&\lim_{i\rightarrow\infty} \bbint_{{B}_{1}(\tilde{p}_{i})}\langle \nabla \tilde{v}_{i}, \nabla \tilde{x}^{i,j}\rangle d\tilde{m}_{i}=0. \nonumber
\end{align}
Hence $\tilde{v}_{\infty}\notin \mathrm{span}\{\tilde{x}^{\infty,1},\ldots,\tilde{x}^{\infty,k}\}$.
Then by Proposition \ref{thm-split-infin}, every tangent cone at infinity of $\tilde{X}_{\infty}$ is $(k+1)$-splitting, and thus $B_{r}(\tilde{p}_{\infty})\subset \tilde{X}_{\infty}$ is $(\frac{\eta}{4},k+1)$-Euclidean for some large $r$, which is a contradiction.
This completes the proof of the claim.

By (\ref{1.25852}), $\sup_{x\in\partial B_{2}(p)}|v(x)|\leq 2$, and by the maximum principle (see Theorem 7.17 in \cite{C99}), $\sup_{x\in B_{2}(p)}|v(x)|\leq 2$.
Hence
\begin{align}
|v(x)|\leq \frac{1}{2}(d(x,p)^{1+\epsilon}+4) \text{ on }X.
\end{align}

Let $\hat{v}=2v$, then $\hat{v}$ is a harmonic function satisfying (\ref{0.111111}), (\ref{0.111112}) and (\ref{0.111113}).
Then by the above claim and the maximum principle again, we have $|\hat{v}(x)|\leq \frac{1}{2}(d(x,p)^{1+\epsilon}+4)$, i.e. $|v(x)|\leq \frac{1}{2^{2}}(d(x,p)^{1+\epsilon}+4)$.
By induction, we can prove $|v(x)|\leq \frac{1}{2^{j}}(d(x,p)^{1+\epsilon}+4)$ for any $j\in\mathbb{N}^{+}$.
Thus $v(x)\equiv0$.

The proof is completed.
\end{proof}

Using the method as in the proof of Theorem \ref{thm-gap-har-RCD}, we can prove the following Liouville-type theorem, whose proof is omitted here.

\begin{prop}\label{thm-liou-har-RCD}
Given $N\geq 1$ and $\eta> 0$, there exists $\epsilon= \epsilon(N,\eta) > 0$ such that the following holds.
Suppose $(X,p,d,m)$ is an $\mathrm{RCD}(0,N)$ space such that $B_{R}(p)$ is not $(\eta,1)$-Euclidean for any $R\geq 1$.
If $u : X\to \mathbb{R}$ is a harmonic function with $\epsilon$-almost linear growth,
then $u$ is constant.
\end{prop}

Proposition \ref{thm-liou-har-RCD} can be strengthened on Alexandrov spaces with nonnegative curvature:

\begin{prop}\label{thm-liouville-Alex}
For any $N\geq 1$ and $\eta> 0$, there exists $\epsilon= \epsilon(N,\eta) > 0$ such that the following holds.
Suppose $(X,d,m)$ is an $\RCD(0,N)$ space and $(X,d)\in \mathrm{Alex}^{n}(0)$ ($n\leq N$).
Suppose in addition $B_{1}(p)\subset X$ is not $(\eta,1)$-Euclidean and $u : X\to \R$ is a harmonic function with $\epsilon$-almost linear growth, then $u$ is constant.
\end{prop}

\begin{proof}
The key is to prove the following claim:

\textbf{Claim: }For any $N\geq 1$ and $\eta> 0$, there exists $\epsilon= \epsilon(N,\eta) > 0$ such that the following holds. Suppose $(X,d,m)$ is an $\RCD(0,N)$ space, $(X,d)\in \mathrm{Alex}^{n}(0)$ ($n\leq N$), and  $B_{1}(p)$ is not $(\eta,1)$-Euclidean, then any non-constant harmonic function $u : X\to\R$ with $u(p)=0$ and
	\begin{align}\label{2.114455123}
		|u(x)|\leq d(x,p)^{1+\epsilon}+ 4
	\end{align}
must satisfy
\begin{align} \label{1.25852755545}
|u(x)|\leq\frac{1}{2}d(x,p)^{1+\epsilon} \text{ for any } x\in X\setminus B_{1}(p).
\end{align}

Once this claim is proved, then together with the maximum principle, and by an induction argument as in the proof of Theorem \ref{thm-gap-har-RCD}, we derive that
$u(x)$ is identical to $0$, which is a contradiction.

We suppose the claim does not hold.
Then for some $\eta>0$, $N\geq 1$, we have a sequence of positive numbers $\epsilon_{i}\rightarrow 0$ and contradiction sequences $\{(X_{i},p_{i}, d_{i},m_{i})\}$, $\{u_{i}\}$ with $u_{i}(p_{i})=0$ and $|u_{i}(x)|\leq d_{i}(x,p_{i})^{1+\epsilon_{i}}+ 4$, but $|u_{i}(x_{i})|>\frac{1}{2}d_{i}(x_{i},p_{i})^{1+\epsilon_{i}}$ for some $x_{i}$ with $R_{i}:=d_{i}(x_{i},p_{i})\geq1$.
Up to passing to a subsequence, we assume $(\tilde{X}_{i},\tilde{p}_{i},\tilde{d}_{i},\tilde{m}_{i}):= (X_{i},p_{i},\frac{1}{R_{i}}d_{i},\frac{1}{m_{i}(B_{R_{i}}(p_{i}))}m_{i})$ converges in the pointed measured Gromov-Hausdorff distance to some $\mathrm{RCD}(0,N)$ space $(\tilde{X}_{\infty},\tilde{p}_{\infty},\tilde{d}_{\infty},\tilde{m}_{\infty})$ and $(\tilde{X}_{\infty},\tilde{d}_{\infty})\in \mathrm{Alex}^{n'}(0)$ for some $n'\leq n\leq N$.
Similar to the proof of Theorem \ref{thm-gap-har-RCD}, suitable scaling of the above $\{u_{i}\}$ converges to a non-constant harmonic function of linear growth on $\tilde{X}_{\infty}$.
By Proposition \ref{thm-split-alex}, we conclude that $\tilde{X}_{\infty}$ splits.
Thus for sufficiently large $i$, $B_{R_{i}}(p_{i})=B_{1}(\tilde{p}_{i})$ is $(\delta_{i},1)$-Euclidean with $\delta_{i}\downarrow 0$.
Then by Theorem \ref{thm-LN20}, $B_{1}(p_{i})$ is $(\Psi(\delta_{i}|N),1)$-Euclidean, and we obtain a contradiction.
\end{proof}

Similarly, we can prove the following proposition, which strengthens Theorem \ref{thm-gap-har-RCD} in the case that $(X,d)$ is an Alexandrov space with nonnegative curvature.

\begin{prop}\label{thm-gap-har-Alex}
Given $\eta> 0$ and $N\geq n\geq k\geq 1$ with $k,n\in\mathbb{Z}^{+}$, there exists $\epsilon= \epsilon(N,\eta) > 0$ such that the following holds.
Suppose $(X,d,m)$ is an $\RCD(0,N)$ space and $(X,d)\in \mathrm{Alex}^{n}(0)$.
Suppose in addition $X$ is $k$-splitting and $B_{1}(p)\subset X$ is not $(\eta,k+1)$-Euclidean.
Let $u : X\to \mathbb{R}$ be a non-constant harmonic function with $\epsilon$-almost linear growth,
then $u$ is a linear function.
In particular, $h_{1+\epsilon}(X)=k+1$.
\end{prop}

The proof is omitted.

\section{Transformation theorems and the non-degeneracy theorem} \label{sec-4}

\subsection{Proof of Theorem \ref{thm-splitting-stable}}\label{sec-4.1}

In this subsection, we prove the transformation theorem \ref{thm-splitting-stable}. In fact, we will prove the following theorem, which is equivalent to Theorem \ref{thm-splitting-stable}, see Remark \ref{rem4.1}.

\begin{thm}\label{thm-splitting-stable-full}
Given any $N \geq 1$, $\epsilon>0$ and $\eta>0$, there exists $\delta_{0}=\delta_0(N, \epsilon, \eta)$ such that, for any $\delta\in(0,\delta_{0})$, the following holds.
	Suppose $(X, d, m)$ is an $\mathrm{RCD}(-(N-1)\delta,N)$ space, and there is some $s\in(0,1)$ such that, for any $r\in[s,1]$, $B_{r}(p)$ is $(\delta, k)$-Euclidean but not $(\eta, k+1)$-Euclidean (where $k\leq N$).
	Let $u : (B_{1}(p),p)\rightarrow(\mathbb{R}^{k},0^k)$ be a $(\delta, k)$-splitting map,
	then for each $r\in[s,1]$, there exists a $k\times k$ lower triangle matrix $T_{r}$ with positive diagonal entries satisfying the followings:
\begin{description}
  \item[(1)] $T_{r}u : B_{r}(p)\rightarrow\mathbb{R}^{k}$ is a $(\epsilon,k)$-splitting map;
  \item[(2)] $\bbint_{B_{r}(p)}\langle\nabla(T_{r}u)^{a},\nabla(T_{r}u)^{b}\rangle=\delta^{ab}$;
  \item[(3)] $|T_{r}\circ T_{2r}^{-1}-\mathrm{Id}|\leq\epsilon$, here $|\cdot|$ means $L^{\infty}$-norm of a matrix;
  \item[(4)] for any $t, r \in [s,1]$,
\begin{align}
| T_{r}\circ T_{t}^{-1}|\leq  (1+C\epsilon) \max\{\bigl(\frac{t}{r}\bigr)^{C\epsilon}, \bigl(\frac{r}{t}\bigr)^{C\epsilon}\}
\end{align}
holds for a constant $C=C(N) > 1$.
\end{description}
\end{thm}

The proof of Theorem \ref{thm-splitting-stable-full} follows closely that of Proposition 7.7 in \cite{CJN21}.
In fact, the difference between the two proofs is that we use the gap theorem \ref{thm-gap-har-RCD} to replace Lemma 7.8 in \cite{CJN21}.
We provide a detailed proof here, one reason is for convenience of the readers, another reason is that some ingredients of its proof will be used later in this paper.

Firstly, we give some preliminary lemmas.

\begin{lem}\label{lem-Ts}
Under the assumptions of Theorem \ref{thm-splitting-stable-full}, there exists a constant $C =C(N) > 1$ such that, if $T_{r}$ and $T_{t}$ are matrixes verifying $\mathbf{(1)}$ and $\mathbf{(2)}$ at scale $r$ and $t\in[r,2r]$ respectively, then automatically
\begin{align}\label{4.11}
|T_{r}\circ T_{t}^{-1}-\mathrm{Id}|\leq C\epsilon.
\end{align}
\end{lem}

The proof of Lemma \ref{lem-Ts} is the same as that of Lemma 7.9 in \cite{CJN21}.
The key in the proof is the uniqueness of Cholesky decomposition of a positive definite symmetric matrix, see \cite{CJN21} for details.

\begin{lem}\label{lem-Ts-growth-estimate}
Under the assumptions of Theorem \ref{thm-splitting-stable-full}, there exists a constant $C=C(N) > 1$ such that, if $T_{t}$ and $T_{r}$ are matrixes verifying $\mathbf{(1)}$ and $\mathbf{(2)}$ at scales $t, r \in [s,1]$ respectively, then
\begin{align}\label{holdergrowth}
| T_{r}\circ T_{t}^{-1}|\leq (1+C\epsilon) \max\{\bigl(\frac{t}{r}\bigr)^{C\epsilon}, \bigl(\frac{r}{t}\bigr)^{C\epsilon}\}.
\end{align}
\end{lem}

The proof of Lemma \ref{holdergrowth} is the same as \cite{CJN21}.
We provide its proof for completeness.

\begin{proof}
We only prove the case $s\leq t\leq r\leq 1$, the other case is proved similarly.
Firstly, by Lemma \ref{lem-Ts}, it is easy to see that, there is some $C=C(N)>1$ such that, for any $r,t$ with $s\leq \frac{r}{2}\leq t \leq r$, it holds
\begin{align}\label{4.4}
|T_{r}\circ T_{t}^{-1}-\mathrm{Id}|\leq C\epsilon.
\end{align}
Note that for $k\times k$ matrices $A_{1}, A_{2}$, by triangle inequality, we have
\begin{align}\label{4.2222}
|A_{1}A_{2}-\mathrm{Id}|\leq |A_{1}-\mathrm{Id}| + |A_{2}-\mathrm{Id}| + k|A_{1}-\mathrm{Id}||A_{2}-\mathrm{Id}|.
\end{align}
For any $t\in[s,r)$, we choose $l\in \mathbb{Z}^{+}$ such that $2^{-l}r\leq t <2^{-l+1}r$.
Thus by (\ref{4.4}), (\ref{4.2222}) and an induction argument, up to replacing $C$ by a larger one (depending only on $N$), we can prove
\begin{align}\label{4.2223321}
|T_{r}\circ T_{t}^{-1}-\mathrm{Id}|\leq  (1+C\epsilon)^{l}-1,
\end{align}
and hence
\begin{align}\label{holdergrowth-1}
| T_{r}\circ T_{t}^{-1}|\leq (1+C\epsilon)^{\log_{2}(\frac{r}{t})+1}\leq (1+C\epsilon) \bigl(\frac{r}{t}\bigr)^{C\epsilon}
\end{align}
for suitable $C=C(N)>1$.
\end{proof}

\begin{rem}\label{rem4.1}
Note that once we have proved that there exists $T_{r}$ satisfying $\mathbf{(1)}$, then we can find a $k\times k$ lower triangle matrix $A_{r}$ with positive diagonal entries and $|A_{r}-\mathrm{Id}|\leq \epsilon$ such that $\tilde{T}_{r}=A_{r}T_{r}$ satisfies $\mathbf{(1)}$ (with $\epsilon$ being replaced by $C(N)\epsilon$ for some $C(N)>1$) and $\mathbf{(2)}$.
In addition, according to Lemmas \ref{lem-Ts} and \ref{lem-Ts-growth-estimate}, $\mathbf{(3)}$ and $\mathbf{(4)}$ hold for $\tilde{T}_{r}$.
Thus Theorem \ref{thm-splitting-stable} implies Theorem \ref{thm-splitting-stable-full}.
And it is obvious that Theorem \ref{thm-splitting-stable-full} implies Theorem \ref{thm-splitting-stable}, hence these two theorems are equivalent.
\end{rem}

Now we prove Theorem \ref{thm-splitting-stable-full}.
By Lemmas \ref{lem-Ts} and \ref{lem-Ts-growth-estimate}, we only need to verify $\mathbf{(1)}$ and $\mathbf{(2)}$.

\begin{proof}[Proof of Theorem \ref{thm-splitting-stable-full}]
Suppose the conclusion in the theorem does not hold, then we have contradicting sequences of $\RCD$ spaces and harmonic maps on them.
Making use of Lemma \ref{lem2.12}, we can suitably blow up this contradiction sequence to obtain the followings:
\begin{description}
  \item[(i)] there exist $\eta>0$, $\epsilon\ll1$, $\delta_{i}\downarrow0$, $r_{i}>0$ and $\mathrm{RCD}(-(N-1)\delta_{i},N)$ space $(X_{i}, p_{i}, d_{i}, m_{i})$ such that $B_{r}(p_{i})$ is $(\delta_{i},k)$-Euclidean but not $(\eta,k+1)$-Euclidean for every $r\in [r_{i},\delta_{i}^{-1}]$;
  \item[(ii)] there exist $(\delta_{i},k)$-splitting maps $u_{i}:(B_{2} (p_{i}),p_i) \rightarrow(\mathbb{R}^{k},0^k)$ and there exists $s_{i} > r_{i}$ such that for every $r\in [s_{i},1]$, there exists a lower triangle matrix $T_{p_{i},r}$ with positive diagonal entries such that $T_{p_{i},r}u_{i}$ is a $(\epsilon,k)$-splitting on $B_{r}(p_{i})$ with $\bbint_{B_{r}(p_{i})}\langle\nabla(T_{p_{i},r}u_{i})^{a},\nabla(T_{p_{i},r}u_{i})^{b}\rangle=\delta^{ab}$;
  \item[(iii)] no such matrix $T_{i}=T_{p_{i},\frac{s_{i}}{10}}$ exists on $B_{\frac{s_{i}}{10}}(p_{i})$.
\end{description}
Since $\delta_{i}\rightarrow 0$, by Lemma \ref{lem2.12} again, we have $s_{i}\rightarrow 0$.

Denote by $(\tilde{X}_{i}, \tilde{p}_{i}, \tilde{d}_{i}, \tilde{m}_{i}):=(X_{i}, p_{i}, \frac{1}{s_{i}}d_{i}, \frac{1}{m_{i}(B_{s_{i}}(p_{i}))}m_{i})$.
Up to passing to a subsequence, we may assume $(\tilde{X}_{i}, \tilde{p}_{i}, \tilde{d}_{i}, \tilde{m}_{i})$ converges in the pointed measured Gromov-Hausdorff sense to an $\RCD(0,N)$ space $(\tilde{X}_{\infty}, \tilde{p}_{\infty}, \tilde{d}_{\infty}, \tilde{m}_{\infty})$.
Since $B_{r}(p_{i})$ is $(k,\delta_{i})$-Euclidean but not $(\eta,k+1)$-Euclidean for every $r\in [r_{i},\delta_{i}^{-1}]$, $\tilde{X}_{\infty}$ is $k$-splitting and $B_{r}(\tilde{p}_{\infty})$ is not $(k+1,\frac{\eta}{2})$-Euclidean for every $r\geq 1$.

Define $v_{i}=s_{i}^{-1}T_{p_{i},s_{i}}u_{i}$, then
\begin{align}\label{4.5555}
\bbint_{B_{1}(\tilde{p}_{i})}\langle\nabla v_{i}^{a},\nabla v_{i}^{b}\rangle=\delta^{ab}.
\end{align}
By Lemma \ref{lem-Ts-growth-estimate}, (\ref{4.5555}) and assumption $(ii)$, for every $r\in[1,\frac{1}{s_{i}}]$, we have
\begin{align}
&\bbint_{B_{r}(\tilde{p}_{i})}|\nabla v_{i}^{a}|^{2} d \tilde{m}_{i}= \bbint_{B_{rs_{i}}({p}_{i})}|\nabla (T_{p_{i},s_{i}}u_{i})|^{2} d {m}_{i} \\
=&  \bbint_{B_{rs_{i}}({p}_{i})}|T_{p_{i},s_{i}}\circ T_{p_{i},rs_{i}}^{-1}(\nabla (T_{p_{i},rs_{i}}u_{i}))|^{2} d {m}_{i}\nonumber
\\ \leq & C r^{C\epsilon}\bbint_{B_{rs_{i}}({p}_{i})}|\nabla (T_{p_{i},rs_{i}}u_{i})|^{2} d {m}_{i} = C r^{C\epsilon}.\nonumber
\end{align}
Here and in the following, $C$ denotes a constant depending only on $N$ but it may change in different lines.
By the mean value inequality on $\RCD(-(N-1)\delta_{i},N)$ spaces (see Lemma 3.6 in \cite{MN19}), we have
\begin{align}\label{4.7}
|\nabla v_{i}(x)|\leq C\tilde{d}_{i}(x,\tilde{p}_{i})^{C\epsilon}+C
\end{align}
for every $x\in B_{\frac{1}{2s_{i}}}(\tilde{p}_{i})$.
Thus by Theorems \ref{AA} and \ref{2.7777777}, up to passing to a subsequence, $v_{i}$ converges locally uniformly and locally $W^{1,2}$ to a harmonic map $v_{\infty}=(v_{\infty}^{1},\ldots ,v_{\infty}^{k}):\tilde{X}_{\infty}\to \R^k$ with $v_{\infty}(\tilde{p}_{\infty})=0^k$ and
\begin{align}
|\nabla v_{\infty}(\tilde{x})|\leq C\tilde{d}_{\infty}(\tilde{x},\tilde{p}_{\infty})^{C\epsilon}+C.
\end{align}

If $\epsilon$ is small as in Theorem \ref{thm-gap-har-RCD}, then every $v_{\infty}^{a}$ is actually linear with respect to the $\mathbb{R}$-factors in $\tilde{X}_{\infty}$.
Moreover, by the $W^{1,2}_{\mathrm{loc}}$-convergence, we have
\begin{align}
\bbint_{B_{1}(\tilde{p}_{\infty})}\langle\nabla v_{\infty}^{a},\nabla v_{\infty}^{b}\rangle=\delta^{ab}
\end{align}
for every $1\leq a, b\leq k$.
Hence $v_{\infty}^{1},\ldots ,v_{\infty}^{k}$ forms a basis of linear functions on $\mathbb{R}^{k}$.
Without loss of generality we assume $v_{\infty}=(x^{1},\ldots,x^{k})$ are the standard coordinates.
By the $W^{1,2}_{\mathrm{loc}}$-convergence again, we have
\begin{align}\label{4.14}
&\lim_{i\rightarrow\infty}4\bbint_{B_{4}(\tilde{p}_{i})}|\langle\nabla v_{i}^{a},\nabla v_{i}^{b}\rangle-\delta^{ab}|\\
=&\lim_{i\rightarrow\infty}\bbint_{B_{4}(\tilde{p}_{i})}\bigl||\nabla (v_{i}^{a}+ v_{i}^{b})|^{2}-|\nabla (v_{i}^{a}- v_{i}^{b})|^{2}-4\delta^{ab}\bigr|\nonumber\\
=&\bbint_{B_{4}(\tilde{p}_{\infty})}\bigl||\nabla (v_{\infty}^{a}+ v_{\infty}^{b})|^{2}-|\nabla (v_{\infty}^{a}- v_{\infty}^{b})|^{2}-4\delta^{ab}\bigr|=0. \nonumber
\end{align}
Thus $v_{i}$ is a $(\epsilon_{i},k)$-splitting map on $B_{1}(\tilde{p}_{i})$ with $\epsilon_{i}\rightarrow0$.
Hence for every $\frac{1}{10}\leq r \leq 1$ and every sufficiently large $i$, there exists a lower triangle matrix $A_{r,i}$ with positive diagonal entries such that $|A_{r,i}-\mathrm{Id}|\leq C \epsilon_{i}$,
\begin{align}
\bbint_{B_{r}(\tilde{p}_{i})}\langle\nabla (A_{r,i}v_{i})^{a},\nabla (A_{r,i}v_{i})^{b}\rangle=\delta^{ab},
\end{align}
and $A_{r,i}v_{i}: B_{r}(\tilde{p}_{i})\rightarrow \mathbb{R}^{k}$ is $(\frac{\epsilon}{10},k)$-splitting for every $r\in[\frac{1}{10}, 1]$, which contradicts to the assumptions $\mathbf{(i)}$-$\mathbf{(iii)}$.
The proof is completed.
\end{proof}

\begin{cor}\label{cor-splitting-stable}
Given any $N \in \mathbb{Z}^{+}$, $\epsilon>0$ and $\eta>0$, there exists $\delta_{0}=\delta(N, \epsilon, \eta)$ such that, for every $\delta\in(0,\delta_{0})$ the following holds.
Let $(X, d, m)$ be an $\mathrm{RCD}(-(N-1)\delta,N)$ space.
Suppose there is some $s\in(0,1)$ such that, for any $r\in[s,1]$, $B_{r}(p)$ is $(\delta, k)$-Euclidean but not $(\eta, k+1)$-Euclidean,
and let $u : (B_{1}(p),p)\rightarrow (\mathbb{R}^{k_{1}},0^{k_{1}})$ be a $(\delta, k_{1})$-splitting map (where $k_{1}\leq k\leq N$),
then for every $r\in[s,1]$, there exists a $k_{1}\times k_{1}$ lower triangle matrix $T_{r}$ with positive diagonal entries such that
\begin{description}
  \item[(i)] $T_{r}u : B_{r}(p)\rightarrow\mathbb{R}^{k_{1}}$ is a $(\epsilon,k_{1})$-splitting map;
  \item[(ii)] $\bbint_{B_{r}(p)}\langle\nabla(T_{r}u)^{a},\nabla(T_{r}u)^{b}\rangle=\delta^{ab}$;
  \item[(iii)] $|T_{r}\circ T_{2r}^{-1}-\mathrm{Id}|\leq\epsilon$;
  \item[(iv)] for any $t, r \in [s,1]$,
\begin{align}
| T_{r}\circ T_{t}^{-1}|\leq (1+C\epsilon) \max\{\bigl(\frac{t}{r}\bigr)^{C\epsilon}, \bigl(\frac{r}{t}\bigr)^{C\epsilon}\}
\end{align}
holds for a constant $C=C(N) > 1$.
\end{description}

\end{cor}
\begin{proof}
Suppose the conclusion does not hold, then we have contradicting sequences of $\RCD$ spaces and harmonic maps on them.
Then by rescaling them with a slowly blow up rate, we have the followings:
\begin{description}
  \item[(i)] there exist $k_{1}\leq k\leq N$, $\eta>0$, $\epsilon\ll1$, $\delta_{i}\downarrow0$, $r_{i}>0$ and $\mathrm{RCD}(-(N-1)\delta_{i},N)$ space $(X_{i}, p_{i}, d_{i}, m_{i})$ such that $B_{r}(p_{i})$ is $(\delta_{i},k)$-Euclidean but not $(\eta,k+1)$-Euclidean for every $r\in [r_{i},\delta_{i}^{-1}]$;
  \item[(ii)] there exist $(\delta_{i},k_{1})$-splitting maps $u_{i}:(B_{2}(p_{i}),p_i)\rightarrow (\mathbb{R}^{k_{1}},0^{k_{1}})$ and there exists $s_{i} > r_{i}$ such that for every $r\in [s_{i},1]$, there exists a $k_{1}\times k_{1}$ lower triangle matrix $T_{p_{i},r}$ with positive diagonal entries such that $T_{p_{i},r}u_{i}$ is a $(\epsilon,k_{1})$-splitting on $B_{r}(p_{i})$ with $\bbint_{B_{r}(p_{i})}\langle\nabla(T_{p_{i},r}u_{i})^{a},\nabla(T_{p_{i},r}u_{i})^{b}\rangle=\delta^{ab}$;
  \item[(iii)] no such matrix $T_{i}=T_{p_{i},\frac{s_{i}}{10}}$ exists on $B_{\frac{s_{i}}{10}}(p_{i})$.
\end{description}

By the above assumptions and Lemma \ref{lem2.12}, we have $s_{i}\rightarrow 0$.
Up to a subsequence, we assume $(\tilde{X}_{i}, \tilde{p}_{i}, \tilde{d}_{i}, \tilde{m}_{i}):=(X_{i}, p_{i}, \frac{1}{s_{i}}d_{i}, \frac{1}{m_{i}(B_{s_{i}}(p_{i}))}m_{i})\xrightarrow{pmGH}(\tilde{X}_{\infty}, \tilde{p}_{\infty}, \tilde{d}_{\infty}, \tilde{m}_{\infty})$, which is a $k$-splitting $\RCD(0,N)$ space and $B_{r}(\tilde{p}_{\infty})$ is not $(\frac{\eta}{2},k+1)$-Euclidean for every $r\geq 1$.
Define $v_{i}=s_{i}^{-1}T_{p_{i},s_{i}}u_{i}$, then
\begin{align}\label{4.55559988}
\bbint_{B_{1}(\tilde{p}_{i})}\langle\nabla v_{i}^{a},\nabla v_{i}^{b}\rangle=\delta^{ab}.
\end{align}
Following the argument before (\ref{4.7}), we derive that
\begin{align}
|\nabla v_{i}(x)|\leq C\tilde{d}_{i}(x,\tilde{p}_{i})^{C\epsilon}+C
\end{align}
for every $x\in B_{\frac{1}{2s_{i}}}(\tilde{p}_{i})$.
Thus, up to passing to a subsequence, $v_{i}$ converges locally uniformly and locally $W^{1,2}$ to a harmonic map $v_{\infty}=(v_{\infty}^{1},\ldots ,v_{\infty}^{k_{1}}):\tilde{X}_{\infty}\to \mathbb{R}^{k_{1}}$ with $v_{\infty}(\tilde{p}_{\infty})=0^{k_{1}}$,
\begin{align}
\bbint_{B_{1}(\tilde{p}_{\infty})}\langle\nabla v_{\infty}^{a},\nabla v_{\infty}^{b}\rangle=\delta^{ab},
\end{align}
\begin{align}
|\nabla v_{\infty}(\tilde{x})|\leq C\tilde{d}_{\infty}(\tilde{x},\tilde{p}_{\infty})^{C\epsilon}+C.
\end{align}
Thus if $\epsilon$ is small as in Theorem \ref{thm-gap-har-RCD}, then $v_{\infty}^{1},\ldots ,v_{\infty}^{k_{1}}$ forms standard coordinates on $\mathbb{R}^{k_{1}}\subset \tilde{X}_{\infty}=\mathbb{R}^{k} \times \tilde{Y}$.
Note that $\tilde{Y}$ dose not splits off any $\mathbb{R}$-factor.

We choose linear harmonic functions $x^{k_{1}+1},\ldots, x^{k}$ on $\tilde{X}_{\infty}$ such that $v_{\infty}^{1},\ldots ,v_{\infty}^{k_{1}},x^{k_{1}+1},\ldots, x^{k}$ forms standard coordinates on $\mathbb{R}^{k}$.
In particular,
\begin{align}
\bbint_{B_{1}(\tilde{p}_{\infty})}\langle\nabla x^{a},\nabla x^{b}\rangle=\delta^{ab},
\end{align}
\begin{align}
\bbint_{B_{1}(\tilde{p}_{\infty})}\langle\nabla v_{\infty}^{c},\nabla x^{a}\rangle=\delta^{ca},
\end{align}
for any $a,b\in\{k_{1}+1,\ldots,k\}$, $c\in\{1,\ldots,k_{1}\}$.
Then we apply Theorem \ref{thm-transplant} to transplant $x^{k_{1}+1},\ldots, x^{k}$ to harmonic functions $w^{a}_{i}:B_{2} (\tilde{p}_{i})\rightarrow\mathbb{R}$ with $w^{a}_{i}(\tilde{p}_{i})=0$ such that $w^{a}_{i}$ converges locally uniformly and locally $W^{1,2}$ to $x^{a}$.

Thus similar to (\ref{4.14}), by the $W^{1,2}$-convergence, $w_{i}:=(v_{i}^{1},\ldots ,v_{i}^{k_{1}},w^{k_{1}+1}_{i},\ldots, w^{k}_{i})$ is an $(\epsilon_{i},k)$-splitting map on $B_{1}(\tilde{p}_{i})$ with $\epsilon_{i}\rightarrow0$.
Hence for sufficiently large $i$, there exists $k\times k$ lower triangle matrix $A_{i}$ with positive diagonal entries such that $|A_{i}-\mathrm{Id}|\leq C \epsilon_{i}$ and
\begin{align}
\bbint_{B_{\frac{1}{10}}(\tilde{p}_{i})}\langle\nabla (A_{i}w_{i})^{a},\nabla (A_{i}w_{i})^{b}\rangle=\delta^{ab},
\end{align}
and $A_{i}w_{i}: B_{\frac{1}{10}}(\tilde{p}_{i})\rightarrow \mathbb{R}^{k}$ is $(\frac{\epsilon}{10},k)$-splitting.
Finally, we take $T_{i}$ to be the top left $k_{1}\times k_{1}$ sub-matrix of $A_{i}$ to obtain a contradiction.
The proof is completed.
\end{proof}

\subsection{Proof of Theorem \ref{TransformationThmUderSplittingMonotonicity}}\label{sec-4.2}

In this subsection, we will prove Theorem \ref{thm-splitting-unstable}, which is equivalent to Theorem \ref{TransformationThmUderSplittingMonotonicity} by Remark \ref{rem4.2}.

\begin{thm}\label{thm-splitting-unstable}
Given $\epsilon>0$, $n, k\in \mathbb{Z}^{+}$ with $k\leq n$, and a function $\Phi:(0,1)\to \R^+$ with $\lim_{\delta\to0}\Phi(\delta)=0$, there exists $\delta_{0}>0$ depending on $n$, $\epsilon$ and $\Phi$ such that the following holds for every $\delta\in(0,\delta_{0})$.
Suppose $(M,g)$ is an $n$-dimensional Riemannian manifold with $\Ric\geq -(n-1)\delta$ and
\begin{equation}\label{4.8888888888888}
\Theta^{(M,g)}_{p,k'}(s)\le \Phi(\max\{\delta,\theta^{(M,g)}_{p,k'}(s)\})
\end{equation}
holds for every $s\in(0,1)$ and every integer $k\leq k'\leq n$.
Suppose in addition $B_{2}(p)$ is $(\delta,k)$-Euclidean, and $u : (B_{2}(p),p)\rightarrow (\mathbb{R}^{k_{1}},0^{k_{1}})$ is a $(\delta,k_{1})$-splitting map (with $k_{1}\leq k\leq n$), then for any $r\in(0,1]$, there exists a $k_{1}\times k_{1}$ lower triangle matrix $T_{r}$ with positive diagonal entries such that $T_{r}u : B_{r}(p)\rightarrow\mathbb{R}^{k_{1}}$ is an $(\epsilon,k_{1})$-splitting map.
\end{thm}

\begin{rem}\label{rem4.2}
By the same reasons as in Remark \ref{rem4.1}, up to composing $T_{r}$ by another lower triangle matrix with positive diagonal entries, we may assume the $T_{r}$'s in Theorems \ref{TransformationThmUderSplittingMonotonicity} and \ref{thm-splitting-unstable} satisfy the following properties:
\begin{description}
  \item[(1)] $T_{r}u : B_{r}(p)\rightarrow\mathbb{R}^{k}$ is a $(\epsilon,k)$-splitting map;
  \item[(2)] $\bbint_{B_{r}(p)}\langle\nabla(T_{r}u)^{a},\nabla(T_{r}u)^{b}\rangle=\delta^{ab}$;
  \item[(3)] $|T_{r}\circ T_{2r}^{-1}-\mathrm{Id}|\leq\epsilon$;
  \item[(4)] for any $t, r \in (0,1]$,
\begin{align}
| T_{r}\circ T_{t}^{-1}|\leq (1+C\epsilon) \max\{\bigl(\frac{t}{r}\bigr)^{C\epsilon}, \bigl(\frac{r}{t}\bigr)^{C\epsilon}\}
\end{align}
holds for a constant $C=C(n) > 1$.
\end{description}
In addition, note that if there exists a $(\delta,k)$-splitting map $u : B_{2}(p)\rightarrow \mathbb{R}^{k}$ as in the statement of Theorem \ref{TransformationThmUderSplittingMonotonicity}, then by Theorem \ref{thm-split-GHisom},
$B_{1}(p)$ is $(\Psi(\delta|n),k)$-Euclidean, thus the conclusion of Theorem \ref{TransformationThmUderSplittingMonotonicity} follows from Theorem \ref{thm-splitting-unstable}.
On the other hand, following the argument in Corollary \ref{cor-splitting-stable}, one can check that Theorem \ref{TransformationThmUderSplittingMonotonicity} implies Theorem \ref{thm-splitting-unstable}.
We will prove Theorem \ref{thm-splitting-unstable} because we can apply Theorem \ref{thm-splitting-stable} and Corollary \ref{cor-splitting-stable} more directly.

\end{rem}

\begin{proof}[Proof of Theorem \ref{thm-splitting-unstable}]

The proof is by a reverse induction over $k$.
For the case $k = n$, the conclusion follows from Corollary \ref{cor-splitting-stable}.
In the following we assume the conclusion holds for every $k\in \{\bar{k}+1,\ldots,n\}$.

Suppose the conclusion does not hold for $k =\bar{k}$.
Thus there exist $k_{1}\in\mathbb{Z}^{+}$ with $k_{1}\leq k$, and $\epsilon_{0} > 0$ sufficiently small, $\delta_{i}\downarrow 0$, and a sequence of $n$-dimensional manifolds $(M_{i},g_{i},p_{i})$ with $\Ric\geq -(n-1)\delta_{i}$, and a sequence of $(\delta_{i},k_{1})$-splitting maps $u_{i}: (B_{2}(p_{i}),p_i)\rightarrow(\mathbb{R}^{k_{1}},0^{k_1})$ such that
$B_{2}(p_{i})$ is $(\delta_{i},k)$-Euclidean and (\ref{4.8888888888888}) holds for every $p_{i}$, every $s\in(0,1)$ and every integer $k\leq k'\leq n$.
In addition, for each $i$, there exists $r_{i}>0$, such that there exists no $k_{1}\times k_{1}$ lower triangle matrix $T_{r_{i}}$ with positive diagonal entries
such that $T_{r_{i}}u_{i}: B_{r_{i}}(p_{i})\rightarrow\mathbb{R}^{k_{1}}$ is an $(\epsilon_{0},k_{1})$-splitting map.
By these assumptions, it is easy to see that $r_{i}\rightarrow 0$.

Passing to a subsequence, we may assume $(M_{i},g_{i},p_{i})$ converges to an $\RCD(0,n)$ space $(X,x,d,m)$ in the pointed measured Gromov-Hausdorff distance.
Note that $B_{2}(x)$ is $k$-Euclidean.

If $B_{2}(x)$ is $(k+1)$-Euclidean, then by the induction assumption, we obtain a contradiction.
Thus in the following we assume there exists $\eta_{0}>0$ such that $d_{GH}(B_{1}(x), B_{1}((0^{k+1},z)))\geq \eta_{0}$ for any $B_{1}((0^{k+1},z))\subset\mathbb{R}^{k+1}\times Z$, with $Z$ being a metric space and $z\in Z$.

For $\eta\in (0,\eta_{0})$ which will be determined later, define
\begin{align}
s_{i}:= \inf\{s\in(0,1]\mid B_{r}(p_{i}) \text{ is not } (\eta,k+1)\text{-Euclidean, for any } r\in(s,1]\}.
\end{align}

Since (\ref{4.8888888888888}) holds and $B_{2}(p_{i})$ is $(\delta_{i},k)$-Euclidean, by definition, we know $B_{r}(p_{i})$ is $(\Phi(\delta_{i}),k)$-Euclidean for every $r\in(0,1]$.
Then by Corollary \ref{cor-splitting-stable}, for any $r\in(s_{i},1]$, there exists a $k_{1}\times k_{1}$ lower triangle matrix $T_{i,r}$ with positive diagonal entries such that
$T_{i,r}u_{i} : B_{r}(p_{i})\rightarrow\mathbb{R}^{k_{1}}$ is a $(\Psi(\Phi(\delta_{i})|n,\eta),k_{1})$-splitting map.
Suppose there is a subsequence of $i$ such that $s_{i}<r_{i}$, then we have obtained a contradiction.
Thus in the following we assume $s_{i}\geq r_{i}$ for every $i$.

By the choice of $s_{i}$ and Corollary \ref{cor-splitting-stable}, there exists a $k_{1}\times k_{1}$ lower triangle matrix $A_{i}$ with positive diagonal entries such that $A_{i}u_{i}:B_{2s_{i}}(p_{i})\rightarrow \mathbb{R}^{k_{1}}$ is a $(\Psi(\Phi(\delta_{i})| n,\eta),k_{1})$-splitting map,
hence $A_{i}u_{i}: B_{s_{i}}(p_{i})\rightarrow \mathbb{R}^{k_{1}}$ is $(\Psi(\Phi(\delta_{i})| n,\eta),k_{1})$-splitting.
Also by the definition of $s_{i}$, $B_{s_{i}}(p_{i})$ is $(\eta,k+1)$-Euclidean.
Thus by the induction assumption, we can find a $\eta(\frac{\epsilon_{0}}{2},n)>0$ such that, if $\eta<\min\{\eta(\frac{\epsilon_{0}}{2},n),\eta_{0}\}$ is fixed, then there exists a sufficiently large $i$ (in order to assure that $\delta_{i}$ is sufficiently small) such that, for every $r\in(0,s_{i}]$ there exists a $k_{1}\times k_{1}$ lower triangle matrix $S_{i,r}$ with positive diagonal entries so that $S_{i,r}(A_{i}u_{i}):B_{r}(p_{i})\rightarrow \mathbb{R}^{k_{1}}$ is a  $(\frac{\epsilon_{0}}{2},k_{1})$-splitting.
This contradicts to $s_{i}\geq r_{i}$ and we have finished the proof of $k=\bar{k}$ case.

The proof is completed.
\end{proof}

\subsection{Non-degeneration of $(\delta,k)$-splitting map on manifolds}\label{sec-4.3}

In this subsection, we prove Theorem \ref{NonDegeneracyofSplittingMaps-RicCase}.
Note that by Proposition \ref{prop-smp-1}, Theorem \ref{NonDegeneracyofSplittingMaps} is a special case of Theorem \ref{NonDegeneracyofSplittingMaps-RicCase}.

\begin{proof}[Proof of Theorem \ref{NonDegeneracyofSplittingMaps-RicCase}]
We first prove the non-degeneration of $(\delta,k)$-splitting maps.
Suppose the non-degeneration does not hold.
Then there exist $k\in \mathbb{Z}^{+}\cap[1,n]$ and a sequence $\delta_{i}\downarrow 0$,
a sequence of $n$-manifolds $(M_{i},g_{i},p_{i})$ with $\mathrm{Ric}_{g_{i}}\geq -\delta_{i}$ and $B_{1}(p_{i})$ satisfying the $(\Phi; k,\delta_{i})$-generalized Reifenberg condition,
and there exists a sequence of $(\delta_{i},k)$-splitting map $u_{i} : B_{3}(p_{i})\rightarrow\mathbb{R}^{k}$, and  $z_{i}\in B_{1}(p_{i})$ such that $du_{i} : T_{z_{i}} M\rightarrow\mathbb{R}^{k}$ has rank $\leq k-1$.
Without loss of generality, we assume $u_{i}(z_{i})=0^{k}$.

For every $i$, we find a sufficiently small $r_{i}$, with $r_{i}\downarrow0$, such that $(\tilde{M}_{i},\tilde{g}_{i},\tilde{z}_{i}):=(M_{i},r_{i}^{-2}g_{i},z_{i})$ converges to $(\mathbb{R}^{n},g_{\mathrm{Eucl}},0^{n})$ in the $C^{\infty}$-Cheeger-Gromov sense.
Furthermore, by the theory of harmonic radius, we may assume $r_{i}$ is sufficiently small such that there exists a map $\psi_{i}=(\tilde{x}_{i}^{1},\ldots,\tilde{x}^{n}_{i}) : B_{\frac{1}{\sqrt{r_{i}}}}(\tilde{z}_{i})\rightarrow \mathbb{R}^{n}$ with the following properties:
\begin{description}
  \item[(1)] $\Delta \tilde{x}_{i}^{a}=0$ for $a=1,\ldots, n$;
  \item[(2)] $\psi_{i}$ is a diffeomorphism onto its image with $B_{\frac{1}{\sqrt{r_{i}}}}(0^{n})\subset \psi_{i}(B_{\frac{1}{\sqrt{r_{i}}}}(\tilde{z}_{i}))$, and hence defines a coordinate chart;
  \item[(3)] the coordinate metric $g^{(i)}_{ab}= \langle\nabla \tilde{x}_{i}^{a},\nabla \tilde{x}_{i}^{b}\rangle$ on $B_{\frac{1}{\sqrt{r_{i}}}}(\tilde{z}_{i})$ satisfies $|g^{(i)}_{ab}-\delta_{ab}|_{C^{0}(B_{\frac{1}{\sqrt{r_{i}}}}(\tilde{z}_{i}))} +\frac{1}{\sqrt{r_{i}}}\cdot\sum_{c=1}^{n}|\frac{\partial}{\partial \tilde{x}_{i}^{c}} g_{ab}^{(i)}|_{C^{0}(B_{\frac{1}{\sqrt{r_{i}}}}(\tilde{z}_{i}))} < \epsilon_{i}$, with $\epsilon_{i}\downarrow0$.
\end{description}
We also assume $\tilde{x}_{i}^{a}$ converges to $x^{a}$ smoothly, with $x^{1},\ldots,x^{n}$ forming a standard coordinate in $\mathbb{R}^{n}$.

Applying Theorem \ref{TransformationThmUderSplittingMonotonicity} to the $(\delta_{i},k)$-splitting map $u_{i}: B_{2}(z_{i})\rightarrow\mathbb{R}^{k}$,
we can find a $k\times k$ lower triangle matrix $T^{(i)}$ with positive diagonal entries so that $T^{(i)}u_{i} : B_{100r_{i}}(z_{i})\rightarrow\mathbb{R}^{k}$ is $(\Psi(\delta_{i}|n),k)$-splitting.
Hence up to passing to a subsequence, $\tilde{v}_{i}:=\frac{1}{r_{i}}T^{(i)}u_{i}: B_{100}(\tilde{z}_{i})\rightarrow\mathbb{R}^{k}$ converges locally uniformly and locally $W^{1,2}$ to a map $\tilde{v}_{\infty}=(\tilde{v}_{\infty}^{1},\ldots,\tilde{v}_{\infty}^{k}):B_{100}(0^{n})(\subset\mathbb{R}^{n})\rightarrow \mathbb{R}^{k}$ with $\Delta \tilde{v}_{\infty}^{a}=0$, $\langle\nabla \tilde{v}_{\infty}^{a},\nabla \tilde{v}_{\infty}^{b}\rangle\equiv\delta_{ab}$ for every $a,b\in\{1,\ldots,k\}$.
Without loss of generality, we may assume $\tilde{v}_{\infty}^{a}=x^{a}$ for every $a\in\{1,\ldots,k\}$.

Recall that under the harmonic coordinate $\psi_{i}$, for every $l\in\{1,\ldots,k\}$, we have
\begin{align}
0=\Delta (\tilde{v}_{i}^{l}-\tilde{x}_{i}^{l}) =\sum_{a,b=1}^{n}g^{(i),ab}\frac{\partial^{2}(\tilde{v}_{i}^{l}-\tilde{x}_{i}^{l})}{\partial \tilde{x}^{a}_{i}\partial \tilde{x}^{b}_{i}}
\end{align}
on $B_{100}(\tilde{z}_{i})$.
By standard Schauder estimates, we know that for any fixed $\alpha\in(0,1)$, for sufficiently large $i$, $\tilde{v}_{i}^{l}-\tilde{x}_{i}^{l}$ has a uniform (i.e. independent of $i$) $C^{1,\alpha}$ bound on ${B}_{90}(\tilde{z}_{i})$.
Thus the convergence of $\tilde{v}_{i}^{l}-\tilde{x}_{i}^{l}$ to $0$ is in uniformly $C^{1,\alpha}$ sense for any $\alpha\in(0,1)$.
In particular,
\begin{align}
|\langle \nabla \tilde{v}_{i}^{a},\nabla \tilde{v}_{i}^{b} \rangle-\delta_{ab}|\rightarrow 0
\end{align}
uniformly in ${B}_{90}(\tilde{z}_{i})$.
In particular, $\mathrm{rank}(d\tilde{v}_{i}|_{\tilde{z}_{i}})= k$.
Thus for sufficiently large $i$, we have $\mathrm{rank}(du_{i}|_{z_{i}})= k$, which gives a contradiction.
This finishes the proof of the non-degeneracy of almost-splitting maps.
	
Now we go to verify the inequality (\ref{BiHolder}).

Given any $x,y\in B_{\frac12}(p)$, without loss of generality, we assume $u(x)\neq u(y)$.
Take $z\in u^{-1}(u(y))$ so that $d(x,z)=d(x,u^{-1}(u(y)))$.
Then $r:=d(x,z)\le1$.
Let $P=\frac{u(z)-u(x)}{|u(z)-u(x)|}$, then it is easy to check that the function $\tilde{u}: B_{\frac{5}{2}}(x)\rightarrow \R$ defined by $\tilde{u}(w)=\langle P,u(w)-u(x)\rangle$ is $\Psi(\delta|n)$-splitting.
In particular, $|\nabla \tilde{u}|\le 1+\Psi(\delta|n)$, and hence
\begin{align}
&|u(x)-u(y)|=|u(x)-u(z)|=|\tilde{u}(z)-\tilde{u}(x)|
\\ \le &(1+\Psi(\delta|n))d(x,z)=(1+\Psi(\delta|n))d(x,u^{-1}(u(y))).\nonumber
\end{align}
For the other side of (\ref{BiHolder}), without loss of generality, we assume $u(x)=0^k$.
By Theorem \ref{TransformationThmUderSplittingMonotonicity}, there exists a lower triangle matrix $T_r$ with positive diagonal entries and $|T_r|\le (1+\Psi(\delta|n,\Phi))r^{-\Psi(\delta|n,\Phi)}$, such that $T_ru:B_r(x)\to\R^k$ is a $(\Psi(\delta|n,\Phi),k)$-splitting map.
Then we have
\begin{align}
&(1+\Psi(\delta|n,\Phi))r^{-\Psi(\delta|n,\Phi)-1}|u(x)-u(z)|\ge r^{-1}|T_ru(x)-T_ru(z)|\\
\ge & r^{-1}d(x,u^{-1}(u(z)))-\Psi(\delta|n,\Phi),\nonumber
\end{align}
where we use Lemma \ref{Surjectivity20220411} in the second inequality.
Thus we obtain
\begin{align}
&|u(x)-u(y)|\ge (1-\Psi(\delta|n,\Phi))r^{1+\Psi(\delta|n,\Phi)} \\ =&(1-\Psi(\delta|n,\Phi))d(x,u^{-1}(u(y)))^{1+\Psi(\delta|n,\Phi)}.\nonumber
\end{align}

The proof is completed.
\end{proof}

In the remaining of this subsection, we prove Lemma \ref{Surjectivity20220411}.

\begin{lem}\label{20220411}
Suppose $(M,g,p)$ is an $n$-dimensional manifold with $\Ric\geq -(n-1)\delta$ and $u:(B_3(p),p)\to(\R^k,0^k)$ is a non-degenerate $\delta$-splitting map which gives a $\delta$-Gromov-Hausdorff approximation from $B_3(p)$ to its image,
then $\overline{B_1(0^k)}\subset u(B_{1+\Psi(\delta|n)}(p))$.
\end{lem}
\begin{proof}
Since $u$ is a $\delta$-Gromov-Hausdorff approximation, we have
$d(x,p)\leq |{u}(x)|+\Psi(\delta|n)$  for every $x\in B_{3}(p)$.
In particular, ${u}^{-1}(\overline{B_{1}(0^{k})})\subset B_{1+\Psi(\delta|n)}(p)$, and it is easy to check that $u(u^{-1}(\overline{B_{1}(0^{k})}))$ is closed in $\overline{B_{1}(0^{k})}$.
On the other hand, by the non-degeneracy of $d{u}$, $u(u^{-1}(\overline{B_{1}(0^{k})}))$ is open in $\overline{B_{1}(0^{k})}$.
Thus $u(u^{-1}(\overline{B_{1}(0^{k})}))=\overline{B_{1}(0^{k})}$ and hence $\overline{B_1(0^k)}\subset u(B_{1+\Psi(\delta|n)}(p))$.
\end{proof}

\begin{lem}\label{lem4.9}
Suppose $(M,g,p)$ is an $n$-dimensional manifold with $\Ric\geq -(n-1)\delta$ and $u:(B_8(p),p)\to(\R^k,0^k)$ is a $(\delta,k)$-splitting map (with $1\leq k\leq n$).
Suppose in addition that $B_{4}(p)$ satisfies the $(\Phi; k,\delta)$-generalized Reifenberg condition for some positive function $\Phi(\cdot)$ with $\lim_{\delta'\to0^+}\Phi(\delta')=0$.
Then $\overline{B_1(0^k)}\subset u(B_{1+\Psi(\delta|n,\Phi)}(p))$.
\end{lem}
\begin{proof}
We argue by a contradiction.
Suppose there exist $\epsilon_{0}>0$ and a sequence of $\delta_{i}\downarrow0$ and a sequence of $n$-manifolds $(M_i,g_i,p_i)$ with $\Ric_{M_i}\ge-(n-1)\delta_i$ and $B_{4}(p_i)$ satisfying the $(\Phi; k,\delta_{i})$-generalized Reifenberg condition, and there exist $(\delta_i,k)$-splitting maps  $u_i:(B_8(p_i),p_i)\to(\R^k,0^k)$, and there are $a_i\in \overline{B_1(0^k)}$ but $a_i\notin u_i(B_{1+\epsilon_0}(p_i))$.
Up to a subsequence, we assume $a_i\rightarrow a_\infty\in \overline{B_1(0^k)}$.
Note that by the first part of Theorem \ref{NonDegeneracyofSplittingMaps-RicCase}, $u_{i}$ is non-degenerate on $B_{3}(p_{i})$.

By Cheeger-Colding's theory (see \cite{CC97} etc.), up to a subsequence, we may assume
$$(B_4(p_i),p_i,d_i,m_{i})\xrightarrow{pmGH}( B_4((0^k,x_\infty)),(0^k,x_\infty),d=d_{\mathrm{Eucl}}\times d_{X}, \mathcal{L}^k\otimes m_X),$$
where $m_{i}=\frac{1}{\vol_{i} (B_1(p_i))}\vol_i$ is the renormalized volume measure on $M_{i}$, $B_4((0^k,x_\infty))\subset \R^{k}\times X$ for some pointed $\RCD(0,n-k)$ space $(X, x_{\infty}, d_{X}, m_{X})$ (here we assume $X$ consists of more than one point, for otherwise by Lemma \ref{20220411} we will get a contradiction directly; the fact that $(X, d_{X}, m_{X})$ is an $\RCD(0,n-k)$ space is due to \cite{Gig13}); and $u_i$ converges locally uniformly and locally $W^{1,2}_{\mathrm{loc}}$ to the projective map $u_{\infty}:B_4((0^k,x_\infty))\to \R^{k}$.

By Cheeger-Colding's theory (see also \cite{MN19} etc.), there exists a regular point $y_\infty\in X$, such that $d_{X}(y_\infty,x_\infty)\le \frac{1}{10}\epsilon_0$.
By definition, there is some $k'> k$ such that every tangent cone at $(a_\infty,y_\infty)$ is $\R^{k'}$.
In addition, we have $d((a_\infty,y_\infty),(0^k,x_\infty))\le \sqrt{1+\frac{1}{100}\epsilon_0^2}<1+\frac{1}{10}\epsilon_0$.

Now we choose $\hat{x}_i\in B_{1+\frac{1}{5}\epsilon_0}(p_i)$ such that $\hat{x}_i\xrightarrow{GH}(a_\infty,y_\infty)$.

Let $f_{i}=\sum_{a,b=1}^{k}|\langle\nabla u_{i}^{a}, \nabla u_{i}^{b}\rangle-\delta^{ab}|+|\mathrm{Hess} u_{i}|^{2}$, and consider the maximal function $Mf_{i}:B_{3}(p_i)\to \R$ defined by
$Mf_{i}(w)=\sup_{0<r\leq\frac{1}{4}}\bbint_{B_{r}(w)}f_{i}$.
By the weak $(1,1)$ type inequality for maximal function operator on doubling metric measure spaces (see e.g. Theorem 2.2 in \cite{He01}), and the Bishop-Gromov volume comparison on $M_{i}$, we have
\begin{align}\label{4.31}
\frac{1}{\vol_{i} (B_1(p_i))}\mathrm{vol}_{i}(\{w|Mf_{i}(w)>\delta_{i}^{\frac{1}{2}}\})\leq \frac{C_{1}(n)}{\delta_{i}^{\frac{1}{2}}}\frac{1}{\vol_{i} (B_1(p_i))}\int_{B_3(p_i)}f d\mathrm{vol}_{i}\leq C_{2}(n)\delta_{i}^{\frac{1}{2}}.
\end{align}
Combining (\ref{4.31}) with the volume comparison, we can always find $x_i$ such that $d_{i}(x_i, \hat{x}_i)\leq C_{3}(n)\delta_{i}^{\frac{1}{2n}}$ and $Mf_{i}(x_{i})\leq\delta_{i}^{\frac{1}{2}}$.

Therefore, for each sufficiently large $i$, $x_i\in B_{1+\frac{2}{5}\epsilon_0}(p_i)$ is chosen so that
\begin{description}
  \item[(a)] $x_i\xrightarrow{GH} (a_\infty,y_\infty)$;
  \item[(b)] for any $r\in(0,\frac{1}{4})$, $u_i:B_r(x_i)\to\R^k$ is a $\Psi(\delta_i|n)$-splitting map.
\end{description}
Obviously, we have $|u_i(x_i)-a_i|\rightarrow0$.

Then we choose a sequence of $r_{i}\downarrow 0$ such that
\begin{description}
  \item[(c)] $r_i\geq |u_i(x_i)-a_i|$;
  \item[(d)] $(M_i,\tilde{g}_i,x_i):=(M_i,r_i^{-2}g_i,x_i)$ converges in the pointed measured Gromov-Hausdorff distance to the tangent cone of $\R^k\times X$ at $(a_\infty,y_\infty)$, i.e. $(\R^{k'},0^{k'},d_{\mathrm{Eucl}},c\mathcal{L}^{k'})$.
\end{description}
By \textbf{(b)} and \textbf{(d)}, $\tilde u_i:=r_i^{-1}(u_i-u_i(x_i)):B^{\tilde{g}_{i}}_{\frac{1}{4r_i}}(x_i)\rightarrow \R^{k}$ converges locally uniformly to a projection map $\tilde u_\infty: (\R^{k'},0^{k'})\to (\R^k,0^k)$.

By adding the other $k'-k$ coordinates, we get an isometry $u=(\tilde u_{\infty},\tilde u_\infty^{k+1},\ldots,\tilde u_\infty^{k'}):(\R^{k'},0^{k'})\to (\R^{k'},0^{k'})$.
By Theorem \ref{thm-transplant}, we transplant $ \tilde u_\infty^{k+1},\ldots,\tilde u_\infty^{k'}$ back onto $(M_i, \tilde{g}_i, x_i)$ to obtain harmonic functions $\tilde u_i^{k+1},\ldots,\tilde u_i^{k'}:(B^{\tilde{g}_{i}}_{16}(x_i),x_i)\to(\R,0)$, so that
$\tilde{v}_i:=(\tilde u_i,\tilde u_i^{k+1},\ldots,\tilde u_i^{k'}):B^{\tilde{g}_i}_{16}(x_i)\to\R^{k'}$ is a $(\delta'_{i},k')$-splitting map giving a $\delta'$-Gromov-Hausdorff approximation from $B^{\tilde{g}_i}_8(x_i)$ onto its image, where $\delta'_{i}\downarrow 0$.

Furthermore, it is easy to see that $B^{\tilde{g}_{i}}_{16}(x_i)$ still satisfies $(\Phi; k,\delta_{i})$-generalized Reifenberg condition.
Then according to the first part of Theorem \ref{NonDegeneracyofSplittingMaps-RicCase}, for every sufficiently large $i$,
$d\tilde{v}_i: T_{w}M_{i}\rightarrow\mathbb{R}^{k'}$ is non-degenerate for every $w\in B^{\tilde{g}_i}_{4}(x_i)$.

Note that by \textbf{(c)}, $(r_i^{-1}(a_i-u_i(x_i)),0^{k'-k})\in B_{\frac{3}{2}}(0^{k'})$.
Then according to Lemma \ref{20220411}, for every sufficiently large $i$, there exists $y_i\in B^{\tilde{g}_i}_{2}(x_i)$ such that $\tilde{v}_i( y_i)=(r_i^{-1}(a_i-u_i(x_i)),0^{k'-k})$.
In particular, we have $u_i(y_i)=a_i$ and $d_{i}(y_i,x_i) <2r_i$, and hence $y_i\in B_{1+\epsilon_0}(p_i)$, which contradicts to the assumptions on $a_i$.
This completes the proof.
\end{proof}

\begin{lem}\label{Surjectivity20220411}
Suppose $(M,g,p)$ is an $n$-dimensional manifold with $\Ric\geq -(n-1)\delta$ and $u:(B_8(p),p)\to(\R^k,0^k)$ is a $(\delta,k)$-splitting map (with $1\leq k\leq n$).
Suppose in addition that $B_{4}(p)$ satisfies the $(\Phi; k,\delta)$-generalized Reifenberg condition for some positive function $\Phi(\cdot)$ with $\lim_{\delta'\to0^+}\Phi(\delta')=0$.
Then for any $x\in B_1(p)$, it holds
\begin{align}
|d(p,u^{-1}(u(x)))-|u(p)-u(x)||\le\Psi(\delta|n,\Phi).\nonumber
\end{align}
\end{lem}
\begin{proof}
We argue by a contradiction. Suppose there exist a sequence of positive numbers $\delta_{i}\to 0$ and a sequence of $n$-manifolds $M_{i}$ with $\Ric_{M_{i}}\ge-(n-1)\delta_{i}$ and $B_{4}(p_{i})\subset M_{i}$ satisfying the $(\Phi; k,\delta_{i})$-generalized Reifenberg condition, and there exists a sequence of $\delta_i$-splitting map $u_i:B_8(p_i)\to\R^k$, and there exist $\epsilon_0>0$ and $x_i\in B_1(p_i)$ so that
\begin{align}\label{6.18}
|d_i(p_i,u_i^{-1}(u_i(x_i)))-|u_i(p_i)-u_i(x_i)||>\epsilon_{0}.
\end{align}

By Cheeger-Colding's theory, up to a subsequence, we may assume $(B_8(p_i),p_i,d_i,m_{i})\xrightarrow{pmGH}(B_8(p_\infty),p_\infty, d_\infty,m_{\infty})$,
where $m_{i}=\frac{1}{\vol_{i} (B_1(p_i))}\vol_i$ is the renormalized volume measure on $M_{i}$, and  $B_8(p_\infty)\subset \R^{k}\times X$ for some pointed $\RCD(0,n-k)$ space $(X, q_{\infty}, d_{X}, m_{X})$, and $p_\infty=(0^k,q_\infty)$.
In addition, we assume $u_{i}$ converges in the locally uniformly and $W^{1,2}$ sense to a projection $u_{\infty}: B_8(p_\infty)\to \R^k$.
In particular, for any $w\in B_4(p_\infty)$, it holds
$$d_{\infty}(p_\infty,u_\infty^{-1}(u_\infty(w)))=|u_\infty(p_\infty)-u_\infty(w)|.$$

For each $i$, take $y_i\in u_i^{-1}(u_i(x_i))$ such that $d_{i}(p_i,y_i)=d_{i}(p_i,u_i^{-1}(u_i(x_i)))\le 1$.
Assuming $x_{i}\myarrow{GH} x_{\infty}$, $y_i\myarrow{GH} y_\infty$, then we have $x_{\infty},y_{\infty}\in \overline{B_1(p_\infty)}$, $u_i(x_i)=u_i(y_i)\to u_\infty(x_\infty)=u_\infty(y_\infty)$, and
\begin{align}\label{4.30}
&d_{i}(p_i,u_i^{-1}(u_i(x_i)))-|u_i(p_i)-u_i(x_i)| = d_{i}(p_i,y_i)-|u_i(p_i)-u_i(y_i)| \\
\geq &d_{\infty}(p_\infty,y_\infty)-|u_\infty(p_\infty)-u_\infty(y_\infty)|-\epsilon_{i}\ge -\epsilon_{i}\nonumber
\end{align}
for a sequence of positive numbers $\epsilon_{i}\downarrow0$.

Take $z_\infty\in u_\infty^{-1}(u_\infty(x_\infty))$ satisfying $$d_{\infty}(p_\infty,z_\infty)=d_{\infty}(p_\infty,u_\infty^{-1}(u_\infty(x_\infty))))\le 1.$$
Choose $B_{\frac{11}{10}}(p_i)\ni z_i\to z_\infty$.
By a maximal function argument as in the proof of Lemma \ref{lem4.9}, we may further assume $z_i$ is chosen so that $u_i:B_{r}(z_i)\to\R^k$ is $\Psi(\delta_i|n)$-splitting for any $r\in(0,1)$.
We claim that
	\begin{equation}\label{20220411'}
		d_{i}(z_i,u_i^{-1}(u_i(x_i)))\to0.
	\end{equation}

Assuming (\ref{20220411'}) is true, we have
\begin{align}\label{4.32}
		&d_{i}(p_i,u_i^{-1}(u_i(x_i)))-|u_i(p_i)-u_i(x_i)|\le d_{i}(p_i,z_i)-|u_i(p_i)-u_i(x_i)|+\eta_i\\
\leq&d_{\infty}(p_\infty,z_\infty)-|u_\infty(p_\infty)-u_\infty(x_\infty)|+2\eta_i=2\eta_i\nonumber
\end{align}
for a sequence of positive numbers $\eta_{i}\downarrow0$.
(\ref{4.30}) and (\ref{4.32}) contradict to (\ref{6.18}), and this finishes the proof of Lemma \ref{Surjectivity20220411}.

Now we just need to verify (\ref{20220411'}).
Take $r_i=|u_i(z_i)-u_i(x_i)|$, then $r_i\to |u_\infty(z_\infty)-u_\infty(x_\infty)|=0$.
We consider $({M}_{i},\tilde{g}_{i}):=(M_{i},r_i^{-2} g_{i})$, and $\Psi(\delta_i|n)$-splitting maps $\tilde{u}_i:B^{\tilde{g}_i}_{8}(z_i)\rightarrow \R^{k}$ defined by $\tilde{u}_i(w)=r_i^{-1}(u_i(w)-u_i(z_i))$.
Note that $B^{\tilde{g}_{i}}_8(z_i)$ satisfies $(\Phi; k,\delta_{i})$-generalized Reifenberg condition, and $|\tilde u_i(x_i)-\tilde u_i(z_i)|=1$, then according to Lemma \ref{lem4.9}, for sufficiently large $i$, there exists $x_i'$ with $\tilde{d}_{i}(z_i,x_i')<2$ such that $\tilde u_i( x_i')=\tilde u_i(x_i)$.
Thus $x_i'\in u_i^{-1}(u_i(x_i))$, and hence
$d_{i}(z_i,u_i^{-1}(u_i(x_i)))\le d_i(z_i,x_i')< 2r_i\to0$.
The proof is completed.
\end{proof}

\section{The generalized Reifenberg condition}\label{sec-5}

The following proposition considers generalized Reifenberg condition under Gromov-Hausdorff convergence. The proof follows easily from the definitions.

\begin{prop}\label{prop-smp-10}
Given a non-decreasing positive function $\Phi(\cdot)$ with $\lim_{\theta\to0^+}\Phi(\theta)=0$, and then define another function $\tilde{\Phi}(\theta)=2{\Phi}(2\theta)$.
Let $(X_{i}, p_{i}, d_{i})$ be a sequence of metric spaces such that $(X_{i}, p_{i}, d_{i})$ converges in the pointed Gromov-Hausdorff sense to $(X_{0}, p_{0}, d_{0})$.
Suppose there exist an integer $k\geq 1$ and sequences of numbers $\delta_{i} \geq 0$, $r_{i}> 0$ with  $\delta_{i} \rightarrow \delta_{0}$ and $r_{i} \rightarrow r_{0}>0$, so that $p_{i}$ satisfies the $(\Phi,r_{i}; k,\delta_{i})$-generalized Reifenberg condition.
Then $p_{0}$ satisfies the $(\tilde{\Phi},r_{0}; k,\delta_{0})$-generalized Reifenberg condition.
\end{prop}

Now we consider generalized Reifenberg condition under local diffeomorphisms.

\begin{prop}\label{prop-smp-2}
Given any $n\in\N^+$ and a positive function $\Phi_{0}(\cdot)$ with $\lim_{\delta'\to0^+}\Phi_0(\delta')=0$, there exists a positive function $\Phi_{1}(\cdot)$, depending only on $n$ and $\Phi_0$, so that $\lim_{\delta'\to0^+}\Phi_{1}(\delta')=0$ and the following holds.
Let $(M,g,p)$ be an $n$-dimensional (not necessarily complete) Riemannian manifold with $\Ric_{M}\ge-(n-1)\delta$ on $B_{8}(p)$ and $B_{8}(p)$ has compact closure in $B_{9}(p)$.
Suppose $\hat M$ is a connected $n$-manifold and there is a surjective local diffeomorphism $\sigma:\hat M\to B_{8}(p)$ with $\sigma(\hat p)=p$.
We equip $\hat M$ with the pull-back metric $\hat{g}=\sigma^{*}g$ and the induced distance $\hat{d}$.
If for some $1\leq k \leq n$, $\hat{p}$ satisfies the $(\Phi_{0}; k,\delta)$-generalized Reifenberg condition, then ${p}$ satisfies the $(\Phi_{1}; k,\delta)$-generalized Reifenberg condition.
\end{prop}

In order to prove Proposition \ref{prop-smp-2}, we need the following lemma, which generalizes the covering lemma in \cite{KW}.
We leave the proof of Lemma \ref{covering-lem} in the appendix.
\begin{lem}\label{covering-lem}
There exists $C(n)\ge1$ such that the following holds.
Suppose $(M,g,p)$ is a not necessarily complete $n$-manifold with $\Ric_{M}\ge-(n-1)$ on $B_{8}(p)$, and $B_{8}(p)$ has compact closure in $B_{10}(p)$.
Suppose $\hat M$ is a connected $n$-manifold and there is a surjective local diffeomorphism $\sigma:\hat M\to B_{10}(p)$ with $\sigma(\hat p)=p$.
We equip $\hat M$ with the pull-back metric $\hat{g}=\sigma^{*}g$ and the induced distance $\hat{d}$.
Then given any non-negative function $f:B_3(p)\to \R$,
\begin{equation}\label{ControlBetweenCover}
	\frac{1}{C(n)}\bbint_{B_r(\hat p)}f\circ\sigma\le \bbint_{B_r(p)}f\le C(n)\bbint_{B_{3r}(\hat p)}f\circ\sigma
\end{equation}
holds for any $r\le 1$.
\end{lem}

\begin{proof}[Proof of Proposition \ref{prop-smp-2}]
Suppose $\delta>0$ is sufficiently small, and $\Ric_{M}\ge-(n-1)\delta$.
Fix $k\leq k'\leq n$.
For any $0<s\leq1$, if $\delta':=\theta^{(M,g)}_{p,k'}(s)\geq\delta$, then $B_s(p)$ is $(\delta',k')$-Euclidean. We assume $\delta'$ is also sufficiently small.
By Lemma \ref{lem2.11}, for any $r\in[\frac16s\sqrt{\delta'},s]$, $B_r(p)$ is $(\frac{1}{6}\sqrt{\delta'},k')$-Euclidean.
On the other hand, there exists a positive constant $\delta_{CC}(n)$ (depending only on $n$) which comes from Cheeger-Colding's almost splitting theorem (\cite{CC97}), so that if $\delta'\le\delta_{CC}(n)$, then there exists a $(\Psi_{0}(\delta'|n),k')$-splitting map $u:B_{10s\sqrt{\delta'}}(p)\to\R^{k'}$.
In particular,
\begin{align}
\bbint_{B_{10s\sqrt{\delta'}}(p)}\sum_{i,j=1}^{k'}|\langle \nabla u^{i}, \nabla u^{j}\rangle-\delta^{ij}|+s^{2}\delta'|\mathrm{Hess} u|^{2}\leq\Psi_0(\delta'|n).
\end{align}

Denote by $\hat{u}:=u\circ\sigma:B_{10s\sqrt{\delta'}}(\hat p)\to\R^{k'}$, which is harmonic.
Applying the first inequality of (\ref{ControlBetweenCover}) to $f=\sum_{i,j=1}^{k'}|\langle \nabla u^{i}, \nabla u^{j}\rangle-\delta^{ij}|+s^{2}\delta'|\mathrm{Hess} u|^{2}$, and then by the Bishop-Gromov volume comparison, we have
\begin{align}
\bbint_{B_{s\sqrt{\delta'}}(\hat{p})}\sum_{i,j=1}^{k'}|\langle \nabla \hat{u}^{i}, \nabla \hat{u}^{j}\rangle-\delta^{ij}|+s^{2}\delta'|\mathrm{Hess} \hat{u}|^{2}\leq C(n)\Psi_0(\delta'|n),
\end{align}
i.e. $\hat{u}:B_{s\sqrt{\delta'}}(\hat p)\to\R^{k'}$ is a $(C(n)\Psi_0({\delta'}|n),k')$-splitting map.

Since $\hat p$ satisfies the $(\Phi_{0}; k',\delta)$-generalized Reifenberg condition,
by Theorem \ref{TransformationThmUderSplittingMonotonicity}, for any $\delta'\le\delta(n,\Phi_0)$ sufficiently small and any $r\in(0,\frac{1}{6}s\sqrt{\delta'}]$,
there exists a $k'\times k'$ lower triangle matrix $T_{6r}$ with positive diagonal entries such that $\hat{w}=T_{6r}\circ \hat{u}:B_{6r}(\hat p)\to\R^{k'}$ is a $(\Psi_{1}(\delta'|n,\Phi_0),k')$-splitting map.
Denote by ${w}=T_{6r}\circ {u}$.
Applying the second inequality of (\ref{ControlBetweenCover}) to $f=\sum_{i,j=1}^{k'}|\langle \nabla {w}^{i}, \nabla {w}^{j}\rangle-\delta^{ij}|+36r^{2}|\mathrm{Hess} {w}|^{2}$,
we obtain that $w: B_{2r}(p)\to\R^{k'}$ is a $(C(n)\Psi_{1}(\delta'|n,\Phi_0),k')$-splitting map.
Then by Cheeger-Colding's theory, $B_{r}(p)$ is $(\Psi_2(\delta'|n,\Phi_0),k')$-Euclidean.

If $\theta^{(M,g)}_{p,k'}(s)<\delta$ for some $0<s\leq1$, then $B_s(p)$ is $(\delta,k')$-Euclidean, thus we reduce the problem to the $\delta'=\delta$ case.

Define the positive function $\Phi_{1}$ by
\begin{align}
\Phi_{1}(\delta')=\left\{
  \begin{array}{ll}
    \max\{\Psi_2(\delta'|n,\Phi_0),\frac{1}{6}\sqrt{\delta'}\}, & \hbox{if $\delta'\in(0,\delta(n,\Phi_0)]$;} \\
    2, & \hbox{if $\delta'>\delta(n,\Phi_0)$,}
  \end{array}
\right.
\end{align}
then $\Phi_{1}$ satisfies the required properties.
The proof is completed.
\end{proof}

In the following, we consider sufficient conditions for the generalized Reifenberg condition.

\begin{prop}\label{prop-smp-1}
For any $n\in\mathbb{Z}^{+}$, there exists a non-decreasing positive function $\Phi(\cdot)$, depending only on $n$, so that $\lim_{\delta'\to0^+}\Phi(\delta')=0$ and the following holds.
For every $\delta>0$, suppose $(X,d)\in \mathrm{Alex}^{n}(-\delta)$, then
\begin{equation}
\Theta^{(X,d)}_{p,k'}(s)\le \Phi(\max\{\delta,\theta^{(X,d)}_{p,k'}(s)\})
\end{equation}
holds for every $s\in(0,1]$ and every integer $1\leq k'\leq n$.
\end{prop}

Proposition \ref{prop-smp-1} follows directly from Theorem \ref{thm-LN20}.

\begin{prop}\label{prop-smp-3}
For any $n\in\mathbb{Z}^{+}$, there exists a non-decreasing positive function $\Phi(\cdot)$, depending only on $n$, so that $\lim_{\delta'\to0^+}\Phi(\delta')=0$ and the following holds.
\begin{description}
  \item[(1)] If $(X,d,\mathcal{H}^{n})$ is a non-collapsed $\RCD(-(n-1)\delta,n)$ space, then $(X,d)$ satisfies the $(\Phi; n,\delta)$-generalized Reifenberg condition.
  \item[(2)] If $(X,d,\mathcal{H}^{n})$ is a non-collapsed $\RCD(-(n-1)\delta,n)$ space and  $d_{GH}(B_1(p),B_1(0^n))\le\delta$, then $p$ satisfies the $(\Phi; 1,\delta)$-generalized Reifenberg condition.
\end{description}
\end{prop}

\begin{proof}
By the proof of Theorem 3.1 of \cite{KM20}, we have:
if $(X,d,\mathcal{H}^{n})$ is a non-collapsed $\RCD(-(n-1)\delta,n)$ space and $d_{GH}(B_{1}(p),B_{1}(0^n))\le\delta$, then $d_{GH}(B_{s}(p),B_{s}(0^n))\le \Psi(\delta|n)s$ for every $0<s\leq 1$.
Note that such kind of property is known to hold on $n$-manifolds with $\Ric \geq -\delta$ by Cheeger and Colding in \cite{CC97}.

It is easy to see that this property implies (1) and (2).
\end{proof}

\begin{prop}\label{cor-smp-3}
There exists a non-decreasing positive function $\Phi$, depending on certain geometric quantities, so that $\lim_{\delta'\to0^+}\Phi(\delta')=0$ and the following holds.

Given an $n$-manifold $(M, g)$ with $\Ric\ge-(n-1)\delta$ and $p\in M$.
Suppose in addition one of the following conditions holds:
	\begin{itemize}
		\item [(i)] there exists $r_0\in(0,1]$ such that every $q\in B_{r_{0}}(p)$ has conjugate radius at least $r_{0}$;

		\item [(ii)] there exists $r_0\in(0,1]$ such that $d_{GH}(B_{r_0}(\tilde p),B_{r_0}(0^n))\le\delta r_0$, where $(\widetilde{B_{r_0}(p)},\tilde p)\to(B_{r_0}(p),p)$ is the universal cover;
		
		\item [(iii)]  $d_{GH}(B_{1}(p),B_{1}(0^k))\le\delta$, $\vol(B_{1}(\tilde p))\ge v$, where $(\widetilde{B_{1}(p)},\tilde p)\to(B_{1}(p),p)$ is the universal cover;
		
		\item [(iv)] $d_{GH}(B_1(p),B_1(0^k))\le\delta$, and the nilpotency rank $\mathrm{rank}\Gamma_{\delta}(p)=n-k$, where $\Gamma_{\delta}(p)$ is the image of the standard homomorphism $\pi_1(B_\delta(p))\to\pi_1(B_1(p))$,
		
	\end{itemize}
then
\begin{equation}
\Theta^{(M,g)}_{p,k'}(s)\le \Phi(\max\{\delta,\theta^{(M,g)}_{p,k'}(s)\})
\end{equation}
holds for every $s\in(0,1)$ and every integer $k'\geq 1$.
More precisely, in cases (i) and (ii), $\Phi$ depends only on $n$ and $r_{0}$; in case (iii), $\Phi$ depends only on $n$ and $v$; in case (iv), $\Phi$ depends only on $n$.
\end{prop}
\begin{proof}
\emph{Case (i): }

By assumptions, the exponential map $\exp_{p}:(T_{p}M \supset)B_{r_{0}}(\hat{p})\mapsto B_{r_{0}}({p})$ is a surjective local diffeomorphism, where $\hat{p}$ denotes the origin in $T_{p}M$.
We equip $\hat{M}:=B_{r_{0}}(\hat{p})$ with the pull-back metric $\hat{g}=\exp_{p}^{*}g$.
Then for any $\hat{q}\in B_{\frac{r_{0}}{2}}(\hat{p})$, the injectivity radius at $\hat{q}$, denoted by $\mathrm{inj}_{\hat{q}}$, is at least $\frac{r_{0}}{2}$ (see \cite{Xu18} for a proof).

\textbf{Claim:} Given $n\in\mathbb{Z}^{+}$ and $r_{0}>0$. For any $\epsilon >0$, there exists a positive constant $\tilde{\delta}_{0}$ depending only on $n$, $r_{0}$ and $\epsilon$ so that, every $\delta' \in (0,\tilde{\delta}_{0})$ satisfies the followings.
Let $(\hat{M},\hat{g})$ be an (not necessarily complete) $n$-manifold with $\Ric_{\hat{M}}\ge-(n-1)\delta'$, and for any $r'<r_{0}$, $B_{r'}(\hat{p})$ has compact closure in $\hat{M}$, and $\mathrm{inj}_{\hat{q}}\geq \frac{r_{0}}{2}$ for any $\hat{q}\in B_{\frac{r_{0}}{2}}(\hat{p})$.
Suppose $B_{r}(\hat{p})$ is $(\delta',k)$-Euclidean for some $r\in (0,r_{0}]$, then for every $s\in (0,r]$, $B_{s}(\hat{p})$ is $(\epsilon,k)$-Euclidean.

Suppose the claim does not hold.
Then there exist $\epsilon_{0}>0$, and some integer $1\leq k\leq n$, and a sequence of positive numbers ${\delta}_{i}\downarrow 0$, a sequence of $n$-manifolds $(\hat{M}_{i},\hat{g}_{i}, \hat{p}_{i})$ with $\Ric_{\hat{M}_{i}}\ge-(n-1)\delta_{i}$, $B_{1}(\hat{p}_{i})$ has compact closure in $\hat{M}_{i}$, and $\mathrm{inj}_{\hat{q}}\geq \frac{r_{0}}{2}$ for every $\hat{q}\in B_{\frac{r_{0}}{2}}(\hat{p}_{i})$,
and for every $i$ there exists $0< s_{i}< r_{i}< r_{0}$ such that $B_{r_{i}}(\hat{p}_{i})$ is $(\delta_{i},k)$-Euclidean, but $B_{s_{i}}(\hat{p}_{i})$ is not $(\epsilon_{0},k)$-Euclidean.
Note that by Lemma \ref{lem2.11}, we have $s_{i}\rightarrow0$.
Denote by $(\tilde{M}_{i},\tilde{g}_{i}, \tilde{p}_{i})=(\hat{M}_{i},s_{i}^{-2}\hat{g}_{i}, \hat{p}_{i})$, then $\Ric_{\tilde{M}_{i}}\ge-(n-1)s_{i}^{2}\delta_{i}$, and for any $R>0$ and $\tilde{q}_{i}\in B_{R}(\hat{p}_{i})$, we have $\mathrm{inj}_{\tilde{q}_{i}}\geq \frac{1}{s_{i}}r_{0}\rightarrow \infty$.
Note that $B_{1}(\tilde{p}_{i})$ is not $(\epsilon_{0},k)$-Euclidean.
By \cite{AC92}, $(\tilde{M}_{i},\tilde{g}_{i}, \tilde{p}_{i})$ converges in the $C^{\beta}$-topology to $(\mathbb{R}^{n}, g_{\mathrm{Eucl}}, 0^{n})$ for any $\beta\in(0,1)$.
Thus $B_{1}(\tilde{p}_{i})$ is $(\epsilon_{0},n)$-Euclidean for sufficiently large $i$, which is a contradiction.
Hence we complete the proof of the claim.

By the above claim, it is easy to construct a non-decreasing positive function $\Phi_{0}$, depending only on $n$ and $r_{0}$, with $\lim_{\delta'\to0^+}\Phi_{0}(\delta')=0$, so that under the assumptions of case (i), $\hat{p}$ satisfies the $(\Phi_{0}; 1,\delta)$-generalized Reifenberg condition.
Then by Proposition \ref{prop-smp-2}, we can complete the proof of case (i).

\emph{Case (ii):}

By Proposition \ref{prop-smp-3} and the assumptions, there exists a non-decreasing positive function $\Phi_{0}$, depending only on $n$ and $r_{0}$, so that $\lim_{\delta'\to0^+}\Phi_{0}(\delta')=0$ and
$\tilde p$ satisfies the $(\Phi_{0}; 1,\delta)$-generalized Reifenberg condition.
Combining this fact with Proposition \ref{prop-smp-2}, we can finish the proof of case (ii).

\emph{Case (iii):}

We first prove the following claim.

\textbf{Claim:} For any $\epsilon >0$, there exists $\delta_{0}>0$ depending only on $n$, $v$ and $\epsilon$, such that for any $\eta\in(0, \delta_{0}]$, if $\Ric \geq -(n-1)\eta$, $d_{GH}(B_{1}(p),B_{1}(0^k))\le\eta$ and $\vol(B_{1}(\tilde p))\ge v$, where $(\widetilde{B_{1}(p)},\tilde p)\to(B_{1}(p),p)$ is the universal cover, then
$d_{GH}(B_{\eta^{\frac{1}{2}}}(\tilde p),B_{\eta^{\frac{1}{2}}}(0^n))\le \eta^{\frac{1}{2}}\epsilon $.

This claim will be proved by an argument by contradiction basing on Lemma 2.2 of \cite{Hua}.
Suppose there exist $\epsilon_{0}>0$, $\delta_{i}\downarrow 0$ and a sequence of $n$-dimensional Riemannian manifolds with $\Ric_{i} \geq -(n-1)\delta_{i}$, $\vol(B_{1}(\tilde p_{i}))\ge v$, and $d_{GH}(B_{1}(p_{i}),B_{1}(0^k))\le\delta_{i}$, but $d_{GH}(B_{\delta_{i}^{\frac{1}{2}}}(\tilde p_{i}),B_{\delta_{i}^{\frac{1}{2}}}(0^n)) > \delta_{i}^{\frac{1}{2}}\epsilon_{0}$.
We denote by $M_{i}=B_{1}(p_{i})$, equipped with a metric $g_{i}$, and denote by $\tilde{M}_{i}=\widetilde{B_{1}(p_{i})}$, equipped with the pull-back metric $\tilde{g}_{i}$.
Up to a subsequence, we assume $(\tilde{M}_{i}, \tilde{p}_{i})$ converges in the point Gromov-Hausdorff sense to some $(\tilde{Y}, \tilde{y}_{\infty}, d_{\tilde{Y}})$, hence there exists a sequence of positive numbers $\xi_{i}\downarrow 0$ such that $d_{GH}(B_{1}(\tilde{p}_{i}), B_{1}(\tilde{y}_{\infty}))\leq \xi_{i}$.
Since $\vol(B_{1}(\tilde p_{i}))\ge v$, by \cite{CC96}, we know every tangent cone of $\tilde{Y}$ at $\tilde{y}_{\infty}$ is a metric cone.

Suppose there is a subsequence of $\{i\}$, still denoted by $\{i\}$, such that $\xi_{i}\leq \delta_{i}$, then we consider $(X_{i},h_{i},q_{i})=(M_{i},\delta_{i}^{-1}g_{i},p_{i})$ and $(\widetilde{X}_{i},\tilde{h}_{i},\tilde{q}_{i})=(\tilde{M}_{i},\delta_{i}^{-1}\tilde{g}_{i},\tilde{p}_{i})$.
Note that $\Ric_{h_{i}} \geq -(n-1)\delta_{i}^{2}$, $d_{GH}(X_{i},B_{\delta_{i}^{-\frac{1}{2}}}(0^k))\le\delta_{i}^{\frac{1}{2}}$, and $d_{GH}(B_{1}(\tilde q_{i}),B_{1}(0^n)) > \epsilon_{0}$.
Up to a subsequence, we assume $(\tilde{Y}, \delta_{i}^{-\frac{1}{2}} d_{\tilde{Y}}, \tilde{y}_{\infty})$ converges in the Gromov-Hausdorff sense to some metric cone $(\tilde{Z}=\R^{m}\times C(Z), d_{\tilde{Z}}, \tilde{z}_{\infty})$, where $m\leq n$.
It is not hard to see that $(\tilde{X}_{i}, \tilde{h}_{i},\tilde{q}_{i})\xrightarrow{pGH}(\tilde{Z}, d_{\tilde{Z}}, \tilde{z}_{\infty})$.
We consider the equivariant Gromov-Hausdorff convergence (see \cite{FY92}) in the following commutative diagram, where $\Gamma_{i}:=\pi_{1}(X_{i})$, and $f_{i}:\tilde{X}_{i} \to X_{i}$ is the natural projection:
\begin{align}
\xymatrix{
  (\tilde{X}_{i}, \tilde{q}_{i}, \Gamma_{i}) \ar[d]_{f_{i}} \ar[r]^{GH} & (\tilde{Z}, \tilde{z}_{\infty}, G) \ar[d]^{f_{\infty}} \\
  ({X}_{i}, q_{i}) \ar[r]^{GH} & (\R^{k},0^{k})  }
\end{align}
By \cite{FY92}, we have $\R^{k}= (\R^{m}\times C(Z))/G$.
But one can prove that this happens unless $\tilde{Z}=\R^{m}\times C(Z)=\R^{n}$, see Lemma 2.2 of \cite{Hua} for details.
Then we have $d_{GH}(B_{1}(\tilde q_{i}),B_{1}(0^n)) \rightarrow 0$, which is a contradiction.

Suppose there is a subsequence of $\{i\}$, still denoted by $\{i\}$, such that $\xi_{i}> \delta_{i}$, then we consider $(X^{(2)}_{i},h_{i}^{(2)},q_{i}^{(2)})=(M_{i},\xi_{i}^{-1}g_{i},p_{i})$ and $(\widetilde{X}_{i}^{(2)},\tilde{h}_{i}^{(2)},\tilde{q}_{i}^{(2)})=(\tilde{M}_{i},\xi_{i}^{-1}\tilde{g}_{i}, \tilde{p}_{i})$.
Note that $\Ric_{h_{i}^{(2)}} \geq -(n-1)\xi_{i}^{2}$, $d_{GH}(X_{i}^{(2)},B_{\xi_{i}^{-\frac{1}{2}}}(0^k))\le\xi_{i}^{\frac{1}{2}}$.
Up to a subsequence, we assume $(\tilde{Y}, \xi_{i}^{-\frac{1}{2}} d_{\tilde{Y}}, \tilde{y}_{\infty})$ converges in the Gromov-Hausdorff sense to some metric cone $(\tilde{Z}^{(2)}, d_{\tilde{Z}^{(2)}}, \tilde{z}_{\infty}^{(2)})$.
Then $(\tilde{X}_{i}^{(2)}, \tilde{h}_{i}^{(2)},\tilde{q}_{i}^{(2)})\xrightarrow{pGH}(\tilde{Z}^{(2)}, d_{\tilde{Z}^{(2)}}, \tilde{z}_{\infty}^{(2)})$, and by Lemma 2.2 of \cite{Hua} together with the equivariant Gromov-Hausdorff convergence
\begin{align}
\xymatrix{
  (\tilde{X}_{i}^{(2)}, \tilde{q}_{i}^{(2)}, \Gamma_{i}) \ar[d]_{} \ar[r]^{GH} & (\tilde{Z}^{(2)}, \tilde{z}_{\infty}^{(2)}, G^{(2)}) \ar[d]^{} \\
  ({X}_{i}^{(2)}, q_{i}^{(2)}) \ar[r]^{GH} & (\R^{k},0^{k})  }
\end{align}
we have $\tilde{Z}^{(2)}=\R^{n}$, and hence $d_{GH}(B_{1}(\tilde q_{i}^{(2)}),B_{1}(0^n)):=\zeta_{i}\rightarrow 0$.
Since $\Ric_{\tilde{h}_{i}^{(2)}} \geq -(n-1)\xi_{i}^{2}$ and $\xi_{i}> \delta_{i}$, by Cheeger-Colding's theory, we have
$\frac{\sqrt{\xi_{i}}}{\sqrt{\delta_{i}}} d_{GH}(B_{\frac{\sqrt{\delta_{i}}}{\sqrt{\xi_{i}}}}(\tilde q_{i}^{(2)}),B_{\frac{\sqrt{\delta_{i}}}{\sqrt{\xi_{i}}}}(0^n))\leq \Psi(\xi_{i},\zeta_{i}|n)\rightarrow0$.
Equivalently, $\delta_{i}^{-\frac{1}{2}}d_{GH}(B_{\delta_{i}^{\frac{1}{2}}}(\tilde p_{i}),B_{\delta_{i}^{\frac{1}{2}}}(0^n))=\Psi(\xi_{i},\zeta_{i}|n)\rightarrow 0$, which is a contradiction.
Thus we have complete the proof of the claim.

By the above claim, it is not hard to find a non-decreasing positive function $\Phi_{0}$, depending only on $n$ and $v$, with $\lim_{\delta'\to0^+}\Phi_{0}(\delta')=0$, so that under the assumptions of case (iii),  $\tilde{p}$ satisfies the $(\Phi_{0}; 1,\delta)$-generalized Reifenberg condition.
Then by Proposition \ref{prop-smp-2}, we can complete the proof of case (iii).

\emph{Case (iv):}

By Theorem 5.1 of \cite{NZ}, there exist $v_{0}, \delta_{0}$ depending only on $n$ so that, if $\delta'\in (0,\delta_{0}]$ and $d_{GH}(B_1(p),B_1(0^k))\le\delta'$ and $\mathrm{rank}\Gamma_{\delta'}(p)=n-k$, where $\Gamma_{\delta'}(p):=\mathrm{Image}(\pi_1(B_{\delta'}(p))\to\pi_1(B_1(p)))$, then
$\vol(B_{\frac{1}{2}}(\tilde p))\ge v_{0}$, where $(\widetilde{B_{\frac{1}{2}}(p)},\tilde p)\to(B_{\frac{1}{2}}(p),p)$ is the universal cover. Then the conclusion of case (iv) follows from case (iii).
\end{proof}

\begin{rem}
There have been many works proving smooth fibration theorems under the conditions listed in Proposition \ref{cor-smp-3}.
More precisely, case (i) in Proposition \ref{cor-smp-3} is handled in \cite{Wei97}, see also \cite{DWY} \cite{NZ} for the case that $|\Ric|\leq(n-1)$ and the conjugate radius is bounded from below.
Cases (ii) and (iii) are handled in \cite{Hua}.
Case (iv) is considered in \cite{HW20} using Ricci flow smoothing.
\end{rem}

\section{Smooth fibration theorems}\label{sec-6}

In this section, we prove the smooth fibration theorem \ref{FiberBundleThm} and a local version of fibration theorem \ref{localFiberBundleThm}.

\begin{proof}[Proof of Theorem \ref{FiberBundleThm}]
	
The construction of $f$ is literally the same as the one in \cite{Hua}, except that Lemma 2.1 of \cite{Hua} is used to verify the non-degeneracy of $f$, while we use Theorem \ref{NonDegeneracyofSplittingMaps-RicCase} instead, so we just give a sketch here.

Scaling the distances by $\delta^{-\frac{1}{2}}$, we may assume $\Ric_M\ge-\delta$, $|\sec_N|\le\delta$, $\inj_N\ge \delta^{-\frac{1}{2}}$, $d_{GH}(M,N)\le\sqrt{\delta}$, and $M$
satisfies the $(\Phi; k,\delta)$-generalized Reifenberg condition.

The first step is to define local maps.
Let $h:M\to N$ be a $\sqrt\delta$-Gromov-Hausdorff approximation.
Fix a $1$-net $\{p_\lambda|\lambda=1,2,\ldots,\Lambda\}$ on $M$. Since $|\sec_N|\le\delta$ and $\inj_N\ge \delta^{-\frac{1}{2}}$, for each $q\in N$, there exists a $C^{1,\alpha}$-harmonic coordinate $\Phi_q:(B_{\delta^{-\frac{1}{4}}}(0^k),0^k)\to(N,q)$ such that $$|\Phi_q^*g_N-g_{\R^k}|_{C^{1,\alpha}(B_{\delta^{-\frac{1}{4}}}(0^k))}\le\Psi(\delta|n).$$
For each $\lambda$, we choose one such coordinate $\Phi_\lambda:=\Phi_{h(p_\lambda)}$.
By the almost splitting theorem,
there exists a family of $(\Psi(\delta|n),k)$-splitting maps $u_\lambda: (B_4(p_\lambda),p_\lambda)\to (\R^k,0^k)$, which is $\Psi(\delta|n)$-close to $\Phi_\lambda^{-1}\circ h|_{B_4(p_\lambda)}$. Then $f_\lambda:=\Phi_\lambda\circ u_\lambda$ are the local maps we need.

The next step is to glue the local maps $\{f_\lambda|\lambda=1,\ldots,\Lambda\}$,  via the center of mass technique, to form a global map $f$.
Let $\phi:[0,\infty)\to[0,1]$ be a smooth function with $\phi|_{[0,\frac{11}{10}]}\equiv1$, and $\supp\phi\subset[0,2]$ and $|\phi'|\leq 10$.
Let $r$ be the distance function to $0^k$ on $\R^k$. Then $\phi_\lambda\triangleq\phi\circ r\circ u_\lambda$ is a smooth function on $M$ with $\supp{\phi_\lambda}\subset B_{\frac{21}{10}}(p_\lambda)$ for small $\delta$.
Define an energy function $E:M\times N\to\R$ by
\begin{align}\label{6.1}
	E(x,y)=\frac 12\sum_{\lambda} \phi_\lambda(x) d^{2}(f_\lambda(x),y).
\end{align}

It is easy to see that, for any fixed $x\in M$, in the summation (\ref{6.1}), only those $\lambda$ satisfying $d(x,p_\lambda)<\frac{25}{10}$ give non-vanishing terms.
By the volume comparison theorem, the cardinality of $\{\lambda|d(x,p_\lambda)<\frac{25}{10}\}$ is bounded from above by some $\Lambda(n)$ depending only on $n$.

By the convex radius estimate, $E(x,\cdot)$ is strictly convex on $B_1(h(x))$, and it is easy to see that $E(x,\cdot)$ achieves the global minimum at a unique point, denoted by $cm(x)$.
We define the map $f: M\to N$ by $f(x)=cm(x)$.
One can check that $f$ is smooth, and is $\Psi(\delta|n)$-close to $h(x)$.

The last step is to verify the non-degeneracy of $f$.
Since $f$ is not harmonic itself, to apply Theorem \ref{NonDegeneracyofSplittingMaps-RicCase}, we need the following lemma which roughly says that at any fixed $p$, $f$ is tangent to an almost splitting map at $p$.
\begin{lem}\label{JacobianEstimate}
Given $p\in M$, and a $C^{1,\alpha}$-harmonic coordinate $\Phi_{f(p)}^{-1}=(y^1,\ldots,y^k)$ centered at $f(p)$, for those $\lambda$ with $d(p,p_\lambda)<\frac{25}{10}$,
there exist constants $C^\alpha_{\lambda,\beta}$, $\alpha,\beta=1,\ldots,k$, and $(\Psi(\delta|n),k)$-splitting maps $v_\lambda:B_{1}(p)\to\R^k$, such that,
\begin{align}\label{JacobianEstimate-1}
\myd f^\alpha(p)=\sum_{\lambda}C^{\alpha}_{\lambda,\beta}\myd v_\lambda^\beta(p).
\end{align}
\begin{align}\label{JacobianEstimate-2}
|v_\lambda-\Phi_{f(p)}^{-1}\circ h|\le\Psi(\delta|n),
\end{align}	
\begin{align}\label{JacobianEstimate-3} |C^\alpha_{\lambda,\beta}-\delta^\alpha_\beta\phi_\lambda(p)D(p)^{-1}|\le\Psi(\delta|n), \text{ where } D(x)=\sum_\lambda\phi_\lambda(x).
\end{align}
\end{lem}

The detailed proof of Lemma \ref{JacobianEstimate-1} is omitted, we only give a sketch here.
The readers can refer to \cite{Hua} for more details.
In the proof, we need two basic ingredients.
One is the following fact, which comes from the basic properties of harmonic coordinate:

Given any $p$, then for each $\lambda$ with $d(p,p_{\lambda}) < \frac{25}{10}$, there exists an isometric map $\omega_{\lambda}: \R^n\to\R^n$ (respect to the standard Euclidean metrics) such that
\begin{align}
|\omega_{\lambda}-\Phi^{-1}_{h(p_{\lambda})}\circ\Phi_{f(p)}|_{C^1(B_{10}(0^n))}\leq \Psi(\delta|n).
\end{align}

Another ingredient is, $y=f(x)$ is the solutions of equations $\frac{\partial}{\partial y^\alpha }E(x,y)=0,\alpha=1,\ldots,k$.
Then by a direct calculation basing on implicit function theorem, one can prove Lemma \ref{JacobianEstimate-1}.

Once Lemma \ref{JacobianEstimate} is proved, for every $p\in M$, we can define a function on $B_{1}(p)$ by
$$v^\alpha(x)=\sum_\lambda C^\alpha_{\lambda,\beta} v_\lambda^\beta(x),\quad\alpha=1,\ldots,k.$$
Combining (\ref{JacobianEstimate-2}), (\ref{JacobianEstimate-3}), we have  $\left|v^{\alpha}(x)-v_{\lambda}^\alpha(x)\right|\le\Psi(\delta|n)$.
Hence $v=(v^1,\ldots,v^k)|_{B_{\frac{1}{2}}(p)}$
is a ($\Psi(\delta|n),k)$-splitting map.
Since $M$ satisfies the $(\Phi; k,\delta)$-generalized Reifenberg condition, by Theorem \ref{NonDegeneracyofSplittingMaps-RicCase}, $dv$ is non-degenerate at $p$, provided $\delta$ is sufficiently small.
By (\ref{JacobianEstimate-1}), $\myd f^\alpha(p)=\myd v^\alpha(p)$.
Thus we have proved the non-degeneracy of $\myd f$ at $p$.
\end{proof}

The following is a local version of the fibration theorem.

\begin{thm}\label{localFiberBundleThm}
Given $\epsilon>0$ and a positive function $\Phi$ with $\lim_{\delta\to0^+}\Phi(\delta)=0$, there exists $\delta_{0}>0$ depending on $n$, $\epsilon$ and $\Phi$ such that the following holds for every $\delta\in(0,\delta_{0})$.
Let $M$ be a not necessarily complete $n$-dimensional manifold such that $B_{3}(p)$ has compact closure in $B_{4}(p)$, and $\Ric_M\ge-(n-1)$ on $B_{3}(p)$, and $B_{2}(p)$ satisfies the $(\Phi; k,\delta)$-generalized Reifenberg condition, and $d_{GH}(B_{3}(p),B_{3}(0^{k}))\leq \delta$ (with $1\leq k\leq n$) holds.
Then there exists a smooth map $u : B_{2}(p)\rightarrow B_{2}(0^{k})$ with $u(p)=0^{k}$ such that $u$ is an $\epsilon$-Gromov-Hausdorff approximation, and $u|_{u^{-1}(B_{1}(0^{k}))}$ is a smooth fibration and $u^{-1}(B_{1}(0^{k}))$ is diffeomorphic to $B_{1}(0^{k})\times F$, where $F$ is a compact manifold of dimension $n-k$.
\end{thm}

\begin{proof}[Sketch of proof:]
Scale the distances by $\frac{1}{\sqrt{\delta}}$, hence we may assume $\Ric_M\ge-\delta$ on $B_{\frac{3}{\sqrt{\delta}}}(p)$, $B_{\frac{2}{\sqrt{\delta}}}(p)$ satisfies the $(\Phi; k,\delta)$-generalized Reifenberg condition, and there exists a $\sqrt{\delta}$-Gromov-Hausdorff approximation  $h:B_{\frac{3}{\sqrt{\delta}}}(p)\rightarrow B_{\frac{3}{\sqrt{\delta}}}(0^{k})$ with $h(p)=0^k$.

Similar to the proof of Theorem \ref{FiberBundleThm}, making use of the center of mass technique, we can glue the locally defined maps to form a smooth global map $f: B_{\frac{2}{\sqrt{\delta}}}(p)\rightarrow B_{\frac{2}{\sqrt{\delta}}}(0^{k})$ such that $f(p)=0^k$, $f$ is $\Psi(\delta|n)$-close to $h|_{B_{\frac{2}{\sqrt{\delta}}}(p)}$.
And by Lemma \ref{JacobianEstimate}, we can prove the non-degeneracy of $f$ on $B_{\frac{3}{2\sqrt{\delta}}}(p)$.
Since $f: B_{\frac{2}{\sqrt{\delta}}}(p)\rightarrow B_{\frac{2}{\sqrt{\delta}}}(0^{k})$ is a $2\Psi(\delta|n)$-Gromov-Hausdorff approximation, for every $x\in B_{\frac{3}{2\sqrt{\delta}}}(p)$, we have
\begin{align}\label{6.19}
\mathrm{diam}(f^{-1}(f(x)))\leq 2\Psi(\delta|n)
\end{align}
and
\begin{align}\label{6.20}
|d(x,p)-|f(x)||\leq2\Psi(\delta|n).
\end{align}
By (\ref{6.20}), it is not hard to prove that $f(f^{-1}(\overline{B_{\frac{1}{\sqrt{\delta}}}(0^k)}))$ is closed in $\overline{B_{\frac{1}{\sqrt{\delta}}}(0^k)}$.
By the non-degeneracy of $f$, we know $f(f^{-1}(\overline{B_{\frac{1}{\sqrt{\delta}}}(0^k)}))$ is open in $\overline{B_{\frac{1}{\sqrt{\delta}}}(0^k)}$.
Hence $f(f^{-1}(\overline{B_{\frac{1}{\sqrt{\delta}}}(0^k)}))=\overline{B_{\frac{1}{\sqrt{\delta}}}(0^k)}$.
By the non-degeneracy of $f$ and (\ref{6.19}), we know $f:f^{-1}({B_{\frac{1}{\sqrt{\delta}}}(0^k)})\rightarrow {B_{\frac{1}{\sqrt{\delta}}}(0^k)}$ is a smooth surjective fibration, whose fibers are  compact $(n-k)$-dimensional manifolds with diameter at most $2\Psi(\delta|n)$.
Since ${B_{\frac{1}{\sqrt{\delta}}}(0^k)}$ is contractible, this bundle map is in fact trivial.
Finally, the required map $u: B_{2}(p)\rightarrow B_{2}(0^{k})$ can be obtained by a scaling of $f$.
\end{proof}

\section{Other Applications of Transformation Theorems}\label{sec-7}

Firstly we prove Theorem \ref{BlowupTransformation}.

\begin{proof}[Proof of Theorem \ref{BlowupTransformation}]
We assume $R_{0}=1$ for simplicity.
We also assume  $\delta$ is sufficiently small.

Given a sequence of positive numbers $R_i$ with $R_{i}\rightarrow\infty$, denote by $(\tilde{X}_{i},\tilde{p}_{i},\tilde{d}_{i},\tilde{m}_{i})=(X, p, \frac{1}{R_{i}}d, \frac{1}{m(B_{R_{i}}(p))}m)$.
By assumptions, each $B_{\delta^{-\frac{1}{2}}}(\tilde{p}_{i})$ is $(\delta,k)$-Euclidean, thus by Theorem 1.5 in \cite{Gig13},
there exists a $(\Psi(\delta|N),k)$-splitting map $u_{i}:B_{1}(\tilde{p}_{i})\rightarrow\R^k$ with $u_{i}(\tilde{p}_{i})=0$.

For any fixed $R\ge 1$, denote by $(X_{R},{p}_{R},{d}_{R},{m}_{R}):=(X, p, \frac{1}{R}d, \frac{1}{m(B_{R}(p))}m)$.
By assumptions, for $i$ sufficiently large, $B_r(\tilde{p}_{i})$ is $(\delta,k)$-Euclidean and not $(\eta,k+1)$-Euclidean for each $r\in[RR_i^{-1},1]$.
Then by Theorem \ref{thm-splitting-stable} and Remark \ref{rem4.1}, there exists a lower diagonal matrix $A_{RR_i^{-1}}$ with positive diagonal entries such that $\bar u_{i,R}:=R^{-1}R_i A_{RR_i^{-1}}u_i:B_{1}(p_{R})\to\R^k$ is $(\Psi(\delta|N,\eta),k)$-splitting and satisfies
\begin{align}\label{0124Normal}
	\bbint_{B_1(p_{R})}\langle\nabla\bar u_{i,R}^\alpha,\nabla \bar u_{i,R}^{\beta}\rangle=\delta^{\alpha\beta},
\end{align}
\begin{align}\label{7.2}
|\nabla \bar u^\alpha_{i,R}(x)|\le Cd_R(x,p_{R})^{\Psi(\delta|N,\eta)}+C
\end{align}
for any $\alpha,\beta=1,\ldots,k$, $x\in B_{R^{-1}R_i}(p_{R})$, where $C$ denotes a constant depending only on $N$ but it may change in different lines (the argument in (\ref{7.2}) is similar to that in (\ref{4.7})).
Thus up to a subsequence, we may assume $\bar u_{i,R}$ converges to a harmonic function $\bar u_{\infty,R}:X_{R}\to \R^k$ in locally uniform and locally $W^{1,2}$-sense with $\bar{u}_{\infty,R}(p_{R})=0$ and
\begin{align}\label{7.4}
|\nabla \bar u^\alpha_{\infty,R}(x)|\le Cd_R(x,p_{R})^{\Psi(\delta|N,\eta)}+C,
\end{align}
\begin{align}\label{0125Normal}
\bbint_{B_1(p_{R})}\langle\nabla\bar u_{\infty,R}^\alpha,\nabla \bar u_{\infty,R}^{\beta}\rangle=\delta^{\alpha\beta}
\end{align}
for any $\alpha,\beta=1,\ldots ,k$, $x\in X_{R}$.

By (\ref{0125Normal}) and note that $B_r({p}_{R})$ is $(\delta,k)$-Euclidean and not $(\eta,k+1)$-Euclidean for each $r\geq 1$,
similar to the argument in the proof of Theorem \ref{thm-splitting-stable-full} (where we make use of Theorem \ref{thm-gap-har-RCD} and the locally $W^{1,2}$-convergence), we can prove that $\bar u_{\infty,R}:B_1(p_{R})\to \R^k$ is $(\Psi(\delta|N,\eta),k)$-splitting.

The above argument applies for $R= 1$, and note that in this case $(X_{R},{p}_{R},{d}_{R},{m}_{R})=(X, p, d, \frac{1}{m(B_{1}(p))}m)$.
Now we fix a convergent subsequence of $\bar u_{i,1}\to\bar u_{\infty,1}$, and denote by $u:=\bar u_{\infty,1}$. Hence
\begin{align}\label{7.4111}
|\nabla u^\alpha(x)|\le Cd(x,p)^{\Psi(\delta|N,\eta)}+C,
\end{align}
\begin{align}\label{0126Normal}
\bbint_{B_1(p)}\langle\nabla u^\alpha,\nabla u^{\beta}\rangle=\delta^{\alpha\beta}
\end{align}
for any $\alpha,\beta=1,\ldots ,k$, $x\in X$.
By (\ref{7.4111}), given any $\epsilon>0$, we choose $\delta$ small so that $\Psi(\delta|N,\eta)<\epsilon$, then $u^{\alpha}\in \mathcal{H}_{1+\epsilon}(X,p)$ for $\alpha=1,\ldots,k$.
By (\ref{0126Normal}), we have $\dim\mathcal{H}_{1+\epsilon}(X,p)\ge k$.

For any $R\geq1$, we have
\begin{align}\label{8.6}
	\bar u_{i,1}=R_iA_{R_i^{-1}}u_i=RA_{R_i^{-1}}A_{RR_i^{-1}}^{-1}R^{-1}R_iA_{RR_i^{-1}}u_i =RA_{R_i^{-1}}A_{RR_i^{-1}}^{-1}\bar u_{i,R}.
\end{align}
Denote by $B_{R,i}:=A_{R_i^{-1}}A_{RR_i^{-1}}^{-1}$.
By Lemma \ref{lem-Ts-growth-estimate}, we have
\begin{align}\label{8.7}
\max\{|B_{R,i}|, |B_{R,i}^{-1}|\} \le  R^{\Psi(\delta|N,\eta)}.
\end{align}
Hence up to a subsequence, we may assume $B_{R,i}$ converges to a lower diagonal matrix  $B_R$ with positive diagonal entries satisfying
\begin{align}\label{8.8}
\max\{|B_{R}|, |B_{R}^{-1}|\} \le  R^{\Psi(\delta|N,\eta)}.
\end{align}
By (\ref{8.6}), we have $\bar u_{\infty,R}=R^{-1}B_R^{-1} u$.
If we take $T_{R}$ to be $B_{R}^{-1}$, then
$u^{1},\ldots,u^{k}$, $T_{R}$ satisfy (1) and (2).

To finish the proof, we only need to prove that there exists $\epsilon(N,\eta)$ such that for any $\epsilon\in(0,\epsilon(N,\eta)]$, we have $\dim\mathcal{H}_{1+\epsilon}(X,p)= k$.
Suppose there is a non-zero function $v\in \mathcal{H}_{1+\epsilon}(X,p)\setminus\mathrm{span}\{u^{1},\ldots,u^{k}\}$.
Without loss of generality, we may assume
\begin{equation}\label{20220125Normal}
\bbint_{B_1(p)}\langle\nabla u^\alpha,\nabla v\rangle=0\quad\text{for }\alpha=1,\ldots,k,
\end{equation}
and
\begin{align}\label{20220125Normal-2}
|v(x)|\le d(x,p)^{1+\epsilon}+4 \quad\text{for every }x\in X.
\end{align}

Similar to the proof of Theorem \ref{thm-gap-har-RCD}, if we can prove the following claim, then together with the maximum principle, by induction we can derive $v\equiv0$, which is a contradiction, and this will complete the proof.

\textbf{Claim:} There exist $\epsilon(N,\eta)$, $\delta(N,\eta,\epsilon)$ such that, for any $\epsilon\le \epsilon(N,\eta)$ and $\delta\le\delta(N,\eta,\epsilon)$, if $v\in \mathcal{H}_{1+\epsilon}(X,p)$ satisfies (\ref{20220125Normal}) and (\ref{20220125Normal-2}), then
\begin{align}
|v(x)|\le\frac{1}{2}d(x,p)^{1+\epsilon}\quad\text{for any }x\in X\setminus B_{1}(p).
\end{align}

In fact, we will take $\epsilon(N,\eta)$ as the one given in Theorem \ref{thm-gap-har-RCD} and prove the above claim holds for any $\epsilon\le \epsilon(N,\eta)$.
Suppose the claim does not hold for some $\epsilon\in(0, \epsilon(N,\eta)]$, then there exist a sequence of positive numbers $\delta_{i}\downarrow 0$, and $\mathrm{RCD}(0,N)$ spaces $(X_{i},p_{i},d_{i},m_{i})$ such that $B_r(p_{i})$ is $(\delta_{i},k)$-Euclidean and not $(\eta,k+1)$-Euclidean for any $r\ge1$;
and there exist non-zero harmonic functions $v_{i}:X_{i}\rightarrow \mathbb{R}$ satisfying $v_{i}(p_{i})=0$,
\begin{align}
|v_{i}(x)|\le d_{i}(x,p_{i})^{1+\epsilon}+4 \quad\text{for any }x\in X_{i},
\end{align}
and
\begin{equation}\label{8.12}
\bbint_{B_1(p_{i})}\langle\nabla u_{i}^\alpha,\nabla v_{i}\rangle=0\quad\text{for }\alpha=1,\ldots,k,
\end{equation}
where $u_{i}^\alpha\in \mathcal{H}_{1+\epsilon}(X_{i},p_{i})$, $\alpha=1,\ldots,k$, are constructed in the previous step of the proof (especially they satisfy (\ref{7.4111}) (\ref{0126Normal}));
and there exists $x_i\in X_i\setminus B_1(p_i)$ with $R_{i}:=d_{i}(x_{i},p_{i})\geq 1$ and
\begin{align}
|v_{i}(x_{i})|> \frac{1}{2}d_{i}(x_{i},p_{i})^{1+\epsilon}.
\end{align}

We denote by $(\tilde{X}_{i},\tilde{p}_{i},\tilde{d}_{i},\tilde{m}_{i})=(X_{i}, p_{i}, \frac{1}{R_{i}}d_{i}, \frac{1}{m_{i}(B_{R_{i}}(p))}m_{i})$.
Up to passing to a subsequence, we assume $(\tilde{X}_{i},\tilde{p}_{i},\tilde{d}_{i},\tilde{m}_{i})\xrightarrow{pmGH}(\tilde{X}_{\infty},\tilde{p}_{\infty},\tilde{d}_{\infty},\tilde{m}_{\infty})$.
Note that $\tilde{X}_{\infty}$ is a $k$-splitting $\RCD(0,N)$ space, and for any $R\ge1$, $B_{R}(\tilde{p}_{\infty})$ is not $(\frac{\eta}{2},k+1)$-Euclidean.
Let $\tilde v_i:=R_i^{-(1+\epsilon)}v_i$, then
\begin{align}\label{8.14}
|\tilde{v}_{i}(x)|\le \tilde{d}_{i}(x,\tilde{p}_{i})^{1+\epsilon}+4R_i^{-(1+\epsilon)}.
\end{align}
Hence by the gradient estimate and Theorems \ref{AA} and \ref{2.7777777}, we may assume $\tilde{v}_{i}$ converges in locally uniform and locally $W^{1,2}$-sense to a harmonic function $\tilde{v}_{\infty}:\tilde{X}_{\infty}\rightarrow \mathbb{R}$ with $\tilde{v}_{\infty}\in \mathcal{H}_{1+\epsilon}(\tilde{X}_{\infty},\tilde{p}_{\infty})$.
$\tilde{v}_{\infty}$ is non-zero because $\tilde{v}_{\infty}(\tilde{x}_{\infty})\geq \frac{1}{2}$ for some $\tilde{x}_{\infty}$ with $\tilde{d}_{\infty}(\tilde{p}_{\infty},\tilde{x}_{\infty})=1$.

Suppose $\{R_{i}\}$ is a bounded set, thus up to a subsequence, we assume $R_{i}\rightarrow R\geq 1$.
By (\ref{7.4111}) and (\ref{0126Normal}), we further assume $\tilde{u}^\alpha_{i}:=\frac{1}{R_{i}}{u}_{i}^\alpha\in \mathcal{H}_{1+\epsilon}(\tilde{X}_{i},\tilde{p}_{i})$ converges in locally uniform and locally $W^{1,2}$-sense to $\tilde{u}^\alpha_{\infty}:\tilde{X}_{\infty}\rightarrow \mathbb{R}$ with $\tilde{u}^\alpha_{\infty}\in \mathcal{H}_{1+\epsilon}(\tilde{X}_{\infty},\tilde{p}_{\infty})$ and
\begin{align}\label{0129Normal}
\bbint_{B_\frac{1}{R}(\tilde{p}_{\infty})}\langle\nabla \tilde{u}_{\infty}^\alpha,\nabla \tilde{u}_{\infty}^{\beta}\rangle=\delta^{\alpha\beta}.
\end{align}
By (\ref{8.12}) and locally $W^{1,2}$-convergence, we have
\begin{equation}\label{8.13}
\bbint_{B_\frac{1}{R}(\tilde{p}_{\infty})}\langle\nabla \tilde{u}_{\infty}^{\alpha},\nabla \tilde{v}_{\infty}\rangle=0\quad\text{for }\alpha=1,\ldots,k.
\end{equation}
Thus $\mathcal{H}_{1+\epsilon}(\tilde{X}_{\infty},\tilde{p}_{\infty})$ has dimension at least $k+1$.
However, according to Theorem \ref{thm-gap-har-RCD}, $h_{1+\epsilon}(\tilde{X}_{\infty},\tilde{p}_{\infty})=k$, which is a contradiction.

So we may assume $R_i\rightarrow \infty$.
By (\ref{8.14}), for any $\tilde{x}\in\tilde{X}_{\infty}$, we have
\begin{align}\label{8.15}
|\tilde{v}_{\infty}(\tilde{x})|\le \tilde{d}_{\infty}(\tilde{x},\tilde{p}_{\infty})^{1+\epsilon}.
\end{align}
By Theorem \ref{thm-gap-har-RCD}, we know
$\tilde{v}_{\infty}$ is the linear combination of the $\mathbb{R}^{k}$-coordinates in $\tilde{X}_{\infty}$.
Without loss of generality, we assume $\tilde{v}_{\infty}((x^{1},\ldots,x^{k},y))=cx^{1}$ for some $c>0$, where $x^{1},\ldots,x^{k}$ are the standard coordinates in $\mathbb{R}^{k}$-factor of $\tilde{X}_{\infty}=\mathbb{R}^{k}\times Y$, and $\tilde{p}_{\infty}=(0,\ldots,0,y_{0})$.
Take $\tilde{x}=(t,0,\ldots,0,y_{0})$ in (\ref{8.15}) (where $t>0$), then we have
$ct\le t^{1+\epsilon}$ for any $t>0$, which is a contradiction.
The proof is completed.
\end{proof}

In the next, we prove Proposition \ref{NonnegativeSecSplitting}.
\begin{proof}[Proof of Proposition \ref{NonnegativeSecSplitting}]

We assume there exist a sequence of $\delta_{i}\downarrow0$ and a sequence of complete $n$-manifolds $(M_i,g_i,p_i)$ with non-negative sectional curvature, and assume that $C_{\infty,1}M_i$, which is the unit ball centered at the cone point $p_{\infty,i}$ of $M_{i}$'s tangent cone at infinity, is $(\delta_i,k)$-Euclidean, but there exists no non-degenerate harmonic map $u:M_i\to\R^k$.

Up to passing to a subsequence, we may assume $(C_{\infty,1}M_i,p_{\infty,i})\xrightarrow{GH} (B_{1}((0^{k+s},x_\infty)),(0^{k+s},x_\infty))$, where $(0^{k+s},x_\infty)\in \R^k\times\R^s\times C(X)$, and  $C(X)$ is a metric cone containing no lines, with $x_\infty$ being its cone point.
Hence there exists $\eta>0$ such that $d_{GH}(B_1((0^{k+s},x_\infty)), B_{1}((0^{k+s+1},z)))\geq \eta$ for any $B_{1}((0^{k+s+1},z))\subset\mathbb{R}^{k+s+1}\times Z$, with $Z$ being a metric space and $z\in Z$.
Then by the uniqueness of tangent cone at infinity of manifolds with nonnegative sectional curvature, up to a scaling down of the manifolds, we may assume for each $R\ge 1$ and $i\in \mathbb{Z}^{+}$, $B_R(p_i)$ is $(\delta_i,k+s)$-Euclidean, and not $(\frac{\eta}{4},k+s+1)$-Euclidean for $\delta_i\downarrow0$.
Hence we apply Theorem \ref{BlowupTransformation} to obtain a harmonic map $u_i=(u_i^1,\ldots,u_i^{k+s}):(M_i,p_i)\to(\R^{k+s},0^{k+s})$ with at most $(1+\Psi(\delta_i|n,\eta))$-growth, and satisfies (1)-(3) there.
In particular, for each $R\ge 1$, there exists a lower diagonal matrix $T_R$ with positive diagonal entries such that $T_R\circ u_{i}:B_R(p_{i})\rightarrow\R^{k+s}$ is $(\Psi(\delta_i|n,\eta),k+s)$-splitting.
Thus by Theorem \ref{NonDegeneracyofSplittingMaps}, for each sufficiently large $i$, and for every $R\geq 1$, $T_R \circ u_{i}$ is non-degenerate on $B_\frac{R}{3}(p_{i})$, and hence $u_{i}$ is non-degenerate on $B_\frac{R}{3}(p_{i})$.
By the arbitrariness of $R$, $u_{i}$ is non-degenerate on $M_{i}$.
This contradicts to the non-existence of non-degenerate harmonic map $u:M_i\to\R^k$, and we have finished the proof of the first part.

Now we prove the second part, where we assume $d_{GH}(C_{\infty,1}M,B_1(0^k))\le\delta$ with $\delta$ sufficiently small.
Up to a scaling down, we may assume that $B_R(p)$ is $(2\delta,k)$-Euclidean and not $(\frac{1}{2},k+1)$-Euclidean for any $R\geq 1$.
Hence we apply Theorem \ref{BlowupTransformation} to obtain a harmonic map $u=(u^1,\ldots,u^k): M\to \R^k$ with $u(p)=0^k$ and at most $(1+\Psi_{1}(\delta|n))$-growth, and satisfies (1)-(3) there.
In addition, we may assume that, for each $R\geq 1$, $\tilde{u}_{R}:=T_R\circ u:B_{\frac{1}{2}R}(p)\mapsto B_{(\frac{1}{2}+\Psi_{2}(\delta|n))R}(0^k)$ is a $(\Psi_{2}(\delta|n)R)$-Gromov-Hausdorff approximation.
In particular, for every $x\in B_{\frac{1}{2}R}(p)$, we have
\begin{align}\label{7.19}
\mathrm{diam}(\tilde{u}_{R}^{-1}(\tilde{u}_{R}(x)))\leq \Psi_{2}(\delta|n)R
\end{align}
and
\begin{align}\label{7.20}
|d(x,p)-|\tilde{u}_{R}(x)||\leq\Psi_{2}(\delta|n)R.
\end{align}

We claim that $u$ is a proper map.

Suppose on the contrary, there exists a sequence of $x_i$ such that $R_i:=d(p,x_i)\rightarrow \infty$, but $|u(x_i)|\leq R_{0}$.
We consider $\tilde{u}_{10R_{i}}$.
By (\ref{7.20}), if $\delta$ is sufficiently small, we have $|\tilde{u}_{10R_i}(x_i)|\geq9R_{i}$.
Since $|T_{10R_i}|\leq (10R_i)^{\Psi_{1}(\delta|n)}$,
we have $|u(x_i)|\geq R_{i}^{1-\Psi_{1}(\delta|n)}$, which is a contradiction.

By the properness of $u$, $u(M)$ is closed in $\R^{k}$.

On the other hand, by Theorem \ref{NonDegeneracyofSplittingMaps}, $d\tilde{u}_{R}$ is non-degenerate on $B_{\frac{1}{6}R}(p)$ for every $R\geq 1$.
By the arbitrariness of $R$, $u:M\to \R^{k}$ is non-degenerate on $M$, hence $u(M)$ is open in $\R^k$.
Thus $u: M\to \R^{k}$ is surjective.
Note that $\tilde{u}_{R}^{-1}(\tilde{u}_{R}(x))=u^{-1}(u(x))$ for every $x\in u^{-1}(B_{\frac{1}{5}R}(0^{k}))$, and by (\ref{7.19}) and the non-degeneracy of $u$, $u: M\to \R^k$ is a surjective fiber bundle map with compact fibers.
Since $\R^{k}$ is contractible, the bundle map $u$ is trivial.
This completes the proof.
\end{proof}

Now we prove Proposition \ref{thm-can-diff}.

\begin{proof}[Proof of Proposition \ref{thm-can-diff}]
By the theorems in \cite{CC96}, for any $R>0$, $B_R(p)$ is $(\Psi(\delta|n),n)$-Euclidean and not $(\frac{1}{2},n+1)$-Euclidean.
Hence we apply Theorem \ref{BlowupTransformation} to obtain a harmonic function $u=(u^1,\ldots,u^n)$ with $u(p)=0^{n}$ and at most $(1+\Psi(\delta|n))$-growth, and satisfies (1)-(3) there.
Similar to the proof of the second part of Proposition \ref{NonnegativeSecSplitting}, we can prove that $u: M\to \R^{n}$ is a proper  surjective local diffeomorphism.
And since $\R^{n}$ is simply-connected, $u$ is injective, hence a diffeomorphism.
\end{proof}

Similar to Proposition \ref{thm-can-diff}, we can prove:

\begin{prop}\label{thm-can-heom}
For any $N\in \mathbb{Z}^{+}$, there exists $\delta(N)>0$, such that the following holds for any $\delta\in(0,\delta(N))$.
Suppose $(X,d,\mathcal{H}^N)$ is a noncompact $\RCD(0,N)$ space, and
\begin{align}
\lim_{R\rightarrow+\infty}\frac{\mathcal{H}^N(B_R(p))}{\mathcal{H}^N(B_R(0^n))}\ge 1-\delta.
\end{align}
Then there exists a proper harmonic map $u:X\to \R^N$ with at most $(1+\Psi(\delta|N))$-growth so that $u$ is a homeomorphism.
\end{prop}

\begin{proof}[Sketch of proof]
The only difference to Proposition \ref{thm-can-diff} is that, because of the lack of smooth structure, we don't have a non-degeneracy property for $T_{R}u: B_{R}(p)\to \R^{N}$.
In this case we will use a bi-H\"{o}lder estimate on non-collapsed $\RCD$ spaces proved in \cite{BNS22}: since $T_{R}u$ is $\Psi(\delta|N)$-splitting,
\begin{equation}\label{holder}
\frac{1}{C(N,R)}d(x,y)^{1+\Psi(\delta|N)}\le |T_{R}u(x)-T_{R}u(y)|\le (1+\Psi(\delta|N))d(x,y)
\end{equation}
holds for any $x,y\in B_{\frac{1}{3}R}(p)$.

By (\ref{holder}), for every $R>0$, $T_{R}u$ is a homeomorphism onto its image on $B_{\frac{1}{3}R}(p)$, and hence $u|_{B_{\frac{1}{3}R}(p)}$ is also a homeomorphism onto its image.
Then it is easy to see that $u:X\rightarrow \R^{N}$ is injective.
On the other hand, according to Theorem 1.3 of \cite{KM20}, %under the assumption of Proposition \ref{thm-can-heom},
$X$ is homeomorphic to $\R^{N}$, then by the invariance of domain theorem, $u$ is an open map.
Similar to the proof of the second part of Proposition \ref{NonnegativeSecSplitting}, by an open-closed argument, we can prove that $u:X\to \R^{N}$ is a proper surjective map, hence $u$ is a homeomorphism.
\end{proof}

In the following we prove Proposition \ref{prop-1.14}.
\begin{proof}[Proof of Proposition \ref{prop-1.14}]
Given any $\epsilon>0$ and $n\in \mathbb{Z}^+$, we will determine $\delta_{0}=\delta(n,\epsilon)$ later.

Given any $n$-manifold $(M,g)$ with non-negative sectional curvature with $h_{1+\delta_0}(M,p)=k$, we suppose that there exists a sequence $R_i\rightarrow\infty$ such that $B_{R_i}(p)\subset M^{n}$ is not $(\epsilon,k)$-Euclidean.
By Theorem \ref{thm-LN20}, it is easy to see that there exist $\eta_{1}=\eta_{1}(\epsilon,n)>0$ and $R^{(1)}_0\geq1$ ($R^{(1)}_0$ depends on $M$) such that $B_{R}(p)$ is not $(\eta_{1},k)$-Euclidean for any $R\geq R_0^{(1)}$.
According to Theorem \ref{BlowupTransformation}, if there exist sufficiently small $\delta_{1}>0$ and $\tilde{R}_0^{(1)}\geq 1$ with $\delta_{1}$ depending only on $\eta_{1}$ and $n$, so that $B_R(p)$ is $(\delta_{1},k-1)$-Euclidean for every $R\ge \tilde{R}_0^{(1)}$, then $\mathcal{H}_{1+\Psi(\delta_{1})}(M,p)$ has dimension $k-1$.
If we take $\delta_0<\Psi(\delta_{1})$, then $h_{1+\delta_0}(M,p)=k$ implies the existence of a sequence $R_i^{(1)}\rightarrow\infty$ such that $B_{R_i^{(1)}}(p)$ is not $(\delta_{1},k-1)$-Euclidean,
and then by Theorem \ref{thm-LN20}, there exist $\eta_{2}>0$ (depending only on $\epsilon$ and $n$) and $R^{(2)}_0\geq1$ so that $B_{R}(p)$ is not $(\eta_{2},k-1)$-Euclidean for any $R\geq R_0^{(2)}$.

Repeat the similar arguments for $k-2$ more times, we will finally conclude that, if $\delta_0$ is chosen to be a sufficiently small number depending only on $\epsilon$ and $n$,
then $h_{1+\delta_0}(M,p)=k$ implies the existence of $\eta_{k}>0$ (depending only on $\epsilon$ and $n$) and $R^{(k)}_0\geq1$ such that $B_{R}(p)$ is not $(\eta_{k},1)$-Euclidean for any $R\geq R_0^{(k)}$.
Then according to Proposition \ref{thm-liou-har-RCD}, there exists $\gamma>0$ depending only on $\eta_{k}$ and $n$ (hence depending only on $\epsilon$ and $n$) so that $h_{1+\gamma}(M,p)=0$.
If $\delta_0$ is further chosen to be smaller than $\gamma$, then $h_{1+\delta_0}(M,p)=k$ cannot happen.
Equivalently, if $\delta_0$ is chosen sufficiently small as above, then $h_{1+\delta_0}(M,p)=k$ will imply that there exists an $R_0\geq1$ such that $B_{R}(p)$ is $(\epsilon,k)$-Euclidean for every $R\geq R_0$.
The proof is completed.
\end{proof}

\begin{rem}
It is easy to see that the proof of Proposition \ref{prop-1.14} can be applied to $\RCD(0,N)$ spaces $(X,d,m)$ with $(X,d)\in \mathrm{Alex}^{n}(0)$ ($n\leq N$).
\end{rem}

\section{Appendix: A generalized covering lemma}\label{sec-8}

In this appendix, we give a detailed proof of the generalized covering lemma \ref{covering-lem}.
The proof follows the ideas in Lemma 1.6 of \cite{KW}.

In the following, suppose on a not necessarily complete $n$-manifold $(M, g)$, $B_{8}(p)$ has compact closure in $B_{10}(p)$, and there is a surjective local diffeomorphism $\sigma:\hat M\to B_{10}(p)$ with $\sigma(\hat p)=p$.
We equip $\hat M$ with the pull-back metric $\hat{g}=\sigma^{*}g$.
Let $\hat{d}$ be the length distance on $\hat M$ induced by $\hat{g}$.
Obviously, $\sigma$ is distance non-increasing.
Since $\overline{B_{8}(p)}$ is compact in $B_{10}(p)$, it is easy to see that
$\overline{B_{8}(\hat{p})}$ is compact in $B_{10}(\hat{p})$.
By the compactness of $\overline{B_{8}(\hat{p})}$, it is easy to see that, for any $\hat{x},\hat{y}\in B_{4}(\hat{p})$, $\hat{d}(\hat{x},\hat{y})$ is realized by a segment connecting $\hat{x},\hat{y}$ and contained in $B_{8}(\hat{p})$.

For any $\hat{q}\in\sigma^{-1}(p)\cap B_{2}(\hat{p})$, define
$$\Omega_{\hat{q}}:=\bigcap_{\tilde{q}\in \sigma^{-1}(p)\cap B_{4}(\hat{p})\setminus\{\hat{q}\}}\{\hat{x}\in  \sigma^{-1}(B_{1}(p))\cap B_{1}(\hat{q})| \hat{d}(\hat{x},\hat{q})< \hat{d}(\hat{x},\tilde{q})\}.$$
\textbf{Claim:} $\Omega_{\hat{q}}$ has the following properties:
\begin{description}
  \item[(1)] Every $\Omega_{\hat{q}}$ is a bounded open set.
  \item[(2)] $\Omega_{\hat{q}}\cap\Omega_{\tilde{q}}=\emptyset$ if $\hat{q}\neq \tilde{q}$.
  \item[(3)] If $\hat{x}\in{\Omega}_{\hat{q}}$ for some $\hat{q}\in \sigma^{-1}(p)\cap B_{2}(\hat{p})$, then $\hat{x}\in B_{3}(\hat{p})$.
  \item[(4)] $B_{1}(\hat{p})\subset \bigcup_{\hat{q}\in\sigma^{-1}(p)\cap B_{2}(\hat{p})} \overline{\Omega}_{\hat{q}}$.
  \item[(5)] If $\hat{x}\in{\Omega}_{\hat{q}}$ for some $\hat{q}\in \sigma^{-1}(p)\cap B_{2}(\hat{p})$, let $\hat{\gamma}: [0, \hat{d}(\hat{q},\hat{x})]\rightarrow B_{1}(\hat{q})$ be a segment connecting $\hat{q}$ and $\hat{x}$, then for any $t\in [0, \hat{d}(\hat{q},\hat{x})]$, $\hat{\gamma}(t)\in {\Omega}_{\hat{q}}$, and $\gamma:=\sigma\circ\hat{\gamma}$ is a segment connecting $p$ and $x=\sigma(\hat{x})$. In particular, we have
      \begin{align}\label{8.2}
      \hat{d}(\hat{x},\hat{q})=d(x,p),\text{ for any } \hat{x}\in{\Omega}_{\hat{q}}\text{ and }x=\sigma(\hat{x}).
      \end{align}
  \item[(6)] Let $\gamma:[0,L]\to B_{1}(p)$ (where $L=d(p,\gamma(1))<1$) be a geodesic such that $\gamma(0)=p$ and $d(\gamma(t),p)=t$ for every $t\in[0,L]$. We lift $\gamma$ to the geodesic $\hat{\gamma}: [0,L]\to B_{1}(\hat{q})$ with $\hat{\gamma}(0)=\hat{q}\in\sigma^{-1}(p)\cap B_{2}(\hat{p})$, then for any $t\in [0,L)$, $\hat{\gamma}(t)\in {\Omega}_{\hat{q}}$, and $\hat{d}(\hat{\gamma}(t),\hat{q})=t$.
      In particular, for any $t\in [0,L)$, $\hat{\gamma}(t)\in {\Omega}_{\hat{q}}\cap \mathrm{Cut}_{\hat{q}}$, and $\hat{\gamma}|_{[0,\hat{d}(\hat{\gamma}(t),\hat{q})]}$ is the unique shortest geodesic connecting $\hat{q}$ and $\hat{\gamma}(t)$, where $\mathrm{Cut}_{\hat{q}}$ denotes the cut locus of $\hat{q}$.
  \item[(7)] For any $\hat{q}\in\sigma^{-1}(p)\cap B_{2}(\hat{p})$, $\sigma|_{\Omega_{\hat{q}}\setminus \mathrm{Cut}_{\hat{q}}}$ is a diffeomorphism onto $B_{1}(p)\setminus \mathrm{Cut}_{p}$. In particular, $\mathcal{H}^{n}(\Omega_{\hat{q}})=\mathcal{H}^{n}(B_{1}(p))$.
\end{description}

\begin{proof}
\textbf{(1)} is trivial because $\sigma^{-1}(p)\cap B_{4}(\hat{p})$ has finitely many elements.
\textbf{(2)} and \textbf{(3)} are also obviously.
In the following, we prove \textbf{(4)}-\textbf{(7)}.

\textbf{Proof of (4):} For any $\hat{x}\in B_{1}(\hat{p})$, it is easy to see
$$\min\{\hat{d}(\hat{x},\tilde{q})|\tilde{q}\in\sigma^{-1}(p)\cap B_{4}(\hat{p})\}<1$$
is realized by some $\hat{q}\in \sigma^{-1}(p)\cap B_{2}(\hat{p})$, and thus $\hat{x}\in\overline{\Omega}_{\hat{q}}$.

\textbf{Proof of (5)}: For any $t\in (0, \hat{d}(\hat{q},\hat{x}))$, suppose $\hat{y}:=\hat{\gamma}(t)\notin {\Omega}_{\hat{q}}$,
then there exists $\tilde{q}\in \sigma^{-1}(p)\cap B_{4}(\hat{p})\setminus\{\hat{q}\}$ such that $\hat{d}(\hat{y},\hat{q})\geq \hat{d}(\hat{y},\tilde{q})$.
Thus $\hat{d}(\hat{x},\tilde{q})\leq \hat{d}(\hat{x},\hat{y})+ \hat{d}(\hat{y},\tilde{q})\leq \hat{d}(\hat{x},\hat{q})$, contradicting to $\hat{x}\in{\Omega}_{\hat{q}}$.
Hence $\hat{\gamma}(t)\in {\Omega}_{\hat{q}}$.

Suppose $\gamma:=\sigma\circ\hat{\gamma}$ is not a shortest geodesic connecting $p$ and $x=\sigma(\hat{x})$, then we find a segment $\eta:[0,d(x,p)]\to B_{1}(p)$ with $\eta(0)=p$, $\eta(d(x,p))=x$, and lift it to a segment $\hat{\eta}:[0,d(x,p)]\to B_{8}(\hat{p})$ with $\hat{\eta}(d(x,p))=\hat{x}\in B_{3}(\hat{p})$, $\hat{\eta}(0):=\tilde{q}\in \sigma^{-1}(p)\cap B_{4}(\hat{p})$.
Now we have
$$\hat{d}(\hat{x},\tilde{q})= \mathrm{Length}(\hat{\eta})=d(x,p)<\mathrm{Length}(\hat{\gamma})=\hat{d}(\hat{q},\hat{x}),$$
contradicting to $\hat{x}\in {\Omega}_{\hat{q}}$.
Thus $\gamma$ is a segment connecting $p$ and $x$.

\textbf{Proof of (6):} Suppose there is some $t\in(0,L)$ such that $\hat{x}:=\hat{\gamma}(t)\notin {\Omega}_{\hat{q}}$.
Note that $\hat{x}\in B_{3}(\hat{p})$.
Denote by $x=\gamma(t)$.
By definition, there exists $\tilde{q}\in \sigma^{-1}(p)\cap B_{4}(\hat{p})\setminus\{\hat{q}\}$ such that $\hat{d}(\hat{x},\tilde{q})\leq \hat{d}(\hat{x},\hat{q})=d(x,p)<1$.
We connect $\tilde{q}$ and $\hat{x}$ by a segment $\hat{c}:[0,\hat{d}(\hat{x},\tilde{q})]\to B_{8}(\hat{p})$ with $\hat{c}(0)=\tilde{q}$ and $\hat{c}(\hat{d}(\hat{x},\tilde{q}))=\hat{x}$.
Then $c=\sigma\circ\hat{c}$ is a geodesic connecting $p$ and $x$ with $\mathrm{Length}(c)=\mathrm{Length}(\hat{c}) \leq \hat{d}(\hat{x},\hat{q})\leq d(x,p)$.
Because $\gamma|_{[0,t]}$ is the unique shortest geodesic connecting $p$ and $x$, we have
$d(x,p)=\hat{d}(\hat{x},\tilde{q})=\hat{d}(\hat{x},\hat{q})$, and $c(s)=\gamma(s)$, $\hat{c}(s)=\hat{\gamma}(s)$ for every $s\in [0,t]$, contradicting to $\tilde{q}\neq\hat{q}$.
Thus $\hat{\gamma}(t)\in {\Omega}_{\hat{q}}$ holds for every $t\in [0,L)$.
The conclusion $\hat{d}(\hat{\gamma}(t),\hat{q})=t$ follows from the distance non-increasing of $\sigma$ and (\ref{8.2}).

\textbf{Proof of (7):} Note that $B_{1}(p)\setminus \mathrm{Cut}_{p}$ is an open set.
For any $x\in B_{1}(p)\setminus \mathrm{Cut}_{p}$, we can always find a segment $\gamma:[0,L]\to B_{1}(p)\setminus \mathrm{Cut}_{p}$ with $L>d(x,p)$, $\gamma(0)=p$,  $\gamma(d(x,p))=x$ and $d(\gamma(s),p)=s$ for any $s\in[0,L]$.
If we lift $\gamma$ to the geodesic $\hat{\gamma}: [0,L]\to B_{1}(\hat{q})$ with $\hat{\gamma}(0)=\hat{q}$, and denote by $\hat{x}=\hat{\gamma}(d(x,p))$.
Then by \textbf{(6)}, $\hat{\gamma}|_{[0,\frac{L+d(x,p)}{2}]}$ is a shortest geodesic, and then by \textbf{(5)(6)} again, $\hat{x}\in \Omega_{\hat{q}}\setminus \mathrm{Cut}_{\hat{q}}$.
This proves $B_{1}(p)\setminus \mathrm{Cut}_{p}\subset \sigma(\Omega_{\hat{q}}\setminus \mathrm{Cut}_{\hat{q}})$.

On the other hand, since $\Omega_{\hat{q}}\setminus \mathrm{Cut}_{\hat{q}}$ is open, for any $\hat{x}\in\Omega_{\hat{q}}\setminus \mathrm{Cut}_{\hat{q}}$, we find a geodesic
$\hat{\gamma}:[0,L]\to B_{1}(\hat{q})\setminus \mathrm{Cut}_{\hat{q}}$ with $L>\hat{d}(\hat{x},\hat{q})$, $\hat{\gamma}(0)=\hat{q}$,  $\hat{\gamma}(\hat{d}(\hat{x},\hat{q}))=\hat{x}$ and $\hat{d}(\hat{\gamma}(s),\hat{q})=s$ for any $s\in[0,L]$.
Then by \textbf{(5)}, $\sigma(\hat{x})\in B_{1}(p)\setminus \mathrm{Cut}_{p}$.
This proves $ \sigma(\Omega_{\hat{q}}\setminus \mathrm{Cut}_{\hat{q}})\subset B_{1}(p)\setminus \mathrm{Cut}_{p}$.

Similarly, by the openness of $\Omega_{\hat{q}}\setminus \mathrm{Cut}_{\hat{q}}$ and \textbf{(5)} \textbf{(6)} again, it is not hard to prove that $\sigma|_{\Omega_{\hat{q}}\setminus \mathrm{Cut}_{\hat{q}}}$ is injective, hence $\sigma|_{\Omega_{\hat{q}}\setminus \mathrm{Cut}_{\hat{q}}}$ is a diffeomorphism.
This completes the proof of \textbf{(7)}.
\end{proof}

\begin{proof}[Proof of Lemma \ref{covering-lem}]
We first prove the $r=1$ case.
The $\Omega_{\hat{q}}$ constructed before will play an important role in the proof.
We denote by $S=\bigcup_{\hat{q}\in\sigma^{-1}(p)\cap B_{2}(\hat{p})} \overline{\Omega}_{\hat{q}}$.
Firstly, note that $\mathcal{H}^{n}(\bar{\Omega}_{\hat{q}}\setminus (\Omega_{\hat{q}}\cap \mathrm{Cut}_{\hat{q}}))=0$, and by property \textbf{(7)}, we have $\mathcal{H}^{n}(B_{1}(p))=\mathcal{H}^{n}({\Omega}_{\hat{q}})=\mathcal{H}^{n}({\bar{\Omega}}_{\hat{q}})$ and
\begin{align}
\bbint_{B_{1}(p)}f=\bbint_{{\Omega}_{\hat{q}}}f\circ\sigma=\bbint_{\bar{\Omega}_{\hat{q}}}f\circ\sigma
\end{align}
for every $\hat{q}\in\sigma^{-1}(p)\cap B_{2}(\hat{p})$, and hence by property \textbf{(2)},
\begin{align}\label{8.98}
\bbint_{B_{1}(p)}f=\bbint_{S}f\circ\sigma.
\end{align}

By properties \textbf{(3)} and \textbf{(4)}, we have $B_{1}(\hat{p})\subset S \subset B_{3}(\hat{p})$.
Then by Bishop-Gromov volume comparison theorem, there exists a constant $C=C(n)$ such that
\begin{align}
\mathcal{H}^{n}(B_{1}(\hat{p}))\leq \mathcal{H}^{n}(S)\leq \mathcal{H}^{n}(B_{3}(\hat{p}))\leq C \mathcal{H}^{n}(B_{1}(\hat{p})).
\end{align}
Thus
\begin{align}\label{8.99}
\frac{1}{C}\bbint_{B_{1}(\hat{p})}f\circ\sigma\leq \bbint_{S}f\circ\sigma\leq C\bbint_{B_{3}(\hat{p})}f\circ\sigma.
\end{align}
The $r=1$ case of (\ref{ControlBetweenCover}) follows from (\ref{8.98}) and (\ref{8.99}).

For the $r<1$ case, the proof is similar: we only need to cover $B_{r}(\hat{p})$ by the union of the closures of
$$\Omega^{(r)}_{\hat{q}}:=\bigcap_{\tilde{q}\in \sigma^{-1}(p)\cap B_{4}(\hat{p})\setminus\{\hat{q}\}}\{\hat{x}\in  \sigma^{-1}(B_{r}(p))\cap B_{r}(\hat{q})| \hat{d}(\hat{x},\hat{q})< \hat{d}(\hat{x},\tilde{q})\},$$
where $\hat{q}\in\sigma^{-1}(p)\cap B_{2r}(\hat{p})$.

The proof is completed.
\end{proof}

\vspace*{20pt}

\noindent\textbf{Acknowledgments.}

The authors would like to thank Prof. X. Rong for suggestions to them which motivates this work.
They also would like to thank Prof. B.-L. Chen, X. Rong, H.-C. Zhang, X.-P. Zhu for encouragements and discussions.
The authors would like to thank the anonymous referees for careful reading and giving helpful suggestions.
The second author is partially supported by National Key R\&D Program of China (2021YFA1002100),  National Natural Science Foundation of China (12271531) and Guangdong Natural Science Foundation (2022A1515011072).

\bibliographystyle{alpha}
\bibliography{ref}% bibliography

\begin{thebibliography}{AGMR15}

\bibitem[AC92]{AC92}
M.~T. Anderson and J.~Cheeger.
\newblock ${C}^{\alpha}$-compactness for manifolds with {R}icci curvature and
  injectivity radius bounded below.
\newblock {\em J. Differential Geom.}, 35:265--281, 1992.

\bibitem[AGMR15]{AGMR15}
L.~Ambrosio, N.~Gigli, A.~Mondino, and T.~Rajala.
\newblock Riemannian {R}icci curvature lower bounds in metric measure spaces
  with $\sigma$-finite measure.
\newblock {\em Trans. Amer. Math. Soc.}, 367(7):4661--4701, 2015.

\bibitem[AGS14]{AGS14-1}
L.~Ambrosio, N.~Gigli, and G.~Savar\'{e}.
\newblock Calculus and heat flow in metric measure spaces and applications to
  spaces with {R}icci bounds from below.
\newblock {\em Invent. Math.}, 195:289–391, 2014.

\bibitem[AH17]{AH17}
L.~Ambrosio and S.~Honda.
\newblock New stability results for sequences of metric measure spaces with
  uniform {R}icci bounds from below.
\newblock {\em Measure Theory in Non-Smooth Spaces, De Gruyter Open, Warsaw},
  68:1--51, 2017.

\bibitem[AH18]{AH18}
L.~Ambrosio and S.~Honda.
\newblock Local spectral convergence in {RCD(K,N)} spaces.
\newblock {\em Nonlinear Anal.}, 177:1--23, 2018.

\bibitem[And92]{An}
M.~T. Anderson.
\newblock Hausdorff perturbations of {R}icci-flat manifolds and the splitting
  theorem.
\newblock {\em Duke Math. J.}, 68(1):67--82, 1992.

\bibitem[BGHZ23]{BGHZ23}
C.~Brena, N.~Gigli, S.~Honda, and X.~Zhu.
\newblock Weakly non-collapsed {RCD} spaces are strongly non-collapsed.
\newblock {\em J. {R}eine {A}ngew. Math.}, 794:215--252, 2023.

\bibitem[BNS22]{BNS22}
E.~Bru\`{e}, A.~Naber, and D.~Semola.
\newblock Boundary regularity and stability for spaces with {R}icci bounded
  below.
\newblock {\em Invent. Math.}, 228:777--891, 2022.

\bibitem[BPS23]{BPS19}
E.~Bru\`{e}, E.~Pasqualetto, and D.~Semola.
\newblock Rectifiability of the reduced boundary for sets of finite perimeter
  over $\textmd{RCD}({K},{N})$ spaces.
\newblock {\em J. Eur. Math. Soc.}, 25:413--465, 2023.

\bibitem[CC96]{CC96}
J.~Cheeger and T.~H. Colding.
\newblock Lower bounds on {R}icci curvature and the almost rigidity of warped
  products.
\newblock {\em Ann. of Math. (2)}, 144(1):189--237, 1996.

\bibitem[CC97]{CC97}
J.~Cheeger and T.~H. Colding.
\newblock On the structure of spaces with {R}icci curvature bounded below. {I}.
\newblock {\em J. Differential Geom.}, 45:406--480, 1997.

\bibitem[CCM95]{CCM}
J.~Cheeger, T.~H. Colding, and W.~P. Minicozzi.
\newblock Linear growth harmonic functions on complete manifolds with
  nonnegative {R}icci curvature.
\newblock {\em Geom. Funct. Anal.}, 5(6):948--954, 1995.

\bibitem[Che99]{C99}
J.~Cheeger.
\newblock Differentiability of {L}ipschitz functions on metric measure spaces.
\newblock {\em Geom. Funct. Anal.}, 9:428--517, 1999.

\bibitem[CJN21]{CJN21}
J.~Cheeger, W.~Jiang, and A.~Naber.
\newblock Rectifiability of singular sets in noncollapsed spaces with {R}icci
  curvature bounded below.
\newblock {\em Ann. of Math. (2)}, 193(2):407--538, 2021.

\bibitem[CN13]{CN13}
T.~H. Colding and A.~Naber.
\newblock Characterization of tangent cones of noncollapsed limits with lower
  {R}icci bounds and applications.
\newblock {\em Geom. Funct. Anal.}, 23:134--148, 2013.

\bibitem[CN15]{CN}
J.~Cheeger and A.~Naber.
\newblock Regularity of {E}instein manifolds and the codimension 4 conjecture.
\newblock {\em Ann. of Math. (2)}, 182:1093--1165, 2015.

\bibitem[Din04]{Ding04}
Y.~Ding.
\newblock An existence theorem of harmonic functions with polynomial growth.
\newblock {\em Proc. Amer. Math. Soc.}, 132(2):543--551, 2004.

\bibitem[DPG18]{DePGil18}
G.~De~Philippis and N.~Gigli.
\newblock Non-collapsed spaces with {R}icci curvature bounded from below.
\newblock {\em J. ´Ec. polytech. Math.}, 5:613--650, 2018.

\bibitem[DWY96]{DWY}
X.~Dai, G.~Wei, and R.~Ye.
\newblock Smoothing {R}iemannian metrics with {R}icci curvature bounds.
\newblock {\em Manuscripta Math.}, 90:49--61, 1996.

\bibitem[EKS15]{EKS15}
M.~Erbar, K.~Kuwada, and K.~T. Sturm.
\newblock On the equivalence of the entropic curvature-dimension condition and
  {B}ochner's inequality on metric measure spaces.
\newblock {\em Invent. Math.}, 201:993--1071, 2015.

\bibitem[Fuk87]{Fu87}
K.~Fukaya.
\newblock Collapsing {R}iemannian manifolds to ones with lower dimension. {I}.
\newblock {\em J. Differential Geom.}, 25:139--156, 1987.

\bibitem[FY92]{FY92}
K.~Fukaya and T.~Yamaguchi.
\newblock The fundamental groups of almost non-negatively curved manifolds.
\newblock {\em Ann. of Math. (2)}, 136(2):253--333, 1992.

\bibitem[Gig13]{Gig13}
N.~Gigli.
\newblock The splitting theorem in non-smooth context.
\newblock {\em arXiv:1302.5555}, 2013.

\bibitem[Gig15]{Gig15}
N.~Gigli.
\newblock On the differential structure of metric measure spaces and
  applications.
\newblock {\em Mem. Amer. Math. Soc.}, 236, 2015.

\bibitem[GMS15]{GMS15}
N.~Gigli, A.~Mondino, and G.~Savar\'{e}.
\newblock Convergence of pointed non-compact metric measure spaces and
  stability of {R}icci curvature bounds and heat flows.
\newblock {\em Proc. Lond. Math. Soc. (3)}, 111(5):1071--1129, 2015.

\bibitem[GR18]{GR18}
N.~Gigli and C.~Rigoni.
\newblock Recognizing the flat torus among $\textmd{RCD}^*(0,{N})$ spaces via
  the study of the first cohomology group.
\newblock {\em Calc. Var. Partial Differential Equations}, 57:104,39 pp, 2018.

\bibitem[Gro78]{Gr78}
M.~Gromov.
\newblock Almost flat manifolds.
\newblock {\em J. Differential Geom.}, 13:231--241, 1978.

\bibitem[Gro81]{Gr81}
M.~Gromov.
\newblock Groups of polynomial growth and expanding maps.
\newblock {\em Publications mathematiques I.H.E.S.}, 53:53--78, 1981.

\bibitem[Han18]{Han18-2}
B.~Han.
\newblock Characterizations of monotonicity of vector fields on metric measure
  spaces.
\newblock {\em Calc. Var. Partial Differential Equations}, 57:113,35 pp, 2018.

\bibitem[Hei01]{He01}
J.~Heinonen.
\newblock Lectures on analysis on metric spaces.
\newblock {\em Universitext, Springer}, pages x+140 pp, 2001.

\bibitem[HH24]{HH}
H.~Huang and X.-T. Huang.
\newblock Fibration properties for manifolds with {R}icci curvature lower bound
  and generalized {R}eifenberg property.
\newblock {\em in preparation}, 2024.

\bibitem[HKRX20]{HKRX}
H.~Huang, L.~Kong, X.~Rong, and S.~Xu.
\newblock Collapsed manifolds with {R}icci bounded covering geometry.
\newblock {\em Trans. Amer. Math. Soc.}, 373(11):8039--8057, 2020.

\bibitem[HKX13]{HKX13}
B.~Hua, M.~Kell, and C.~Xia.
\newblock Harmonic functions on metric measure spaces.
\newblock {\em arXiv:1308.3607}, 2013.

\bibitem[HP23]{HP22}
S.~Honda and Y.~Peng.
\newblock A note on the topological stability theorem from {RCD} spaces to
  {R}iemannian manifolds.
\newblock {\em {M}anuscripta {M}ath.}, 172:971--1007, 2023.

\bibitem[Hua20]{Hua}
H.~Huang.
\newblock Fibrations and stability for compact group actions on manifolds with
  local bounded {R}icci covering geometry.
\newblock {\em Front. Math. China}, 15(1):69--89, 2020.

\bibitem[Hua24]{Hua22}
H.~Huang.
\newblock A finite topological type theorem for open manifolds with
  non-negative {R}icci curvature and almost maximal local rewinding volume.
\newblock {\em Int. Math. Res. Not.}, 10:8568--8591, 2024.

\bibitem[HW22]{HW20}
S.~Huang and B.~Wang.
\newblock Ricci flow smoothing for locally collapsing manifolds.
\newblock {\em Calc. Var. Partial Differential Equations}, 61:Paper No. 64, 32
  pp, 2022.

\bibitem[Jia14]{J14}
R.~Jiang.
\newblock Cheeger-harmonic functions in metric measure spaces revisited.
\newblock {\em J. Funct. Anal.}, 266(3):1373--1394, 2014.

\bibitem[Kit19]{Ki19}
Y.~Kitabeppu.
\newblock A sufficient condition to a regular set being of positive measure on
  {RCD} spaces.
\newblock {\em Potential Anal.}, 51:179--196, 2019.

\bibitem[KM21]{KM20}
V.~Kapovitch and A.~Mondino.
\newblock On the topology and the boundary of {N}-dimensional
  $\textmd{RCD}({K},{N})$ spaces.
\newblock {\em Geom. Topol.}, 25(1):445--495, 2021.

\bibitem[KW11]{KW}
V.~Kapovitch and B.~Wilking.
\newblock Structure of fundamental groups of manifolds with {R}icci curvature
  bounded below.
\newblock {\em arXiv:1105.5955v2}, 2011.

\bibitem[Li86]{L86}
P.~Li.
\newblock Large time behavior of the heat equation on complete manifolds with
  non-negative {R}icci curvature.
\newblock {\em Ann. of Math. (2)}, 124(1):1--21, 1986.

\bibitem[LN20]{LN20}
N.~Li and A.~Naber.
\newblock Quantitative estimates on the singular sets of {A}lexandrov spaces.
\newblock {\em Peking Mathematical Journal}, 3:203--234, 2020.

\bibitem[LT89]{LT89}
P.~Li and L.-F. Tam.
\newblock Linear growth harmonic functions on a complete manifold.
\newblock {\em J. Differential Geom.}, 29:421--425, 1989.

\bibitem[MN19]{MN19}
A.~Mondino and A.~Naber.
\newblock Structure theory of metric-measure spaces with lower {R}icci
  curvature bounds.
\newblock {\em J. Eur. Math. Soc. (JEMS)}, 21(6):1809--1854, 2019.

\bibitem[NZ16]{NZ}
A.~Naber and R.~Zhang.
\newblock Topology and $\epsilon$-regularity theorems on collapsed manifolds
  with {R}icci curvature bounds.
\newblock {\em Geom. Topol.}, 20(5):2575--2664, 2016.

\bibitem[Per02]{Per02}
G.~Perelman.
\newblock The entropy formula for the {R}icci flow and its geometric
  applications.
\newblock {\em arXiv: math/0211159v1}, 2002.

\bibitem[Pet11]{Pet11}
A.~Petrunin.
\newblock Alexandrov meets {L}ott--{V}illani--{S}turm.
\newblock {\em M\"unster J. Math.}, 4:53--64, 2011.

\bibitem[Ron18]{Ro18}
X.~Rong.
\newblock Manifolds of {R}icci curvature and local rewinding volume bounded
  below (in {C}hinese).
\newblock {\em Sci Sin Math}, 48(6):791--806, 2018.

\bibitem[Ron22]{Ro22}
X.~Rong.
\newblock Collapsed manifolds with local {R}icci bounded covering geometry.
\newblock {\em arXiv:2211.09998v1}, 2022.

\bibitem[Vil09]{Vi09}
C.~Villani.
\newblock {\em Optimal transport, old and new}, volume 338, pages xxii+973.
\newblock Springer Verlag, 2009.

\bibitem[Wan20]{Wang20}
B.~Wang.
\newblock The local entropy along {R}icci flow, part {B}: the pseudo-locality
  theorems.
\newblock {\em https://arxiv.org/abs/2010.09981}, 2020.

\bibitem[Wei97]{Wei97}
G.~Wei.
\newblock Ricci curvature and {B}etti numbers.
\newblock {\em J. Geom. Anal.}, 7(3):493--509, 1997.

\bibitem[WZ23]{WZh21}
B.~Wang and X.~Zhao.
\newblock Canonical diffeomorphisms of manifolds near spheres.
\newblock {\em J. Geom. Anal.}, 33:Paper No. 304, 31 pp, 2023.

\bibitem[Xu18]{Xu18}
S.~Xu.
\newblock Local estimate on convexity radius and decay of injectivity radius in
  a {R}iemannian manifold.
\newblock {\em Commun. Contemp. Math.}, 20:19 pp, 2018.

\bibitem[Yam91]{Y91}
T.~Yamaguchi.
\newblock Collapsing and pinching under a lower curvature bound.
\newblock {\em Ann. of Math. (2)}, 133(2):317--357, 1991.

\bibitem[ZZ10]{ZZ10}
H.-C. Zhang and X.-P. Zhu.
\newblock Ricci curvature on {A}lexandrov spaces and rigidity theorems.
\newblock {\em Comm. Anal. Geom.}, 18(3):503--553, 2010.

\bibitem[ZZ19]{ZZ17}
H.-C. Zhang and X.-P. Zhu.
\newblock Weyl's law on $\textmd{RCD}^*({K},{N})$ metric measure spaces.
\newblock {\em Comm. Anal. Geom.}, 9:1869--1914, 2019.

\end{thebibliography}

\end{document}